\pdfoutput=1
\documentclass[11pt, letterpaper]{article}
\usepackage{blindtext}
\usepackage{hyperref}
\hypersetup{
	colorlinks=true,
	linkcolor=blue!70!black,
	citecolor=blue!70!black,
	urlcolor=blue!70!black
}

\usepackage[english]{babel}
\usepackage[T1]{fontenc}

\usepackage[margin=1in]{geometry}

\usepackage{amsmath}
\usepackage{amssymb}
\usepackage{amsthm}
\usepackage{booktabs}
\usepackage{algorithm}
\usepackage[lined,boxed,ruled,norelsize,algo2e,linesnumbered]{algorithm2e}

\usepackage{graphicx}
\graphicspath{{./figs/}}

\usepackage{cleveref}
\newtheorem{theorem}{Theorem}
\newtheorem{lemma}{Lemma}
\newtheorem{fact}{Fact}
\newtheorem{definition}{Definition}
\newtheorem{corollary}{Corollary}
\newtheorem{proposition}{Proposition}

\newtheorem{problem}{Problem}

\SetKwComment{Comment}{$\triangleright$\ }{}
\SetCommentSty{mycommfont}

\newcommand{\defeq}{:=}
\newcommand{\norm}[1]{\left\lVert#1\right\rVert}

\newcommand{\inprod}[2]{\left\langle#1, #2\right\rangle}

\newcommand{\eps}{\epsilon}
\newcommand{\lam}{\lambda}

\newcommand{\argmin}{\textup{argmin}} 
\newcommand{\R}{\mathbb{R}}

\newcommand{\N}{\mathbb{N}}

\newcommand{\gK}{\mathcal{K}}

\newcommand{\half}{\frac{1}{2}}

\newcommand{\E}{\mathbb{E}}

\newcommand{\Nor}{\mathcal{N}}

\newcommand{\xset}{\mathcal{X}}

\newcommand{\id}{\mathbf{I}}

\usepackage{xcolor}
\definecolor{burntorange}{rgb}{0.8, 0.33, 0.0}
\newcommand{\kjtian}[1]{\textcolor{burntorange}{\textbf{kjtian:} #1}}
\newcommand{\yujia}[1]{\textcolor{purple}{\textbf{yujia:} #1}}
\newcommand{\daogao}[1]{\textcolor{blue}{\textbf{daogao:} #1}}
\newcommand{\sidford}[1]{{\textcolor{green!70!black}{\textbf{sidford: } #1}}}
\newcommand{\yintat}[1]{\textcolor{pink}{\textbf{yintat:} #1}}
\newcommand{\arun}[1]{\textcolor{red}{\textbf{arun:} #1}}
\newcommand{\yair}[1]{\textcolor{cyan!80!black}{\textbf{yair:} #1}}

\newcommand{\tO}{\widetilde{O}}

\newcommand{\Prox}{\textup{Prox}}
\newcommand{\Par}[1]{\left(#1\right)}
\newcommand{\Brack}[1]{\left[#1\right]}
\newcommand{\Brace}[1]{\left\{#1\right\}}
\newcommand{\Abs}[1]{\left|#1\right|}

\newcommand{\alg}{\mathcal{A}}
\newcommand{\mech}{\mathcal{M}}
\newcommand{\data}{\mathcal{D}}

\newcommand{\ptot}{\eps_{\textup{opt}}}
\newcommand{\hf}{\widehat{f}}

\newcommand{\hx}{\widehat{x}}

\newcommand{\bx}{\bar{x}}

\newcommand{\tnbx}{\widetilde{\nabla}_{\bar{x}}}
\newcommand{\msig}{\boldsymbol{\Sigma}}
\newcommand{\dist}{\mathcal{P}}
\newcommand{\calS}{\mathcal{S}}
\newcommand{\Fpop}{f^{\textup{pop}}}
\newcommand{\Hpop}{H^{\textup{pop}}}
\newcommand{\Herm}{H^{\textup{erm}}}
\newcommand{\Ferm}{f^{\textup{erm}}}

\newcommand{\ball}{\mathbb{B}}
\newcommand{\hT}{\widehat{T}}
\newcommand{\set}{\mathcal{K}}
\newcommand{\proj}{\Pi}
\newcommand{\by}{\bar{y}}
\newcommand{\xsbx}{x^\star_{\bar{x}}}
\newcommand{\xsbxl}{x^\star_{\bar{x}, \lam}}

\newcommand{\ind}{\mathcal{I}}
\newcommand{\event}{\mathcal{E}}
\newcommand{\Dtv}{D_{\textup{TV}}}
\newcommand{\symsum}{\sum_{\textup{sym}}}
\newcommand{\Geom}{\textup{Geom}}
\newcommand{\jmax}{j_{\max}}

\newcommand{\Tmax}{T_{\max}}
\newcommand{\Csc}{C_{\textup{sc}}}
\newcommand{\Cbias}{C_{\textup{bias}}}
\newcommand{\Cvar}{C_{\textup{var}}}
\newcommand{\Ccvx}{C_{\textup{cvx}}}
\newcommand{\Cpriv}{C_{\textup{priv}}}
\newcommand{\Cba}{C_{\textup{ba}}}
\newcommand{\Ols}{\oracle_{\textup{ls}}}
\newcommand{\Obo}{\oracle_{\textup{bo}}}
\newcommand{\Opg}{\oracle_{\textup{sp}}}
\newcommand{\Ohp}{\oracle_{\textup{hp}}}
\newcommand{\lams}{\lam_\star}
\newcommand{\ballacc}{\mathsf{BallAccel}}
\newcommand{\agg}{\mathsf{Aggregate}}
\renewcommand{\d}{\mathrm{d}}
\newcommand{\hnabla}{\widehat{\nabla}}
\newcommand{\Cls}{C_{\textup{ls}}}

\newcommand{\xopt}{x^\star}

\newcommand{\epsdp}{\eps_{\textup{dp}}}
\newcommand{\epsopt}{\eps_{\textup{opt}}}

\newcommand{\oracle}{\mathcal{O}}
\newcommand{\oracleGrad}{\mathcal{O}}

\newcommand{\grad}{\nabla}

\newcommand{\epochSGD}{\mathsf{EpochSGD}}
\newcommand{\NAGD}{\mathsf{AC}\textup{-}\mathsf{SA}}
\newcommand{\ag}{\mathsf{ag}}
\newcommand{\md}{\mathsf{md}}

\newcommand{\ERM}{\mathrm{erm}}
\newcommand{\SCO}{\mathrm{sco}}
 
\title{ReSQueing Parallel and Private Stochastic Convex Optimization}

\author{Yair Carmon\thanks{Tel Aviv University, \texttt{ycarmon@tauex.tau.ac.il}.}\\
	\and 
	Arun Jambulapati\thanks{University of Washington, \texttt{\{jmblpati, dgliu\}@uw.edu}.}\\
	\and 
	Yujia Jin\thanks{Stanford University, \texttt{\{yujiajin, sidford\}@stanford.edu}.}\\
	\and
	Yin Tat Lee\thanks{Microsoft Research, \texttt{\{yintatlee, tiankevin\}@microsoft.com}.}\\
	\and
	Daogao Liu\footnotemark[2]\\
	\and
	Aaron Sidford\footnotemark[3]\\
	\and
	Kevin Tian\footnotemark[4]\\
}

\date{}
\begin{document}

\maketitle

\begin{abstract}
We introduce a new tool for stochastic convex optimization (SCO): a Reweighted Stochastic Query (ReSQue) estimator for the gradient of a function convolved with a (Gaussian) probability density. Combining ReSQue with recent advances in \emph{ball oracle acceleration} \cite{CarmonJJJLST20, AsiCJJS21}, we develop algorithms achieving state-of-the-art complexities for SCO in parallel and private settings. For a SCO objective constrained to the unit ball in $\R^d$, we obtain the following results (up to polylogarithmic factors).
\begin{enumerate}
    \item We give a parallel algorithm obtaining optimization error $\epsopt$ with $d^{1/3}\epsopt^{-2/3}$ gradient oracle query depth and $d^{1/3}\epsopt^{-2/3} + \epsopt^{-2}$ gradient queries in total, assuming access to a bounded-variance stochastic gradient estimator. For $\epsopt \in [d^{-1}, d^{-1/4}]$, our algorithm matches the state-of-the-art oracle depth of \cite{BubeckJLLS19} while maintaining the optimal total work of stochastic gradient descent.
    \item Given $n$ samples of Lipschitz loss functions, prior works \cite{bftt19, BFGT20, AFKT21, KLL21} established that if $n \gtrsim d \epsdp^{-2}$, $(\epsdp, \delta)$-differential privacy is attained at no asymptotic cost to the SCO utility. However, these prior works all required a superlinear number of gradient queries. We close this gap for sufficiently large $n \gtrsim d^2 \epsdp^{-3}$, by using ReSQue to design an algorithm with near-linear gradient query complexity in this regime.
\end{enumerate}
\end{abstract}

\newpage
\pagenumbering{gobble}
	\setcounter{tocdepth}{2}
	{
		\tableofcontents
	}
	\newpage
	\pagenumbering{arabic}

\section{Introduction}
\label{sec:intro}

Stochastic convex optimization (SCO) is a foundational problem in optimization theory, machine learning, theoretical computer science, and modern data science. Variants of the problem underpin a wide variety of applications in machine learning, statistical inference, operations research, signal processing, and control and systems engineering \cite{Shalev-Shwartz07, Shalev-ShwartzBD14}. Moreover, SCO provides a fertile ground for the design and analysis of scalable optimization algorithms such as the celebrated stochastic gradient descent (SGD), which is ubiquitous in machine learning practice \cite{Bottou12}.

SGD approximately minimizes a function $f:\R^d \to \R$ by iterating $x_{t+1} \gets x_t - \eta g(x_t)$, where $g(x_t)$ is an unbiased estimator to a (sub)gradient of $f$ at iterate $x_t$. When $f$ is convex,  $\E\norm{g(x)}^2\le 1$ for all $x$ and $f$ is minimized at $x^\star$ in the unit ball, SGD finds an $\epsopt$-optimal point (i.e.\ $x$ satisfying $\E f(x) \le f(x^\star) + \epsopt$) using $O(\epsopt^{-2})$ stochastic gradient evaluations~\cite{Bubeck15}. This complexity is unimprovable without further assumptions~\cite{D18}; for sufficiently large $d$, this complexity is optimal even if $g$ is an exact subgradient of $f$~\cite{DG19}.

Although SGD is widely-used and theoretically optimal in this simple setting, the algorithm in its basic form has natural limitations. For example, when parallel computational resources are given (i.e., multiple stochastic gradients can be queried in batch), SGD has suboptimal sequential depth in certain regimes \cite{DBM12,BubeckJLLS19}.
Furthermore, standard SGD is not differentially private, and existing private\footnote{Throughout this paper, when we use the description ``private'' without further description we always refer to differential privacy \cite{DR14}. For formal definitions of differential privacy, see Section~\ref{ssec:privacy}.} SCO algorithms are not as efficient as SGD in terms of  gradient evaluation complexity \cite{BassilyST14, bftt19, FKT20, BFGT20, AFKT21, KLL21}.   Despite substantial advances in both the parallel and private settings, the optimal complexity of each SCO problem remains open (see Sections~\ref{sec:intro:parallel} and~\ref{sec:intro:dp} for more precise definitions of problem settings and the state-of-the-art rates, and Section~\ref{sec:intro:related_work} for a broader discussion of related work).

Though seemingly disparate at first glance, in spirit parallelism and privacy impose similar constraints on effective algorithms.  Parallel algorithms must find a way to query the oracle multiple times (possibly at multiple points) without using the oracle's output at these points to determine where they were queried. In other words, they cannot be too reliant on a particular outcome to adaptively choose the next query. Likewise, private algorithms must make optimization progress without over-relying on any individual sample to determine the optimization trajectory. In both cases, oracle queries must be suitably robust to preceding oracle outputs. 

In this paper, we provide a new stochastic gradient estimation tool which we call \emph{Reweighted Stochastic Query (ReSQue) estimators} (defined more precisely in Section~\ref{sec:intro:approach}). ReSQue is essentially an efficient parallel method for computing an unbiased estimate of the gradient of a convolution of $f$ with a continuous (e.g.\ Gaussian) kernel. These estimators are particularly well-suited for optimizing a convolved function over small Euclidean balls, as they enjoy improved stability properties over these regions. In particular, these local stability properties facilitate tighter control over the stability of SGD-like procedures. We show that careful applications of ReSQue in conjunction with recent advances in accelerated ball-constrained optimization \cite{CarmonJJJLST20,AsiCJJS21} yield complexity improvements for both parallel and private SCO. 
\paragraph{Paper organization.} In Sections~\ref{sec:intro:parallel} and~\ref{sec:intro:dp} respectively, we formally describe the problems of parallel and private SCO we study, stating our results and contextualizing them in the prior literature. We then cover additional related work in \Cref{sec:intro:related_work} and, in \Cref{sec:intro:approach}, give an overview of our approach to obtaining these results. In Section~\ref{sec:notation}, we describe the notation we use throughout.

In Section~\ref{ssec:estimator} we introduce our ReSQue estimator and prove some of its fundamental properties. In Section~\ref{sec:ball-constrained} we describe our adaptation of the ball acceleration frameworks of \cite{AsiCJJS21, CarmonH22}, reducing SCO to minimizing the objective over small Euclidean balls, subproblems which are suitable for ReSQue-based stochastic gradient methods. Finally, in Sections~\ref{sec:parallel} and~\ref{sec:subproblem}, we prove our main results for parallel and private SCO (deferring problem statements to Problem~\ref{prob:sco_gen} and Problem~\ref{prob:sco_basic}), respectively, by providing suitable implementations of our ReSQue ball acceleration framework.

\subsection{Parallelism}
\label{sec:intro:parallel}

In Section~\ref{sec:parallel} we consider the following formulation of the SCO problem, simplified for the purposes of the introduction.  We assume there is a convex function $f: \R^d \to \R$ whose minimizer lies in the Euclidean ball with unit radius, which can be queried through a \emph{stochastic gradient oracle} $g$, satisfying $\E g \in \partial f$ and $\E \norm{g}^2 \le 1$. We wish to minimize $f$ to expected additive error $\epsopt$. In the standard sequential setting, SGD achieves this goal using roughly $\epsopt^{-2}$ queries to $g$; as  previously mentioned, this complexity is optimal. A generalization of this formulation is restated in Problem~\ref{prob:sco_gen} with a variance bound $L^2$ and a radius bound $R$, which are both set to $1$ here.

In settings where multiple machines can be queried simultaneously, the parallel complexity of an SCO algorithm is a further important measure for consideration. In \cite{Nem94}, this problem was formalized in the setting of oracle-based convex optimization, where the goal is to develop iterative methods with a  number of parallel query batches to $g$. In each batch, the algorithm can submit polynomially many queries to $g$ in parallel, and then perform computations (which do not use $g$) on the results. The \emph{query depth} of a parallel algorithm in the \cite{Nem94} model is the number of parallel rounds used to query $g$, and was later considered in stochastic algorithms~\cite{DBM12}. Ideally, a parallel SCO algorithm will also have bounded \emph{total queries} (the number of overall queries to $g$), and bounded \emph{computational depth}, i.e., the parallel depth used by the algorithm treating the depth of each oracle query as $O(1)$. We discuss these three complexity measures more formally in Section~\ref{ssec:parallel_prelims}.

In the low-accuracy regime $\epsopt \ge d^{-1/4}$, recent work \cite{BubeckJLLS19} showed that SGD indeed achieves the optimal oracle query depth among parallel algorithms.\footnote{We omit logarithmic factors when discussing parameter regimes throughout the introduction.} Moreover, in the high-accuracy regime $\epsopt \le d^{-1}$, cutting plane methods (CPMs) by e.g.\ \cite{KTE88} (see \cite{JiangLSW20} for an updated overview) achieve the state-of-the-art oracle query depth of $d$, up to logarithmic factors in $d, \epsopt$. %

In the intermediate regime $\epsopt \in [d^{-1}, d^{-1/4}]$, \cite{DBM12, BubeckJLLS19} designed algorithms with oracle query depths that improved upon SGD, as summarized in Table~\ref{tab:parallel}. In particular, \cite{BubeckJLLS19} obtained an algorithm with query depth $\widetilde{O} (d^{1/3} \epsopt^{-2/3})$, which they conjectured is optimal for intermediate $\epsopt$. However, the total oracle query complexity of \cite{BubeckJLLS19} is $\widetilde{O} ( d^{4/3}\epsopt^{-8/3})$, a (fairly large) polynomial factor worse than SGD.

\paragraph{Our results.} The main result of Section~\ref{sec:parallel} is a pair of improved parallel algorithms in the setting of Problem~\ref{prob:sco_gen}. Both of our algorithms achieve the ``best of both worlds'' between the \cite{BubeckJLLS19} parallel algorithm and SGD, in that their oracle query depth is bounded by  $\widetilde{O}(d^{1/3}\epsopt^{-2/3})$ (as in \cite{BubeckJLLS19}), but their total query complexity matches  SGD's in the regime $\epsopt \le d^{-1/4}$. We note that $\epsopt \le d^{-1/4}$ is the regime where a depth of $\widetilde{O} (d^{1/3}\epsopt^{-2/3})$ improves upon \cite{DBM12} and SGD. Our guarantees are formally stated in Theorems~\ref{thm:parallel-sgd} and~\ref{thm:parallel-agd}, and summarized in Table~\ref{tab:parallel}.
\begin{table}
    \centering
    \renewcommand{\arraystretch}{1.75}
    \begin{tabular}{{c}{c}{c}{c}}
    \toprule
      Method   &  $g$ query depth & computational depth & \# $g$ queries \\
      \midrule
       SGD~\cite{Nesterov18}  & $\eps^{-2}$ & $\eps^{-2}$ & $\eps^{-2}$\\
       \cite{DBM12} & $d^{\frac{1}{4}}\eps^{-1}$ & $d^{\frac{1}{4}}\eps^{-1}$ & $d^{\frac{1}{4}}\eps^{-1}+\eps^{-2}$\\
       \cite{BubeckJLLS19} & $d^{\frac{1}{3}}\eps^{-\frac{2}{3}}$ & $d^{\frac{4}{3}}\eps^{-\frac{8}{3}}$ & $d^{\frac{4}{3}}\eps^{-\frac{8}{3}}$ \\
       CPM~\cite{KTE88} & $d$ & $d$ & $d$ \\
       \midrule
       $\ballacc+\epochSGD$ (\Cref{thm:parallel-sgd}) & $d^{\frac{1}{3}}\eps^{-\frac{2}{3}}$ & $d^{\frac{1}{3}}\eps^{-\frac{2}{3}}+\eps^{-2}$ & $d^{\frac{1}{3}}\eps^{-\frac{2}{3}}+\eps^{-2}$ \\
        $\ballacc+\NAGD$ (\Cref{thm:parallel-agd}) & $d^{\frac{1}{3}}\eps^{-\frac{2}{3}}$ & $d^{\frac{1}{3}}\eps^{-\frac{2}{3}}+d^{\frac{1}{4}}\eps^{-1}$ & $d^{\frac 1 3}\eps^{-\frac 2 3}+\eps^{-2}$ \\
       \bottomrule
    \end{tabular}
    \caption{\textbf{Comparison of parallel SCO results.} The complexity of finding a point with expected error $\eps \defeq \epsopt$ in Problem~\ref{prob:sco_gen}, where $L = R = 1$. We hide polylogarithmic factors in $d$ and $\eps^{-1}$. }
    \label{tab:parallel}
\end{table}

Our first algorithm (Theorem~\ref{thm:parallel-sgd}) is based on a batched SGD using our ReSQue estimators, within the ``ball acceleration'' framework of \cite{AsiCJJS21} (see Section~\ref{sec:intro:approach}). By replacing SGD with an accelerated counterpart \cite{GhadimiL12}, we obtain a further improved \emph{computational depth} in Theorem~\ref{thm:parallel-agd}. %
Theorem~\ref{thm:parallel-agd} simultaneously achieves the query depth of \cite{BubeckJLLS19}, the computational depth of \cite{DBM12}, and the total query complexity of SGD in the intermediate regime $\epsopt \in [d^{-1}, d^{-1/4}]$.

\subsection{Differential privacy}
\label{sec:intro:dp}

Differential privacy (DP) is a mathematical quantification for privacy risks in algorithms involving data. When performing stochastic convex optimization with respect to a sampled dataset from a population, privacy is frequently a natural practical desideratum \cite{BassilyST14, EPK14, Abo16, Apple17}. For example, the practitioner may want to privately learn a linear classifier or estimate a regression model or a statistical parameter from measurements.

In this paper, we obtain improved rates for private SCO in the following model, which is standard in the literature and restated in Problem~\ref{prob:sco_basic} in full generality. Symmetrically to the previous section, in the introduction, we only discuss the specialization of Problem~\ref{prob:sco_basic} with $L = R = 1$, where $L$ is a Lipschitz parameter and $R$ is a domain size bound. We assume there is a distribution $\dist$ over a population $\calS$, and we obtain independent samples $\{s_i\}_{i \in [n]} \sim \dist$. Every element $s \in \calS$ induces a $1$-Lipschitz convex function $f(\cdot; s)$, and the goal of SCO is to approximately optimize the population loss $\Fpop \defeq \E_{s \sim \dist}[f(\cdot; s)]$. The setting of Problem~\ref{prob:sco_basic} can be viewed as a specialization of Problem~\ref{prob:sco_gen} which is more compatible with the notion of DP, discussed in more detail in Section~\ref{ssec:privacy}.

The cost of achieving approximate DP with privacy loss parameter $\epsdp$ (see Section~\ref{ssec:privacy} for definitions) has been studied by a long line of work, starting with \cite{BassilyST14}. The optimal error (i.e., excess population loss) given $n$ samples scales as (omitting logarithmic factors)
\begin{equation}\label{eq:opt_err_dp} \frac 1 {\sqrt n} + \frac{\sqrt d}{n\epsdp}, \end{equation}
with matching lower and upper bounds given by \cite{BassilyST14} and \cite{bftt19}, respectively. The $n^{-1/2}$ term is achieved (without privacy considerations) by simple one-pass SGD, i.e., treating sample gradients as unbiased for the population loss, and discarding samples after we query their gradients. Hence, the term $\sqrt d \cdot (n\epsdp)^{-1}$ can be viewed as the ``cost of privacy'' in SCO. Assuming that we have access to $n \ge d\epsdp^{-2}$ samples is then natural, as this is the setting where privacy comes at no asymptotic cost from the perspective of the bound \eqref{eq:opt_err_dp}. Moreover, many real-world problems in data analysis have low intrinsic dimension, meaning that the effective number of degrees of freedom in the optimization problem is much smaller than the ambient dimension \cite{sstt21, li2022does}, which can be captured via a dimension-reducing preprocessing step. For these reasons, we primarily focus on the regime when the number of samples $n$ is sufficiently large compared to $d$.

An unfortunate property of private SCO algorithms achieving error \eqref{eq:opt_err_dp} is they all query substantially more than $n$ sample gradients without additional smoothness assumptions \cite{BassilyST14, bftt19, FKT20, BFGT20, AFKT21, KLL21}, which can be viewed as a statistical-computational gap. For example, analyses of simple perturbed SGD variants result in query bounds of $\approx n^2$ \cite{BFGT20}. In fact, \cite{BFGT20} conjectured this quadratic complexity was necessary, which was disproven by \cite{AFKT21, KLL21}.
The problem of obtaining the optimal error \eqref{eq:opt_err_dp} using $n$ gradient queries has been repeatedly highlighted as an important open problem by the private optimization community, as discussed in \cite{BFGT20, AFKT21, KLL21, AsiCJJS21} as well as the recent research overview \cite{Talwar22}. 

Qualitatively, optimality of the bound \eqref{eq:opt_err_dp} shows that there is no statistical cost of privacy when the number of samples $n$ is large enough, as the solver relies less on any specific sample. A natural first step towards developing optimal private SCO algorithms is to ask a similar qualitative question regarding their computational guarantees. Concretely, given enough samples $n$, can we develop statistically-optimal SCO algorithms which only query $\approx n$ sample gradients?

\paragraph{Our results.} In Section~\ref{sec:subproblem}, we develop the first private SCO algorithm with this aforementioned computational guarantee. Our algorithm achieves the error bound \eqref{eq:opt_err_dp} up to logarithmic factors, as well as a new gradient query complexity. Our result is formally stated in Theorem~\ref{thm:DP-SCO} and summarized in Table~\ref{tab:general_SCO} and Figure~\ref{fig:compare_result}. Up to logarithmic factors, our gradient query complexity is
\[\min\Par{n, \frac{n^2 \epsdp^2}{d}} + \min\Par{\frac{(nd)^{\frac 2 3}}{\epsdp}, n^{\frac 4 3}\epsdp^{\frac 1 3}}. \]
Theorem~\ref{thm:DP-SCO} improves upon the prior state-of-the-art gradient query complexity by polynomial factors whenever $d \ll n^{4/3}$ (omitting $\epsdp$ dependencies for simplicity). As with prior recent SCO advancements, our result has the appealing property that it achieves the optimal $n^{-1/2}$ error for SCO when $n \gtrsim d\epsdp^{-2}$. Moreover, given $n \gtrsim d^2\epsdp^{-3}$ samples, the gradient query complexity of Theorem~\ref{thm:DP-SCO} improves to $\tO(n)$, the first near-linear query complexity for a statistically-optimal private SCO algorithm in any regime. In~\Cref{tab:general_SCO} and \Cref{fig:compare_result}, we compare our bounds with the prior art.

While there remains a gap between the sample complexity at which our algorithm is statistically optimal, and that at which it is computationally (nearly)-optimal, we find it promising that our result comes within logarithmic factors of achieving the best-of-both-worlds for sufficiently large $n$. This is a key step towards optimal algorithms for the fundamental problem of private SCO. It is an interesting open question to refine current algorithmic techniques for private SCO to remove this gap, and we are optimistic that the tools developed in this paper will be fruitful in this endeavor.

\begin{table}
    \centering
    \renewcommand{\arraystretch}{1.75}
    \begin{tabular}{{c}{c}{c}{c}}
    \toprule
Method & excess $\Fpop$ loss & \# gradient queries to samples \\
\midrule
\cite{BassilyST14} & $\frac{\sqrt[4]{d}\log\frac n \delta}{\sqrt{n}}+\frac{\sqrt d\log^2\frac n \delta}{n\epsilon}$ & $n^2$ \\ 

\cite{bftt19} &$ \frac{1}{\sqrt{n}}+\frac{\sqrt{d\log\frac 1 \delta}}{n\epsilon}$& $n^{\frac 9 2}$ \\ 

\cite{FKT20} & $  \frac{1}{\sqrt{n}}+\frac{\sqrt{d\log\frac 1 \delta}}{n\epsilon}$ & $n^2$ \\

\cite{BFGT20} & $ \frac{1}{\sqrt{n}}+\frac{\sqrt{d\log\frac 1 \delta}}{n\epsilon}$ &  $n^{2}$\\

\cite{AFKT21} & $  \frac{1}{\sqrt{n}}+\frac{\sqrt{d\log\frac 1 \delta}}{n\epsilon}$ & $\min\Par{n^{\frac 3 2}, \frac{n^2 \eps}{\sqrt d}}$ \\

\cite{KLL21} & $ \frac{1}{\sqrt{n}}+\frac{\sqrt{d\log\frac 1 \delta}}{n\epsilon}$ & $\min\Par{n^{\frac 5 4}d^{\frac 1 8}\sqrt \eps, \frac{n^{\frac 3 2} \eps}{d^{\frac 1 8}} }$ \\
\midrule
Theorem~\ref{thm:DP-SCO} &   $\frac{1}{\sqrt{n}}+\frac{\sqrt{d\log\frac 1 \delta}\log n\log^{1.5}\frac n \delta}{n\epsilon}$ & $\min\Par{n, \frac{n^2\eps^2}{d}} + \min\Par{\frac{(nd)^{\frac 2 3}}{\eps}, n^{\frac 4 3} \eps^{\frac 1 3}}$ \\
\bottomrule
\end{tabular}
    \caption{\textbf{Comparison of private SCO results.} The excess loss and gradient complexity of $(\eps \defeq \epsdp,\delta)$-DP in Problem~\ref{prob:sco_basic}, where $L = R = 1$. We hide polylogarithmic factors in $d, n, \delta^{-1}, \eps^{-1}$ in the third column. The optimal loss  \cite{BassilyST14, SU15} is achieved by rows 2-6.
    }
    \label{tab:general_SCO}
\end{table}

\begin{figure}[h]
    \centering
    \includegraphics[width = .6 \textwidth]{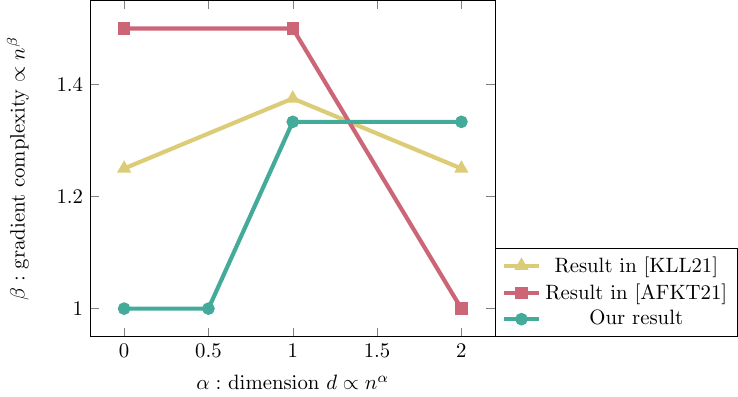}
    \caption{Comparison among our gradient complexity and previous results in \cite{AFKT21,KLL21} for the non-trivial regime $d\leq n^2$. We omit dependencies on $\epsdp$ (treated as $\Theta(1)$ in this figure) and logarithmic terms for simplicity.}
    \label{fig:compare_result}
\end{figure}

\subsection{Related work}
\label{sec:intro:related_work}

\paragraph{Stochastic convex optimization.} 
Convex optimization is a fundamental task with numerous applications in computer science, operations research, and statistics \cite{BoydV14, Bubeck15, Nesterov18}, and has been the focus of extensive research over the past several decades. This paper's primary setting of interest is non-smooth (Lipschitz) stochastic convex optimization in private and parallel computational models. Previously, \cite{G64} gave a gradient method that used $O(\eps^{-2})$ gradient queries to compute a point achieving $\eps$ error for Lipschitz convex minimization. This rate was shown to be optimal in an information-theoretic sense in \cite{NY83}. The stochastic gradient descent method extends \cite{G64} to tolerate randomized, unbiased gradient oracles with bounded second moment: this yields algorithms for Problem~\ref{prob:sco_gen} and Problem~\ref{prob:sco_basic} (when privacy is not a consideration).

\paragraph{Acceleration.} Since the first proposal of accelerated (momentum-based) methods~\cite{Polyak64, Nes83, nes03}, acceleration has become a central topic in optimization. This work builds on the seminal Monteiro-Svaiter acceleration technique~\cite{Monteiro13} and its higher-order variants~\cite{Gasnikov19, BubeckJLLS19}. More specifically, our work follows recent developments in accelerated ball optimization \cite{CarmonJJJLST20, CarmonJJS21, AsiCJJS21}, which can be viewed as a limiting case of high-order methods. Our algorithms directly leverage error-robust variants of this framework developed by \cite{AsiCJJS21, CarmonH22}.

\paragraph{Parallel SCO.} Recently, parallel optimization has received increasing interest in the context of large-scale machine learning. Speeding up SGD by averaging stochastic gradients across mini-batches is extremely common in practice, and optimal in certain distributed optimization settings; see e.g.~\cite{Dekel12, Duchi18, WBSS21}. Related to the setting we study are the distributed optimization methods proposed in~\cite{Scaman18}, which also leverage convolution-based randomized smoothing and apply to both stochastic and deterministic gradient-based methods (but do not focus on parallel depth in the sense of \cite{Nem94}). Finally, lower bounds against the oracle query depth of parallel SCO algorithms in the setting we consider have been an active area of study, e.g.\ \cite{Nem94, BalkanskiS18, DG19, BubeckJLLS19}.

\paragraph{Private SCO.} Both the private stochastic convex optimization problem (DP-SCO) and the private empirical risk minimization problem (DP-ERM) are well-studied by the DP community \cite{CM08,rbht09,cms11,jt14,BassilyST14,kj16,fts17,zzmw17,Wang18,ins+19,bftt19,FKT20}. In particular, \cite{BassilyST14} shows that the exponential mechanism and noisy stochastic gradient descent achieve the optimal loss for DP-ERM for $(\epsdp,0)$-DP and $(\epsdp,\delta)$-DP. In follow-up works, \cite{bftt19,FKT20} show that one can achieve the optimal loss for DP-SCO as well, by a suitable modification of noisy stochastic gradient descent. However, these algorithms suffer from large (at least quadratic in $n$) gradient complexities. Under an additional assumption that the loss functions are sufficiently smooth (i.e., have Lipschitz gradient), \cite{FKT20} remedies this issue by obtaining optimal loss and optimal \emph{gradient complexity} under differential privacy. In a different modification of Problem~\ref{prob:sco_basic}'s setting (where sample function access is modeled through value oracle queries instead of subgradients), \cite{GLL22} designs an exponential mechanism-based method that uses the optimal value oracle complexity to obtain the optimal SCO loss for non-smooth functions.

Most directly related to our approach are the recent works~\cite{KLL21} and~\cite{AsiCJJS21}. Both propose methods improving upon the quadratic gradient complexity achieved by noisy SGD, by using variants of smoothing via Gaussian convolution. The former proposes an algorithm that uses noisy accelerated gradient descent for private SCO with subquadratic gradient complexity. The latter suggests a ball acceleration framework to solve private SCO with linear gradient queries, under a hypothetical algorithm to estimate subproblem solutions. Our work can be viewed as a formalization of the connection between ball acceleration strategies and private SCO as suggested in \cite{AsiCJJS21}, by way of ReSQue estimators, which we use to obtain improved query complexities.

\subsection{Our approach}
\label{sec:intro:approach}

Here we give an overview of our approach towards obtaining the results outlined in Section~\ref{sec:intro:parallel} and Section~\ref{sec:intro:dp}. To illustrate and situate our approach, we first briefly discuss prior approaches, their insights that we leverage, and obstacles that we overcome. Then we discuss a common framework based on a new stochastic gradient estimation tool we introduce and call \emph{Reweighted Stochastic Query (ReSQue) estimators} which enables our results on parallel and private SCO. Our new tool is naturally compatible with ball-constrained optimization frameworks, where an optimization problem is localized to a sequence of constrained subproblems (solved to sufficient accuracy), whose solutions are then stitched together. We exploit this synergy, as well as the local stability properties of our ReSQue estimators, to design our SCO algorithms. We discuss the different instantiations of our framework for parallel and private SCO at the end of this section.

\paragraph{Convolutions and prior approaches.} All new results on parallel and private SCO in this paper use the convolution of a function of interest $f: \R^d \to \R$ with a Gaussian density $\gamma_\rho$ (with covariance $\rho^2 \id_d$), which we denote by $\hf_\rho$. Such \emph{Gaussian convolutions} have a longer history of facilitating algorithmic advances for SCO. All previous advances on parallel SCO and Lipschitz convex function minimization used Gaussian convolutions, i.e., \cite{DBM12,BubeckJLLS19}, as did a state-of-the-art (in some regimes) private SCO algorithm \cite{KLL21}. Each of \cite{DBM12, KLL21} leverage that $\hf_\rho$ is a smooth, additive approximation to $f$, and \cite{BubeckJLLS19} further used that the higher derivatives of $\hf_\rho$ are bounded, as well as the fact that its gradients can be well-approximated within small balls.

As one of our motivating problems, we seek to move beyond the reliance on (high-order) smoothness properties of $\hf_\rho$, and achieve total work bounds improving upon \cite{BubeckJLLS19}. Unfortunately, doing so while following the strategy of \cite{BubeckJLLS19} poses an immediate challenge. Though \cite{BubeckJLLS19} achieves improved parallel depth bounds for Lipschitz convex optimization, it comes at a cost. Their approach, which relies on the $p^{\textup{th}}$-order Lipschitzness of $\hf_\rho$, would naively involve computing $p^{\textup{th}}$ derivatives of the objective, and their approach to gradient approximation involves estimating the gradient everywhere inside a ball of sufficient radius. Na\"ively, either of these approaches would involve making $\Omega(d)$ queries per parallel step. Removing this cost is one of our main contributions to parallel SCO, and our corresponding development is key to enabling our private SCO results. 

\paragraph{ReSQue estimators and ball acceleration.} 
To overcome this bottleneck to prior approaches, we introduce a new tool that capitalizes upon a different property of Gaussian convolutions: the fact that the Gaussian density is locally stable in a small ball around its center. This property is arguably closely related to how \cite{BubeckJLLS19} are able to prove that they can approximate the gradients of $\hf_\rho$ inside a ball. However, rather than building such a complete model of $\hf_\rho$, we instead use only use this property to suitably implement independent stochastic gradient queries to $\hf_\rho$.

Given a reference point $\bx$ and a query point $x$, our proposed estimator for $\grad \hf_\rho(x)$ is
\begin{equation}\label{eq:rgcg_def}\text{draw } 
\xi \sim \Nor(0, \rho^2 \id_d),\text{ and output estimate } \frac{\gamma_\rho(x - \bx - \xi)}{\gamma_\rho(\xi)} g(\bx + \xi), 
\end{equation}
where $g(z)$ is an unbiased estimate for a subgradient of $f$, i.e., $\E g(z)\in \partial f(z)$. That is, to estimate the gradient of $\hf_\rho$, we simply reweight (stochastic) gradients of $f$ that were queried at random perturbations of reference point $\bx$. This reweighted stochastic query (ReSQue) estimator is unbiased for $\grad \hf_\rho(x)$, regardless of $\bx$.  However, when $\norm{x - \bx} \ll \rho$, i.e., $x$ is contained in a small ball around $\bx$, the reweighting factor $\frac{\gamma_\rho(x - \bx - \xi)}{\gamma_\rho(\xi)}$ is likely to be close to $1$. As a result, when $g$ is bounded and $x$ is near $\bx$, the estimator \eqref{eq:rgcg_def} enjoys regularity properties such as moment bounds. 
Crucially, the stochastic gradient queries performed by ReSQue (at points of the form $\bx+\xi$) \emph{do not depend} on the point $x$ at which we eventually estimate the gradient.

We develop this theory in Section~\ref{sec:framework}, but mention one additional property here, which can be thought of as a ``relative smoothness'' property.
We show that when $\norm{x - x'}$ is sufficiently smaller than $\rho$, the \emph{difference} of estimators of the form \eqref{eq:rgcg_def} has many bounded moments, where bounds scale as a function of $\norm{x - x'}$. When we couple a sequence of stochastic gradient updates by the randomness used in defining \eqref{eq:rgcg_def}, we can use this property to bound how far sequences drift apart. In particular, initially nearby points are likely to stay close. We exploit this property when analyzing the stability of private stochastic gradient descent algorithms later in the paper. 

To effectively use these local stability properties of \eqref{eq:rgcg_def}, we combine them with an optimization framework called \emph{ball-constrained optimization} \cite{CarmonJJJLST20}. It is motivated by the question: given parameters $0 < r < R$, and an oracle which minimizes $f: \R^d$ in a ball of radius $r$ around an input point, how many oracles must we query to optimize $f$ in a ball of larger radius $R$? It is not hard to show that simply iterating calls to the oracle gives a good solution in roughly $\frac R r$ queries. In recent work, \cite{CarmonJJJLST20} demonstrated that the optimal number of calls scales (up to logarithmic factors) as $(\frac R r)^{2/3}$, and \cite{AsiCJJS21} gave an approximation-tolerant variant of the \cite{CarmonJJJLST20} algorithm. We refer to these algorithms as \emph{ball acceleration}. Roughly, \cite{AsiCJJS21} shows that running stochastic gradient methods on $\approx (\frac R r)^{2/3}$ subproblems constrained to balls of radius $r$ obtains total gradient query complexity comparable to directly running SGD on the global function of domain radius $R$.

Importantly, in many structured cases, we have dramatically more freedom in solving these subproblems, compared to the original optimization problem, since we are only required to optimize over a small radius. One natural form of complexity gain from ball acceleration is when there is a much cheaper gradient estimator, which is only locally defined, compared to a global estimator. This was the original motivation for combining ball acceleration with stochastic gradient methods in \cite{CarmonJJS21}, which exploited local smoothness of the softmax function; the form of our ReSQue estimator \eqref{eq:rgcg_def} is motivated by the \cite{CarmonJJS21} estimator. In this work, we show that using ReSQue with reference point $\bx$ significantly improves the parallel and private complexity of minimizing the convolution $\hf_\rho$ inside a ball of radius $r \approx \rho$ centered at $\bx$.

\paragraph{Parallel subproblem solvers.} A key property of the ReSQue estimator \eqref{eq:rgcg_def} is that its estimate of $\grad \hf_\rho (x)$ is a scalar reweighting of $g(\bx + \xi)$, where $\xi \sim  N(0, \rho^2 \id_d)$ and $\bx$ is a fixed reference point. Hence, in each ball subproblem (assuming $r = \rho$), we can make \emph{all} the stochastic gradient queries in parallel, and use the resulting pool of vectors to perform standard (ball-constrained) stochastic optimization using ReSQue. Thus, we solve each ball subproblem with a single parallel stochastic gradient query, and --- using ball acceleration --- minimize $\hf_\rho$ with query depth of roughly $\rho^{-2/3}$. To ensure that $\hf_\rho$ is a uniform $\epsopt$-approximation of the original $f$, we must set $\rho$ to be roughly $\epsopt/\sqrt{d}$, leading to the claimed $d^{1/3}\epsopt^{-2/3}$ depth bound. Furthermore, the ball acceleration framework guarantees that we require no more than roughly $\rho^{-2/3} + \epsopt^{-2}$ stochastic gradient computations throughout the optimization, yielding the claimed total query bound. However, the computational depth of the algorithm described thus far is roughly $\epsopt^{-2}$, which is no better than SGD. In \Cref{sec:parallel} we combine our approach with the randomized smoothing algorithm of~\cite{DBM12} by using an accelerated mini-batched method \cite{GhadimiL12} for the ball-constrained stochastic optimization, leading to improved computational depth as summarized in Table~\ref{tab:parallel}. Our parallel SCO results use the ReSQue/ball acceleration technique in a simpler manner than our private SCO results described next and in Section~\ref{sec:subproblem}, so we chose to present them first.

\paragraph{Private subproblem solvers.} To motivate our improved private SCO solvers, we make the following connection. First, it is straightforward to show that the convolved function $\hf_\rho$ is $\frac 1 \rho$-smooth whenever the underlying function $f$ is Lipschitz. Further, recently \cite{FKT20} obtained a linear gradient query complexity for SCO, under the stronger assumption that each sample function (see Problem~\ref{prob:sco_basic}) is $\lesssim \sqrt n$-smooth (for $L = R = 1$ in Problem~\ref{prob:sco_basic}). This bound is satisfied by the result of Gaussian convolution with radius $\frac 1 {\sqrt n}$; however, two difficulties arise. First, to preserve the function value approximately up to $\epsopt$, we must take a Gaussian convolution of radius $\rho \approx \frac{\epsopt}{\sqrt d}$. For $\epsopt$ in \eqref{eq:opt_err_dp}, this is much smaller than $\frac{1}{\sqrt n}$ in many regimes. Second, we cannot access the exact gradients of the convolved sampled functions. Hence, it is natural to ask: is there a way to simulate the smoothness of the convolved function, under stochastic query access? 

Taking a step back, the primary way in which \cite{FKT20} used the smoothness assumption was through the fact that gradient steps on a sufficiently smooth function are \emph{contractive}. This observation is formalized as follows: if $x' \gets x - \eta \nabla f(x)$ and $y' \gets y - \eta \nabla f(y)$, when $f$ is $O(\frac 1 \eta)$-smooth, then $\norm{x' - y'} \le \norm{x - y}$. As alluded to earlier, we show that  ReSQue estimators \eqref{eq:rgcg_def} allow us to simulate this contractivity up to polylogarithmic factors. We show that by coupling the randomness $\xi$ in the estimator \eqref{eq:rgcg_def}, the drift growth in two-point sequences updated with \eqref{eq:rgcg_def} is predictable. We give a careful potential-based argument (see Lemma~\ref{lem:group_privacy_convex}) to bound higher moments of our drift after a sequence of updates using ReSQue estimators, when they are used in an SGD subroutine over a ball of radius $\ll \rho$. This allows for the use of ``iterative localization'' strategies introduced by \cite{FKT20}, based on iterate perturbation via the Gaussian mechanism.
 
We have not yet dealt with the fact that while this ``smoothness simulation'' strategy allows us to privately solve \emph{one} constrained ball subproblem, we still need to solve $K \approx (\frac 1 r)^{2/3}$ ball subproblems to optimize our original function, where $r \ll \rho$ is the radius of each subproblem. Here we rely on arguments based on amplification by subsampling, a common strategy in the private SCO literature \cite{ACG+16, BalleBG18}. We set our privacy budget for each ball subproblem to be approximately $(\epsdp, \delta)$ (our final overall budget), before subsampling. We then use solvers by suitably combining the \cite{FKT20} framework and our estimator \eqref{eq:rgcg_def} to solve these ball subproblems using $\approx n \cdot K^{-1/2}$ gradient queries each. Finally, our algorithm obtains the desired
\begin{equation}\label{eq:tradeoff}
\begin{gathered}
\text{query complexity: }\approx \underbrace{\frac n {\sqrt K}}_{\text{gradient queries per subproblem}} \cdot \underbrace{K}_{\text{number of subproblems}} = n\sqrt{K}, \text{ and }\\
\text{privacy: } \approx \underbrace{\epsdp}_{\textup{privacy budget per subproblem}} \cdot \underbrace{\frac 1 {\sqrt K}}_{\textup{subsampling}} \cdot \underbrace{\sqrt K}_{\textup{advanced composition}} = \epsdp.
\end{gathered}
\end{equation}
Here we used the standard technique of advanced composition (see e.g.\ Section 3.5.2, \cite{DR14}) to bound the privacy loss over $K$ consecutive ball subproblems.

Let us briefly derive the resulting complexity bound and explain the bottleneck for improving it further.
First, the ball radius $r$ must be set to $\approx \rho$ (the smoothing parameter) for our ReSQue estimators to be well-behaved. Moreover, we have to set $\rho \approx \frac{\epsopt}{\sqrt d}$, otherwise the effect of the convolution begins to dominate the optimization error. For $\epsopt \approx \frac 1 {\sqrt n} + \sqrt{d}(n\epsdp)^{-1}$ (see \eqref{eq:opt_err_dp}), this results in $\frac 1 r \approx \min(\sqrt{nd}, n\epsdp)$. Next, $K \approx (\frac 1 r)^{2/3}$ is known to be essentially tight for ball acceleration with $R = 1$ \cite{CarmonJJJLST20}. For the subproblem accuracies required by the \cite{AsiCJJS21} ball acceleration framework,\footnote{These subproblem accuracy requirements cannot be lowered in general, because combined they recover the optimal gradient complexities of SGD over the entire problem domain.} known lower bounds on private empirical risk minimization imply that $\approx \frac n {\sqrt K}$ gradients are necessary for each subproblem to preserve a privacy budget of $\epsdp$ \cite{BassilyST14}. 
As subsampling requires the privacy loss before amplification to already be small (see discussion in \cite{Smith09, BalleBG18}), all of these parameter choices are optimized, leading to a gradient complexity of $n\sqrt{K}$. For our lower bound on $\frac 1 r$, this scales as $\approx \min(n^{4/3}, (nd)^{2/3})$ as we derive in Theorem~\ref{thm:DP-SCO}.\footnote{In the low-dimensional regime $d \le n \epsdp^2$, the gradient queries used per subproblem improves to $\frac {\sqrt{nd}} {\epsdp\sqrt K}$.} To go beyond the strategies we employ, it is natural to look towards other privacy amplification arguments (for aggregating ball subproblems) beyond subsampling, which we defer to future work.

Our final algorithm is analyzed through the machinery of R\'enyi differential privacy (RDP) \cite{Mir17}, which allows for more fine-grained control of the effects of composition and subsampling. We modify the standard RDP machinery in two main ways. We define an approximate relaxation and control the failure probability of our relaxation using high moment bounds on our drift (see Section~\ref{ssec:erm_convex}). We also provide an analysis of amplification under subsampling with replacement by modifying the truncated CDP (concentrated DP) tools introduced by \cite{BunDRS18}, who analyzed subsampling without replacement. Sampling with replacement is crucial in order to guarantee that our ReSQue estimators are unbiased for the empirical risks we minimize when employing a known reduction \cite{FKT20, KLL21} from private SCO to private regularized empirical risk minimization. 
\subsection{Notation}\label{sec:notation}

Throughout $\tO$ hides polylogarithmic factors in problem parameters. For $n \in \N$, we let $[n] \defeq \{i \mid 1 \le i \le n\}$. For $x \in \R^d$ we let $\norm{x}$ denote the Euclidean norm of $x$, and let $\ball_{x}(r) \defeq \{x' \in \R^d \mid \norm{x' - x} \le r\}$ denote a Euclidean ball of radius $r$ centered at $x$; when $x$ is unspecified we take it to be the origin, i.e., $\ball(r) \defeq \{x' \in \R^d \mid \norm{x'} \le r\}$. We let $\Nor(\mu, \msig)$ denote a multivariate Gaussian distribution with mean $\mu \in \R^d$ and covariance $\msig \in \R^{d \times d}$, and $\id_d$ is the identity matrix in $\R^{d \times d}$. For $\set \subseteq \R^d$, we define the Euclidean projection onto $\set$ by $\proj_{\set}(x) \defeq \argmin_{x' \in \set} \norm{x - x'}$. For $p \in [0, 1]$, we let $\Geom(p)$ denote the geometric distribution with parameter $p$.

\paragraph{Optimization.} We say a function $f: \R^d \to \R$ is $L$-Lipschitz if for all $x, x' \in \R^d$ we have $|f(x) - f(x')| \le L\norm{x - x'}$. We say $f$ is $\lam$-strongly convex if for all $x, x' \in \R^d$ and $t\in[0,1]$ we have
\[f(tx + (1 - t)y) \le tf(x) + (1 - t)f(y) - \frac{\lam t (1 - t)}{2}\norm{x - x'}^2.\]
We denote the subdifferential (i.e., set of all subgradients) of a convex function $f: \R^d \to \R$ at $x \in \R^d$ by $\partial f(x)$. Overloading notation, when clear from the context we will write $\partial f(x)$ to denote an arbitrary subgradient. 

\paragraph{Probability.} Let $\mu, \nu$ be two probability densities $\mu$, $\nu$ on the same probability space $\Omega$. We let $\Dtv(\mu, \nu) \defeq \half \int |\mu(\omega) - \nu(\omega)| \d\omega$ denote the total variation distance. The following fact is straightforward to see and will be frequently used.
\begin{fact}\label{fact:cond_TV}
	Let $\event$ be any event that occurs with probability at least $1 - \delta$ under the density $\mu$. Then $\Dtv(\mu, \mu \mid \event) \le \delta$, where $\mu \mid \event$ denotes the conditional distribution of $\mu$ under $\event$.
\end{fact}
For two densities $\mu$, $\nu$, we say that a joint distribution $\Gamma(\mu, \nu)$ over the product space of outcomes is a coupling of $\mu, \nu$ if for $(x, x') \sim \Gamma(\mu, \nu)$, the marginals of $x$ and $x'$ are $\mu$ and $\nu$, respectively. When $\mu$ is absolutely continuous with respect to $\nu$, and $\alpha > 1$, we define the $\alpha$-R\'enyi divergence by
\begin{equation}\label{eq:renyi}
	D_\alpha(\mu \| \nu) \defeq \frac 1 {\alpha - 1} \log \Par{\int \Par{\frac{\mu(\omega)}{\nu(\omega)}}^\alpha \d\nu(\omega)}. 
\end{equation}
$D_\alpha$ is quasiconvex in its arguments, i.e.\ if $\mu = \E_\xi \mu_\xi$ and $\nu = \E_\xi \nu_\xi$ (where $\xi$ is a random variable, and $\mu_\xi$, $\nu_\xi$ are distribution families indexed by $\xi$), then $D_\alpha(\mu \| \nu) \le \max_\xi D_\alpha(\mu_\xi \| \nu_\xi)$. 

%
\section{Framework}
\label{sec:framework}

We now outline our primary technical innovation, a new gradient estimator for stochastic convex optimization (ReSQue). We define this estimator in Section~\ref{ssec:estimator} and prove that it satisfies several local stability properties in a small ball around a ``centerpoint'' used for its definition. In Section~\ref{sec:ball-constrained}, we then give preliminaries on a ``ball acceleration'' framework developed in \cite{CarmonJJJLST20, AsiCJJS21}. This framework aggregates solutions to proximal subproblems defined on small (Euclidean) balls, and uses these subproblem solutions to efficiently solve an optimization problem on a larger domain. Our algorithms in Sections~\ref{sec:parallel} and~\ref{sec:subproblem} instantiate the framework of Section~\ref{sec:ball-constrained} with new subproblem solvers enjoying improved parallelism or privacy, based on our new ReSQue estimator.

\subsection{ReSQue estimators}
\label{ssec:estimator}

Throughout we use $\gamma_\rho: \R^d \to \R_{\ge 0}$ to denote the probability density function of $\Nor(0, \rho^2 \id_d)$, i.e., $\gamma_\rho(x) = (2\pi \rho)^{-\frac d 2} \exp(-\frac{1}{2\rho^2} \norm{x}^2)$. We first define the Gaussian convolution operation.

\begin{definition}[Gaussian convolution]\label{def:gaussian-convolution}
For a function $f: \R^d \to \R$ we denote its convolution with a Gaussian of covariance $\rho^2 \id_d$ by $\hf_\rho \defeq f \ast \gamma_\rho$, i.e.,
\begin{equation}
\hf_\rho(x) \defeq \E_{y\sim \Nor(0, \rho^2 \id_d)}f(x+y)=\int_{y \in \R^n} f(x - y) \gamma_\rho(y)\d y.
\end{equation}
\end{definition}

Three well-known properties of $\hf_\rho$ are that it is differentiable, that if $f$ is $L$-Lipschitz, so is $\hf_\rho$ for any $\rho$, and that $|\hf_\rho - f| \le L\rho\sqrt d$ pointwise (Lemma 8, \cite{BubeckJLLS19}). Next, given a centerpoint $\bx$ and a smoothing radius $\rho$, we define the associated reweighted stochastic query (ReSQue) estimator. 

\begin{definition}[ReSQue estimator]\label{def:stoch_RGCG}
Let $\bx \in \R^d$ and let $f: \R^d \to \R$ be convex. Suppose we have a gradient estimator $g: \R^d \to \R^d$ satisfying $\E g \in \partial f$. We define the \emph{ReSQue estimator of radius $\rho$} as the random vector
\[\tnbx^g \hf_\rho(x) \defeq \frac{\gamma_\rho(x - \bx - \xi)}{\gamma_\rho(\xi)} g(\bx + \xi) \text{ where } \xi \sim \Nor(0, \rho^2\id_d), \]
where we first sample $\xi$, and then independently query $g$ at $\bx + \xi$. When $g$ is deterministically an element of $\partial f$, we drop the superscript and denote the estimator by $\tnbx \hf_\rho$.
\end{definition}

When $g$ is unbiased for $\partial f$ and enjoys a variance bound, the corresponding ReSQue estimator is unbiased for the convolved function, and inherits a similar variance bound.

\begin{lemma}\label{lem:stochastic_varbound}
The estimator in Definition~\ref{def:stoch_RGCG} satisfies the following properties, where expectations are taken over both the randomness in $\xi$ and the randomness in $g$.
\begin{enumerate}
    \item Unbiased: $\E \tnbx^g \hf_\rho(x) = \nabla \hf_\rho(x)$.
    \item Bounded variance: If $\E \norm{g}^2 \le L^2$ everywhere, and $x \in \ball_{\bar{x}}(\rho)$, then
    $\E \|\tnbx^g \hf_\rho(x)\|^2 \le 3L^2$.
\end{enumerate}
\end{lemma}

\begin{proof}
The first statement follows by expanding the expectation over $\xi$ and $g$:
\begin{align*}
\E_g\int \frac{\gamma_\rho(x - \bx - \xi)}{\gamma_\rho(\xi)} g(\bx + \xi) \gamma_\rho(\xi) \d\xi =& \int \frac{\gamma_\rho(x - \bx - \xi)}{\gamma_\rho(\xi)} \partial f(\bx + \xi) \gamma_\rho(\xi) \d\xi\\
=& \int \partial f(\bx + \xi) \gamma_\rho(x - \bx - \xi) \d\xi = \nabla \hf_\rho(x).    
\end{align*}
The last equality used that the integral is a subgradient of $\hf_\rho$, and $\hf_\rho$ is differentiable.

For the second statement, denote $v \defeq x - \bx$ for simplicity. Since $f$ is $L$-Lipschitz,
\begin{align*}
    \E \|\tnbx^g \hf_\rho(x)\|^2 &= \E_g \int \frac{\Par{\gamma_\rho(v - \xi)}^2}{\gamma_\rho(\xi)} \norm{g(\bx + \xi)}^2\d\xi \\
    &\le  L^2 (2\pi \rho)^{-\frac d 2}\int \exp\Par{-\frac{\norm{v - \xi}^2}{\rho^2} + \frac{\norm{\xi}^2}{2\rho^2}} \d\xi.
\end{align*}
Next, a standard calculation for Gaussian integrals shows
\begin{equation}\label{eq:gaussian_integral}\int \exp\Par{\frac{2\inprod{v}{\xi} - \norm{\xi}^2}{2\rho^2}} \d \xi = \exp\Par{\frac{\norm{v}^2}{2\rho^2}} \int \exp\Par{-\frac{\norm{\xi - v}^2}{2\rho^2}} \d\xi = \exp\Par{\frac{\norm{v}^2}{2\rho^2}}(2\pi\rho)^{\frac d 2}.\end{equation}
The statement then follows from \eqref{eq:gaussian_integral}, which yields
\begin{equation}\label{eq:varintegral}
\begin{aligned}
\int \exp\Par{-\frac{\norm{v - \xi}^2}{\rho^2} + \frac{\norm{\xi}^2}{2\rho^2}} \d\xi &= \exp\Par{-\frac{\norm{v}^2}{\rho^2}} \int \exp\Par{\frac{4 \inprod{v}{\xi} - \norm{\xi}^2}{2\rho^2}} \d\xi \\
&= (2\pi \rho)^{\frac d 2} \exp\Par{\frac{2\norm{v}^2}{\rho^2}} \le 3 \cdot (2\pi \rho)^{\frac d 2}
\end{aligned}
\end{equation}
and completes the proof of the second statement.
\end{proof}

When the gradient estimator $g$ is deterministically a subgradient of a Lipschitz function, we can show additional properties about ReSQue. The following lemma will be used in Section~\ref{sec:subproblem} both to obtain higher moment bounds on ReSQue, as well as higher moment bounds on the difference of ReSQue estimators at nearby points, where the bound scales with the distance between the points. 
\begin{lemma}
\label{lem:p_moment_Gaussian}
If $x, x' \in \ball_{\bx}(\frac{\rho}{p})$ for $p \geq 2$ then 
    \begin{gather*}\E_{\xi \sim \Nor(0, \rho^2 \id_d)} \Brack{\Par{\frac{\gamma_\rho(x - \bx -  \xi)}{\gamma_\rho(\xi)}}^p} \le 2,\\
    	\E_{\xi \sim \Nor(0, \rho^2 \id_d)} \Brack{\Abs{\frac{\gamma_\rho(x - \bx - \xi) - \gamma_\rho(x' - \bx - \xi)}{\gamma_\rho(\xi)}}^p} \le \Par{\frac{24p\norm{x - x'}}{\rho}}^p. \end{gather*}
\end{lemma}
We defer a proof to Appendix~\ref{app:facts}, where a helper calculation (Fact~\ref{fact:exp_bound_p}) is used to obtain the result.

\subsection{Ball acceleration}
\label{sec:ball-constrained}

We summarize the guarantees of a recent ``ball acceleration'' framework originally proposed by \cite{CarmonJJJLST20}. For specified parameters $0 < r < R$, this framework efficiently aggregates (approximate) solutions to constrained optimization problems over Euclidean balls of radius $r$ to optimize a function over a ball of radius $R$. Here we give an approximation-tolerant variant of the \cite{CarmonJJJLST20} algorithm in Proposition~\ref{prop:mainballaccel}, which was developed by \cite{AsiCJJS21}. Before stating the guarantee, we require the definitions of three types of oracles. In each of the following definitions, for some function $F: \R^d \to \R$, scalars $\lam, r$, and point $\bx \in \R^d$ which are clear from context, we will denote
\begin{equation}\label{eq:xsbxldef}\xsbxl \defeq \argmin_{x \in \ball_{\bx}(r)} \Brace{ F(x) + \frac \lam 2 \norm{x - \bx}^2}.\end{equation}

We mention that in the non-private settings of prior work \cite{AsiCJJS21, CarmonH22} (and under slightly different oracle access assumptions), it was shown that the implementation of line search oracles (Definition~\ref{def:Ols}) and stochastic proximal oracles (Definition~\ref{def:Opg}) can be reduced to ball optimization oracles (Definition~\ref{def:Obo}). Indeed, such a result is summarized in Proposition~\ref{prop:mainballaccel2} and used in Section~\ref{sec:parallel} to obtain our parallel SCO algorithms. To tightly quantify the privacy loss of each oracle for developing our SCO algorithms in Section~\ref{sec:subproblem} (and to implement these oracles under only the function access afforded by Problem~\ref{prob:sco_basic}), we separate out the requirements of each oracle definition separately.

\begin{definition}[Line search oracle]\label{def:Ols}
We say $\Ols$ is a \emph{$(\Delta, \lam)$-line search oracle} for $F: \R^d \to \R$ if given $\bx \in \R^d$, $\Ols$ returns $x \in \R^d$ with 
\[\norm{x - \xsbxl} \le \Delta.\]
\end{definition}

\begin{definition}[Ball optimization oracle]\label{def:Obo}
We say $\Obo$ is a \emph{$(\phi, \lam)$-ball optimization oracle} for $F: \R^d \to \R$ if given $\bx \in \R^d$, $\Obo$ returns $x \in \R^d$ with 
\[\E\Brack{F(x) + \frac \lam 2 \norm{x - \bx}^2} \le F(\xsbxl) + \frac \lam 2 \norm{\xsbxl - \bx}^2 + \phi.\]
\end{definition}

\begin{definition}[Stochastic proximal oracle]\label{def:Opg}
We say $\Opg$ is a \emph{$(\Delta, \sigma, \lam)$-stochastic proximal oracle} for $F: \R^d \to \R$ if given $\bx \in \R^d$, $\Opg$ returns $x \in \R^d$ with
\[\norm{\E x - \xsbxl} \le \frac \Delta \lam,\; \E\norm{x - \xsbxl}^2 \le \frac{\sigma^2}{\lam^2}.\]
\end{definition}

Leveraging Definitions~\ref{def:Ols},~\ref{def:Obo}, and~\ref{def:Opg}, we state a variant of the main result of \cite{AsiCJJS21}. Roughly speaking, Proposition~\ref{prop:mainballaccel} states that to optimize a function $F$ over a ball of radius $R$, it suffices to query $\approx (\frac R r)^{\frac 2 3}$ oracles which approximately optimize a sufficiently regularized variant of $F$ over a ball of radius $r$. We quantify the types of approximate optimization of such regularized functions in Proposition~\ref{prop:mainballaccel}, and defer a detailed discussion of how to derive this statement from \cite{AsiCJJS21} in Appendix~\ref{app:ballexplain}, as it is stated slightly differently in the original work.\footnote{In particular, we use an error tolerance for the ball optimization oracles, which is slightly larger than in \cite{AsiCJJS21}, following a tighter error analysis given in Proposition 1 of \cite{CarmonH22}.} 

\begin{proposition}\label{prop:mainballaccel}
Let $F: \R^d \to \R$ be $L$-Lipschitz and convex, and let $x^\star \in \ball(R)$. There is an algorithm $\ballacc$ (Algorithm 4, \cite{AsiCJJS21}) taking parameters $r \in [0, R]$ and $\ptot \in (0, LR]$ with the following guarantee. Define 
\[\kappa \defeq \frac{LR}{\epsopt},\; K \defeq \Par{\frac R r}^{\frac 2 3},\; \lams \defeq \frac{\ptot K^2}{R^2} \log^2\kappa.\]
For a universal constant $\Cba > 0$, $\ballacc$ runs in at most $\Cba K \log \kappa$ iterations and produces $x \in \R^d$ such that 
\[\E F(x) \le F(x^\star) + \ptot.\]
Moreover, in each iteration $\ballacc$ requires the following oracle calls (all for $F$).
\begin{enumerate}
    \item At most $\Cba\log(\frac{R\kappa}{r})$ calls to a $(\frac{r}{\Cba}, \lam)$-line search oracle with values of $\lam \in [\frac{\lams}{\Cba}, \frac{\Cba L}{\ptot}]$.
    \item A single call to $(\frac{\lam r^2}{\Cba \log^3\kappa}, \lam)$-ball optimization oracle with  $\lam \in [\frac{\lams}{\Cba}, \frac{\Cba L}{\ptot}]$.
    \item A single call to $(\frac{\ptot}{\Cba R}, \frac{\ptot \sqrt K}{\Cba R}, \lam)$-stochastic proximal oracle with $\lam \in [\frac{\lams}{\Cba}, \frac{\Cba L}{\ptot}]$.
\end{enumerate}
\end{proposition}

The optimization framework in Proposition~\ref{prop:mainballaccel} is naturally compatible with our ReSQue estimators, whose stability properties are local in the sense that they hold in balls of radius $\approx \rho$ around the centerpoint $\bx$ (see Lemma~\ref{lem:p_moment_Gaussian}). Conveniently, $\ballacc$ reduces an optimization problem over a domain of size $R$ to a sequence of approximate optimization problems on potentially much smaller domains of radius $r$. In Sections~\ref{sec:parallel} and~\ref{sec:subproblem}, by instantiating Proposition~\ref{prop:mainballaccel} with $r \approx \rho$, we demonstrate how to use the local stability properties of ReSQue estimators (on smaller balls) to solve constrained subproblems, and consequently design improved parallel and private algorithms. 

Finally, as mentioned previously, in settings where privacy is not a consideration, Proposition 1 of \cite{CarmonH22} gives a direct implementation of all the line search and stochastic proximal oracles required by Proposition~\ref{prop:mainballaccel} by reducing them to ball optimization oracles. The statement in \cite{CarmonH22} also assumes access to \emph{function evaluations} in addition to gradient (estimator) queries; however, it is straightforward to use geometric aggregation techniques (see Lemma~\ref{lem:agg}) to bypass this requirement. We give a slight rephrasing of Proposition 1 in \cite{CarmonH22} without the use of function evaluation oracles, and defer further discussion to Appendix~\ref{app:ballexplain2} where we prove the following.

\begin{proposition}\label{prop:mainballaccel2}
Let $F: \R^d \to \R$ be $L$-Lipschitz and convex, and let $x^\star \in \ball(R)$. There is an implementation of $\ballacc$ (see Proposition~\ref{prop:mainballaccel}) taking parameters $r \in [0, R]$ and $\epsopt \in (0, LR]$ with the following guarantee, where we define $\kappa, K, \lams$ as in Proposition~\ref{prop:mainballaccel}. For a universal constant $\Cba > 0$, $\ballacc$ runs in at most $\Cba K\log\kappa$ iterations and produces $x \in \R^d$ such that $\E F(x) \le F(x^\star) + \epsopt$.
\begin{enumerate}
    \item Each iteration makes at most $\Cba\log^2(\frac{R\kappa}{r})$ calls to $(\frac{\lam r^2}{\Cba}, \lam)$-ball optimization oracle with values of $\lam \in [\frac{\lams}{\Cba}, \frac{\Cba L}{\ptot}]$.
    \item For each $j \in [\lceil\log_2 K + \Cba\rceil]$, at most $\Cba^2 \cdot 2^{-j} K\log(\frac{R\kappa}{r})$ iterations query a $(\frac{\lam r^2}{\Cba 2^j} \cdot \log^{-2}(\frac{R\kappa}{r}), \lam)$-ball optimization oracle for some $\lam \in [\frac{\lams}{\Cba}, \frac{\Cba L}{\ptot}]$.
\end{enumerate}
\end{proposition}

%
\section{Parallel stochastic convex optimization}\label{sec:parallel}

In this section, we present our main results on parallel convex optimization with improved computational depth and total work. We present our main results below in Theorems~\ref{thm:parallel-sgd} and~\ref{thm:parallel-agd}, after formally stating our notation and the SCO problem we study in this section.

\subsection{Preliminaries}\label{ssec:parallel_prelims}

In this section, we study the following SCO problem, which models access to an objective only through the stochastic gradient oracle.

\begin{problem}\label{prob:sco_gen}
	Let $f: \R^d \to \R$ be convex. We assume there exists a \emph{stochastic gradient oracle} $g: \R^d \to \R^d$ satisfying for all $x \in \R^d$, $\E g(x) \in \partial f(x)$, $\E \norm{g(x)}^2 \le L^2$. Our goal is to produce $x \in \R^d$ such that $\E f(x) \le \min_{x^\star \in \ball(R)} f(x^\star) + \epsopt$. We define parameter
	\begin{equation}\label{eq:kappadef}\kappa\defeq \frac{LR}{\epsopt}. \end{equation} 
\end{problem}

When discussing a parallel algorithm which queries a stochastic gradient oracle, in the sense of Problem~\ref{prob:sco_gen}, we separate its complexity into four parameters. The \emph{query depth} is the maximum number of sequential rounds of interaction with the oracle, where queries are submitted in batch. The \emph{total number of queries} is the total number of oracle queries used by the algorithm. The \emph{computational depth and work} are the sequential depth and total amount of computational work, treating each oracle query as requiring $O(1)$ depth and work. For simplicity we assume that all $d$-dimensional vector operations have a cost of $d$ when discussing computation.

\subsection{Proofs of Theorems~\ref{thm:parallel-sgd} and~\ref{thm:parallel-agd}}

\begin{theorem}[Parallel $\epochSGD$-based solver]\label{thm:parallel-sgd}
$\ballacc$ (\Cref{prop:mainballaccel2}) using parallel $\epochSGD$ (Algorithm~\ref{alg:parallel-sgd}) as a ball optimization oracle solves Problem~\ref{prob:sco_gen} with expected error $\epsopt$, with 
\[O\Par{d^{\frac 1 3}\kappa^{\frac 2 3}\log^3(d\kappa)}\text{ query depth and }O\Par{ d^{\frac 1 3}\kappa^{\frac 2 3}\log^3\Par{d\kappa}+\kappa^2\log^4\Par{d\kappa}}\text{ total queries,}\]
and an additional computational cost of 
\[O\Par{ d^{\frac 1 3}\kappa^{\frac 2 3}\log^3\Par{d\kappa}+\kappa^2\log^4\Par{d\kappa}}\text{ depth and }O\Par{\Par{ d^{\frac 1 3}\kappa^{\frac 2 3}\log^3\Par{d\kappa}+\kappa^2\log^4\Par{d\kappa}} \cdot d}\text{ work}.\]
\end{theorem}

\begin{theorem}[Parallel $\NAGD$-based solver]\label{thm:parallel-agd}
$\ballacc$ (\Cref{prop:mainballaccel2}) using parallel $\NAGD$ (Algorithm~\ref{alg:parallel-agd}) as a ball optimization oracle solves Problem~\ref{prob:sco_gen} with expected error $\epsopt$, with 
\begin{gather*}O\Par{d^{\frac 1 3}\kappa^{\frac 2 3}\log\kappa}\text{ query depth}\\
\text{and }O\Par{ d^{\frac 1 3}\kappa^{\frac 2 3}\log^3\Par{d\kappa}+d^{\frac 1 4}\kappa\log^4\Par{ d\kappa}+\kappa^2\log^4\Par{d\kappa}}\text{ total queries,}\end{gather*}
and an additional computational cost of
\begin{gather*}
O\Par{ d^{\frac 1 3}\kappa^{\frac 2 3}\log^3\Par{d\kappa}+d^{\frac 1 4}\kappa\log^4\Par{ d\kappa}}\text{ depth}\\
\text{and } O\Par{\Par{d^{\frac 1 3}\kappa^{\frac 2 3}\log^3\Par{d\kappa}+d^{\frac 1 4}\kappa\log^4\Par{ d\kappa}+\kappa^2\log^4\Par{d\kappa}} \cdot d}\text{ work}.
\end{gather*}
\end{theorem}

The query depth, total number of queries, and total work for both of our results are the same (up to logarithmic factors). The main difference is that $\NAGD$ attains an improved computational depth for solving SCO, compared to using $\epochSGD$. Our results build upon the $\ballacc$ framework in~Section~\ref{sec:ball-constrained}, combined with careful parallel implementations of the required ball optimization oracles to achieve improved complexities. 

We begin by developing our parallel ball optimization oracles using our ReSQue estimator machinery from Section~\ref{ssec:estimator}. First, Proposition~\ref{prop:mainballaccel2} reduces Problem~\ref{prob:sco_gen} to implementation of a ball optimization oracle. Recall that a ball optimization oracle (Definition~\ref{def:Obo}) requires an approximate solution $x$ of a regularized subproblem. In particular, for some accuracy parameter $\phi$, and defining $\xsbxl$ as in \eqref{eq:xsbxldef}, we wish to compute a random $x \in \ball_{\bar{x}}(r)$ such that
\begin{align*}
\E\Brack{\hf_\rho(x) + \frac \lam 2 \norm{x - \bx}^2} \le \hf_\rho(\xsbxl) + \frac \lam 2 \norm{\xsbxl - \bx}^2 + \phi,\; x \in \ball_{\bx}(r).
\end{align*}
Note that such a ball optimization oracle can satisfy the requirements of Proposition~\ref{prop:mainballaccel2} with $F \gets \hf_\rho$, $r \gets \rho$. In particular, Lemma~\ref{lem:stochastic_varbound} gives a gradient estimator variance bound under the setting $r = \rho$.

\paragraph{$\epochSGD$.} We implement $\epochSGD$~\cite{HazanK14,AsiCJJS21}, a variant of standard stochastic gradient descent on regularized objective functions, in parallel using the stochastic ReSQue estimator constructed in~Definition~\ref{def:stoch_RGCG}. Our main observation is that the gradient queries in Definition~\ref{def:stoch_RGCG} can be implemented in parallel at the beginning of the algorithm. We provide the pseudocode of our parallel implementation of $\epochSGD$ in~Algorithm~\ref{alg:parallel-sgd} and state its guarantees in~Proposition~\ref{lem:parallel-sgd}.

\begin{algorithm2e}\label{alg:parallel-sgd}
\DontPrintSemicolon
\caption{$\epochSGD(f, g, \bar{x}, r, \rho,\lambda,\phi)$}
\textbf{Input:} $f:\R^d\rightarrow\R$ and $g: \R^d \to \R$ satisfying the assumptions of Problem~\ref{prob:sco_gen}, $\bar{x} \in \R^d$, $r, \rho, \lam, \phi >0$\;
$\eta_1 \gets \frac 1 {4\lam}$, $T_1 \gets 16$, $T \gets \lceil \frac{48L^2}{\lam \phi}\rceil$\;
Sample $\xi_i\sim\Nor(0,\rho^2\id_d)$, $i\in[2T]$ independently \;
Query $g(\bar{x}+\xi_i)$ for all $i\in[2T]$ (in parallel)\;
$x_1^0 \gets \bar{x}$, $k \gets 1$\;
\While{$\sum_{j\in[k]}T_j\le T$}{
$x_k^1 \gets \argmin_{x\in\ball_{\bar{x}}(r)}\Brace{\frac{\eta_k\lambda}{2}\|x-\bar{x}\|^2+\frac{1}{2}\|x-x_k^0\|^2}$\;
\For{$t \in [T_k-1]$}{
$i \gets \sum_{j\in[k-1]} T_j+t$\;
$\tnbx^g \hf_\rho(x_k^t) \gets \frac{\gamma_\rho(x_k^t - \bx - \xi_i)}{\gamma_\rho(\xi_i)} g(\bx + \xi_i)$\;
$x_k^{t+1} \gets \argmin_{x\in\ball_{\bx}(r)}\Brace{ \eta_k\langle \tnbx^g \hf_\rho(x_k^t), x\rangle+\frac{\eta_k\lambda}{2}\|x-\bar{x}\|^2+\frac{1}{2}\|x-x_k^t\|^2}$
}
$x_{k+1}^0 \gets \frac{1}{T_k}\sum_{t\in[T_k]}x_k^t$, $T_{k+1} \gets 2T_k$, $\eta_{k+1} \gets \frac{\eta_k} 2$, $k \gets k+1$
}
\textbf{return} $x_{k}^0$
\end{algorithm2e}

\begin{proposition}[Proposition 3,~\cite{AsiCJJS21}]\label{lem:parallel-sgd}
Let $f, g$ satisfy the assumptions of Problem~\ref{prob:sco_gen}. When $\rho = r$, Algorithm~\ref{alg:parallel-sgd} is a $(\phi,\lambda)$-ball optimization oracle for $\hf_\rho$ which makes $O(\frac{L^2}{\phi\lam})$ total queries to $g$ with constant query depth,
and an additional computational cost of $O(\frac{L^2}{\phi\lam})$ depth and work.
\end{proposition}

\paragraph{$\NAGD$.} We can also implement $\NAGD$~\cite{GhadimiL12}, a variant of accelerated gradient descent under stochastic gradient queries, in parallel using stochastic ReSQue estimators. We provide the pseudocode of our parallel implementation of $\NAGD$ in~Algorithm~\ref{alg:parallel-agd} and state its guarantees in~Lemma~\ref{lem:parallel-agd}.

\begin{algorithm2e}[t]\label{alg:parallel-agd}
\DontPrintSemicolon
\caption{$\NAGD(f, \bar{x}, r, \rho,\lambda,\phi)$}
\textbf{Input:} $f:\R^d\to\R$, $g: \R^d \to \R$ satisfying the assumptions of Problem~\ref{prob:sco_gen}, $\bar{x} \in \R^d$, $r,\rho, \lam, \phi >0$\;
$K \gets \lceil \log_2(\frac{\lam r^2}{\phi}) \rceil $, $T \gets \lceil4\sqrt{\frac{ L}{\rho\lambda} + 1}\rceil$, $N_k \gets \left\lceil48 \cdot 2^k \cdot \frac{L^2}{\lam^2 r^2 T}\right\rceil$ for $k \in [K]$\;
Sample $\xi_i\sim\Nor(0,\rho^2\id_d)$, $i\in[N]$ independently, for $N = T\cdot(\sum_{k\in[K]}N_k)$ \;
Query $g(\bar{x}+\xi_i)$ for all $i\in[N]$ (in parallel)\;
$x_0^{\ag} \gets \bx$, $x_0 \gets \bar{x}$\;
\For{$k \in [K]$}{
\For{$t \in [T]$}{
$\alpha_t \gets \frac{2}{t+1}$, $\gamma_t \gets \frac{4(\frac L \rho + \lam)}{ t(t+1)}$\;
$x_t^{\md} \gets \frac{(1-\alpha_t)(\lambda+\gamma_t)}{\gamma_t+(1-\alpha_t^2)\lambda}x_{t-1}
^{\ag}+\frac{\alpha_t(1-\alpha_t)(\lambda+\gamma_t)}{\gamma_t+(1-\alpha_t^2)\lambda}x_{t-1}$\;
$N_{T,[k-1]} \gets T\cdot\sum_{k'\in[k-1]}N_{k'}$\;
$\widehat{\nabla} f(x_t^\md) \gets \frac{1}{N_k}\sum_{n\in[N_k]} \frac{\gamma_\rho(x_t^\md - \bx - \xi_{N_{T,[k-1]}+n})}{\gamma_\rho(\xi_{N_{T,[k-1]}+n})} g(\bx + \xi_{N_{T,[k-1]}+n})$\;
$x_{t} \gets \argmin_{x\in\ball_{\bar{x}}(r)} \Psi_t(x)$, where $\Psi_t(x) \defeq \langle \alpha_t \widehat{\nabla} f(x_t^\md) + \lam(x_t^{\md} - \bx), x-x_t\rangle+\frac{\gamma_t+\lambda(1-\alpha_t)}{2}\|x-x_{t-1}\|^2+\frac{\lambda\alpha_t}{2}\|x-x_{t}^{\md}\|^2$\;
$x_{t}^\ag \gets \alpha_tx_{t}+(1-\alpha_t)x_{t-1}^\ag$
}
$x_0^\ag\gets x_T^{\ag}$, $x_0\gets x_T^\ag$\;
}
\textbf{Return:} $x_{T}^\ag$
\end{algorithm2e}

\begin{proposition}[Special case of Theorem 1, \cite{GhadimiL12}]\label{lem:parallel-agd}
Let $f, g$ satisfy the assumptions of Problem~\ref{prob:sco_gen}. When $\rho = r$, Algorithm~\ref{alg:parallel-agd} is a $(\phi,\lambda)$-ball optimization oracle for $\hf_\rho$ which makes
\[O\Par{\sqrt{1 + \frac{L}{\rho\lambda}}\log\Par{\frac{\lam r^2}{\phi}}+\frac{L^2}{\lambda\phi}} \text{ total queries}\]
with constant query depth, 
and an additional computational cost of
\[O\Par{\sqrt{1 + \frac{L}{\rho\lambda}}\log\Par{\frac{\lam r^2}{\phi}}} \text{ depth and } O\Par{\sqrt{1 + \frac{L}{\rho\lambda}}\log\Par{\frac{\lam r^2}{\phi}}+\frac{L^2}{\lambda\phi}} \text{ work.}\]
\end{proposition}

Because the statement of Proposition~\ref{lem:parallel-agd} follows from specific parameter choices in the main result in \cite{GhadimiL12}, we defer a more thorough discussion of how to obtain this result to Appendix~\ref{app:GLexplain}.

\paragraph{Main results.} We now use our parallel ball optimization oracles to prove Theorems~\ref{thm:parallel-sgd} and~\ref{thm:parallel-agd}.
\begin{proof}[Proofs of~Theorems~\ref{thm:parallel-sgd} and~\ref{thm:parallel-agd}]
We use~Proposition~\ref{prop:mainballaccel2} with $r=\rho = \frac{\epsopt}{\sqrt{d}L}$ on $F \gets \hf_\rho$, which approximates $f$ to additive $\epsopt$, and $x^\star \defeq \arg\min_{x \in \ball(R)} f(x)$. Rescaling $\epsopt$ by a constant from the guarantee of Proposition~\ref{prop:mainballaccel2} gives the error claim. For the oracle query depths, note that each ball optimization oracle (whether implemented using Algorithm~\ref{alg:parallel-sgd} or Algorithm~\ref{alg:parallel-agd}) has constant query depth, and at most $O(\log^2(d\kappa))$ ball optimization oracles are queried per iteration on average. Note that (see Proposition~\ref{prop:mainballaccel})
\[\kappa = \frac{LR}{\epsopt},\; K = \Par{\frac R r}^{\frac 2 3} = d^{\frac 1 3}\kappa^{\frac 2 3},\; \lams = \frac{\epsopt K^2}{R^2}\log^2\kappa = \frac{\epsopt d^{\frac 2 3}\kappa^{\frac 4 3}}{R^2}\log^2\kappa. \]

For the total oracle queries, computational depth, and work, when implementing each ball optimization oracle with $\epochSGD$, we have that for $\jmax \defeq \lceil\log_2 K + \Cba\rceil$, these are all
\begin{gather*}
O\Par{K\log\Par{d\kappa}\cdot\Par{\sum_{j \in [\jmax]}\frac{1}{2^j}\Par{\frac{L^2 \cdot 2^j\log^2(d\kappa)}{\lams^2 r^2}}+\Par{\frac{L^2}{\lams^2 r^2}}\log^2\Par{d\kappa}}}\\
= O\Par{K\log^4\Par{d\kappa}\cdot\frac{L^2}{\lams^2 r^2}}
= O\Par{\kappa^2\log^4\Par{d\kappa}}
\end{gather*}
due to~Proposition~\ref{lem:parallel-sgd}. The additional terms in the theorem statement are due to the number of ball oracles needed. For the computational depth when implementing each ball optimization oracle with $\NAGD$ we have that (due to~Proposition~\ref{lem:parallel-agd}), it is bounded by
\begin{gather*}
O\Par{K\log^3(d\kappa) \cdot \sqrt{\frac{L}{r\lams}}\log(d\kappa)} = O\Par{K\log^4(d\kappa) \cdot \frac{\sqrt \kappa}{K^{\frac 1 4 }}} = O\Par{d^{\frac 1 4} \kappa \log^4(d\kappa)}.
\end{gather*}
Finally, for the total oracle queries and work bounds, the bound due to the $\frac{L^2}{\lam\phi}$ term is as was computed for Theorem~\ref{thm:parallel-sgd}, and the bound due to the other term is the same as the above display.
\end{proof} %
\section{Private stochastic convex optimization}\label{sec:subproblem}

We now develop our main result on an improved gradient complexity for private SCO. 
First, in Section~\ref{ssec:privacy}, we introduce several variants of differential privacy including a relaxation of R\'enyi differential privacy \cite{Mir17}, which tolerates a small amount of total variation error.
Next, in Sections~\ref{ssec:erm_convex},~\ref{ssec:erm_sc}, and~\ref{ssec:proxgrad}, we build several private stochastic optimization subroutines which will be used in the ball acceleration framework of Proposition~\ref{prop:mainballaccel}. Specifically, these subroutines will be called as the oracles in Definitions~\ref{def:Ols},~\ref{def:Obo}, and~\ref{def:Opg} with the parameters required by Proposition~\ref{prop:mainballaccel} in the proof of our main result (see \eqref{eq:ols_params}, \eqref{eq:obo_params}, and \eqref{eq:opg_params}). Finally, in Sections~\ref{ssec:private_erm} and~\ref{ssec:private_sco}, we give our main results on private ERM and SCO respectively, by leveraging the subroutines we develop.

\subsection{Preliminaries}\label{ssec:privacy}

In this section, we study the following specialization of Problem~\ref{prob:sco_gen} naturally compatible with preserving privacy with respect to samples, through the formalism of DP (to be defined shortly).

\begin{problem}\label{prob:sco_basic}
	Let $\dist$ be a distribution over $\calS$, and suppose there is a family of functions indexed by $s \in \calS$, such that $f(\cdot; s): \R^d \to \R$ is convex for all $s \in \calS$. Let $\data \defeq \{s_i\}_{i \in [n]}$ consist of $n$ i.i.d.\ draws from $\dist$, and define the \emph{empirical risk} and \emph{population risk} by
	\begin{align*}
		\Ferm(x) \defeq \frac 1 n \sum_{i \in [n]} f(x; s_i)
		\text{ and }
		\Fpop(x) \defeq \E_{s\sim \dist} {f(x; s)}.
	\end{align*}
	We denote $f^i \defeq f(\cdot; s_i)$ for all $i \in [n]$, and assume that for all $s \in \calS$, $f(\cdot; s)$ is $L$-Lipschitz. We are given $\data$, and can query subgradients of the ``sampled functions'' $f^i$. Our goal is to produce $x \in \R^d$ such that $\E \Fpop(x) \le \min_{x^\star \in \ball(R)} \Fpop(x^\star) + \epsopt$. We again define $\kappa=\frac{LR}{\epsopt}$ as in \eqref{eq:kappadef}. 
\end{problem}

In the ``one-pass'' setting where we only query each $\partial f^i$ a single time, we can treat each $\partial f^i$ as a bounded stochastic gradient of the underlying population risk $\Fpop$. We note the related problem of \emph{empirical risk minimization}, i.e., optimizing $\Ferm$ (in the setting of Problem~\ref{prob:sco_basic}), can also be viewed as a case of Problem~\ref{prob:sco_gen} where we construct $g$ by querying $\partial f^i$ for $i\sim_{\text{unif.}}[n]$. We design $(\epsdp, \delta)$-DP algorithms for solving Problem~\ref{prob:sco_basic} which obtain small optimization error for $\Ferm$ and $\Fpop$. To disambiguate, we will always use $\epsopt$ to denote an optimization error parameter, and $\epsdp$ to denote a privacy parameter. Our private SCO algorithm will require querying $\partial f^i$ multiple times for some $i \in [n]$, and hence incur bias for the population risk gradient. Throughout the rest of the section, following the notation of Problem~\ref{prob:sco_basic}, we will fix a dataset $\data \in \calS^n$ and define the empirical risk $\Ferm$ and population risk $\Fpop$ accordingly. We now move on to our privacy definitions.

We say that two datasets $\data = \{s_i\}_{i \in [n]} \in \calS^n$ and $\data' = \{s'_i\}_{i \in [n]} \in \calS^n$ are \emph{neighboring} if $|\{i \mid s_i \neq s'_i \}| = 1$. We say a mechanism (i.e., a randomized algorithm) $\mech$ satisfies $(\epsdp, \delta)$-differential privacy (DP) if, for its output space $\Omega$ and all neighboring $\data$, $\data'$, we have for all $S \subseteq \Omega$,
\begin{equation}\label{eq:dpepsdelta}
	\Pr[\mech(\data) \in S] \le \exp(\epsdp) \Pr[\mech(\data') \in S] + \delta.
\end{equation}
We extensively use the notion of R\'enyi differential privacy (RDP) due to its compatibility with the subsampling arguments we will use, as well as an approximate relaxation of its definition which we introduce. While it is likely that our results can be recovered (possibly up to logarithmic terms) by accounting for privacy losses via approximate differential privacy, we present our privacy accounting via RDP to simplify calculations.

We say that a mechanism $\mech$ satisfies $(\alpha, \eps)$-R\'enyi differential privacy if for all neighboring $\data, \data' \in \calS^n$, the $\alpha$-R\'enyi divergence~\eqref{eq:renyi} satisfies
\begin{equation}\label{eq:rdp}
D_\alpha(\mech(\data) \| \mech(\data')) \le \eps.
\end{equation}
RDP has several useful properties which we now summarize.

\begin{proposition}[Propositions 1, 3, and 7, \cite{Mir17}]\label{prop:rdp_facts}
RDP has the following properties.
\begin{enumerate}
    \item (Composition): Let $\mech_1: \calS^n \to \Omega$ satisfy $(\alpha, \eps_1)$-RDP and $\mech_2: \calS^n \times \Omega \to \Omega'$ satisfy $(\alpha, \eps_2)$-RDP for any input in $\Omega$. Then the composition of $\mech_2$ and $\mech_1$, defined as $\mech_2(\data, \mech_1(\data))$ satisfies $(\alpha, \eps_1 + \eps_2)$-RDP.
    \item (Gaussian mechanism): For $\mu, \mu' \in \R^d$, $D_\alpha(\Nor(\mu, \sigma^2 \id_d) \| \Nor(\mu', \sigma^2 \id_d)) \le \frac{\alpha}{2\sigma^2} \norm{\mu - \mu'}^2$.
    \item (Standard DP): If $\mech$ satisfies $(\alpha, \eps)$-RDP, then for all $\delta \in (0, 1)$, $\mech$ satisfies $(\eps + \frac 1 {\alpha - 1} \log \frac 1 \delta, \delta)$-DP.
\end{enumerate}
\end{proposition}

We also use the following definition of approximate R\'enyi divergence:
\begin{equation}\label{eq:approx_renyi}
	D_{\alpha, \delta}(\mu \| \nu) \defeq \min_{\Dtv(\mu', \mu) \le \delta, \Dtv(\nu', \nu) \le \delta} D_{\alpha}(\mu' \| \nu').
\end{equation}

We relax the definition \eqref{eq:rdp} and say that $\mech$ satisfies $(\alpha, \eps, \delta)$-RDP if for all neighboring $\data$, $\data' \in \calS^n$, recalling definition \eqref{eq:approx_renyi},
\[D_{\alpha, \delta}(\mech(\data) \| \mech(\data')) \le \eps. \]
The following is then immediate from Proposition~\ref{prop:rdp_facts}, and our definition of approximate RDP, by coupling the output distributions with the distributions realizing the minimum \eqref{eq:approx_renyi}.

\begin{corollary}\label{cor:approx_rdp}
If $\mech$ satisfies $(\alpha, \eps, \delta)$-RDP, then for all $\delta' \in (0, 1)$,  $\mech$ satisfies $(\epsdp, \delta' + (1 + \exp(\epsdp))\delta)$-DP for $\epsdp \defeq \eps + \frac 1 {\alpha - 1} \log \frac 1 {\delta'}$.
\end{corollary}

\begin{proof}
Let $\mu$, $\nu$ be within total variation $\delta$ of $\mech(\data)$ and $\mech(\data')$, such that $D_{\alpha}(\mu\|\nu)\le \epsilon$ 
 and hence for any event $S$,
\[\Pr_{\omega \sim \mu}\Brack{\omega \in S} \le \exp(\epsdp) \Pr_{\omega \sim \nu}[\omega \in S] + \delta'. \]
Combining the above with
\[\Pr_{\omega \sim \mech(\data)}\Brack{\omega \in S} - \delta \le \Pr_{\omega \sim \mu}[\omega \in S],\; \Pr_{\omega \sim \nu}[\omega \in S] \le \Pr_{\omega \sim \mech(\data')}\Brack{\omega \in S} + \delta, \]
we have
\begin{align*}
    \Pr_{\omega\sim \mech(\data)}[\omega\in S] &\le \exp(\epsdp)\Pr_{\omega\sim \nu}[\omega\in S]+\delta'+\delta \\
    &\le \exp(\epsdp)\Pr_{\omega\sim\mech(\data')}[\omega\in S]+\delta'+(1+\exp(\epsdp))\delta.
\end{align*}
\end{proof}

Finally, our approximate RDP notion enjoys a composition property similar to standard RDP.

\begin{lemma}\label{lem:composition_rdp}
Let $\mech_1: \calS^n \to \Omega$ satisfy $(\alpha, \eps_1, \delta_1)$-RDP and $\mech_2: \calS^n \times \Omega \to \Omega'$ satisfy $(\alpha, \eps_2, \delta_2)$-RDP for any input in $\Omega$. Then the composition of $\mech_2$ and $\mech_1$, defined as $\mech_2(\data, \mech_1(\data))$ satisfies $(\alpha, \eps_1 + \eps_2, \delta_1 + \delta_2)$-RDP.
\end{lemma}
\begin{proof}
Let $\data$, $\data'$ be neighboring datasets, and let $\mu$, $\mu'$ be distributions within total variation $\delta_1$ of $\mech_1(\data)$, $\mech_1(\data')$ realizing the bound $D_\alpha(\mu \|\mu') \le \eps_1$. For any $\omega \in \Omega$, similarly let $\nu_\omega$, $\nu'_\omega$ be the distributions within total variation $\delta_2$ of $\mech_2(\data, \omega)$ and $\mech_2(\data', \omega)$ realizing the bound $D_\alpha(\nu_\omega \| \nu'_\omega) \le \eps_2$. Finally, let $P_1$ be the distribution of $\omega \in \Omega$ according to $\mech_1(\data)$, and $Q_1$ to be the distribution of $\mech_1(\data')$; similarly, let $P_{2, \omega}$, $Q_{2, \omega}$ be the distributions of $\omega' \in \Omega'$ according to $\mech_2(\data, \omega)$ and $\mech_2(\data', \omega)$. We first note that by a union bound,
\begin{align*}\Dtv\Par{\int \nu_\omega(\omega') \mu(\omega) d\omega d\omega', \int P_1(\omega) P_{2, \omega}(\omega') d\omega d\omega'} \le \delta_1 + \delta_2,\\ \Dtv\Par{\int \nu'_\omega(\omega') \mu'(\omega) d\omega d\omega', \int Q_1(\omega) Q_{2, \omega}(\omega') d\omega d\omega'} \le \delta_1 + \delta_2. \end{align*}
Finally, by Proposition 1 of \cite{Mir17}, we have
\[D_\alpha\Par{\int \nu_\omega(\omega') \mu(\omega) d\omega d\omega' \Bigg\| \int \nu'_\omega(\omega') \mu'(\omega) d\omega d\omega'} \le \eps_1 + \eps_2. \]
Combining the above two displays yields the claim.
\end{proof}

\subsection{Subsampled smoothed ERM solver: the convex case}\label{ssec:erm_convex}

We give an ERM algorithm that takes as input a dataset $\data \in \calS^n$, parameters $T \in \N$ and $r, \rho, \beta > 0$, and a center point $\bx \in \R^d$. Our algorithm is based on a localization approach introduced by \cite{FKT20} which repeatedly decreases a domain size to bound the error due to adding noise for privacy. In particular we will obtain an error bound on $\widehat{\Ferm_\rho}$ 
with respect to the set $\ball_{\bx}(r)$, using at most $T$ calls to the ReSQue estimator in Definition~\ref{def:stoch_RGCG} with a deterministic subgradient oracle. Here we recall that $\Ferm$ is defined as in Problem~\ref{prob:sco_basic}, and $\widehat{\Ferm_\rho}$ is correspondingly defined as in Definition~\ref{def:gaussian-convolution}.
Importantly, our ERM algorithm developed in this section attains RDP bounds improving with the subsampling parameter $\frac T n$ when $T \ll n$, due to only querying $T$ random samples in our dataset.

\begin{algorithm2e}[t!]\label{alg:subsample_convex}
\caption{Subsampled ReSQued ERM solver, convex case}
\textbf{Input:} $\bx \in \R^d$, ball radius, convolution radius, and privacy parameter $r, \rho, \beta > 0$, dataset $\data \in \calS^n$, iteration count $T \in \N$ \\
$\hT \gets 2^{\lfloor \log_2 T \rfloor}$, $k \gets \log_2 \hT$, $\eta \gets \frac r L \min(\frac{1}{\sqrt{T}}, \frac \beta {\sqrt d})$, $x_0 \gets \bx$ \\
\For{$i \in [k]$}
{
    $T_i \gets 2^{-i} \hT$, $\eta_i \gets 4^{-i}\eta$, $\sigma_i \gets \frac{L\eta_i}{\beta}$\label{line:params}\\
    $y_0 \gets x_{i-1}$\\
    \For{$j \in [T_i]$}
    {
    	$z_{i,j}\sim_{\text{unif.}} [n]$ \\
        $y_j \gets \Pi_{\ball_{\bx}(r)}(y_{j-1}-\eta_i\tnbx \hf_\rho^{z_{i, j}}(y_{j-1}))$ \label{line:psgd} \Comment*{PSGD step using ReSQue (See Definition~\ref{def:stoch_RGCG}) for a subsampled function. Lemma~\ref{lem:group_privacy_convex} denotes the random Gaussian sample by $\xi_{i,j}$.}
    }
    $\by_i \gets \frac 1 {T_i} \sum_{j \in [T_i]} y_j$ \\
    $x_{i} \gets \by_i + \zeta_i$, for $\zeta_i\sim\Nor(0,\sigma^2_i \id_d)$\label{line:addnoise}
}

\textbf{return} $x_{k}$

\end{algorithm2e}

We summarize our optimization and privacy guarantees on Algorithm~\ref{alg:subsample_convex} in the following. The proof follows by combining Lemma~\ref{lem:erm_convex_utility} (the utility bound) and Lemma~\ref{lem:erm_convex_privacy} (the privacy bound).

\begin{proposition}\label{prop:subsample_convex_bounds}
Let $\xsbx \in \argmin_{x \in \ball_{\bx}(r)} \widehat{\Ferm_\rho}(x)$. Algorithm~\ref{alg:subsample_convex} uses at most $T$ gradients and produces $x \in \ball_{\bx}(r)$ such that, for a universal constant $\Ccvx$,
\[\E\Brack{\widehat{\Ferm_\rho}(x)} - \widehat{\Ferm_\rho}(\xsbx) \le \Ccvx Lr\Par{\frac{\sqrt d}{\beta T} + \frac {1} {\sqrt T}}. \]
Moreover, there is a universal constant $\Cpriv \ge 1$, such that if $\frac T n \le \frac 1 {\Cpriv}$, $\beta^2 \log^2(\frac 1 \delta) \le \frac 1 {\Cpriv}$, $\delta \in (0, \frac 1 6)$, and $\frac \rho r \ge \Cpriv\log^2(\frac{\log T}{\delta})$, Algorithm~\ref{alg:subsample_convex} satisfies $(\alpha, \alpha\tau, \delta)$-RDP for
\[\tau \defeq \Cpriv\Par{\beta\log\Par{\frac 1 \delta} \cdot \frac T n}^2
\text{ and }
\alpha \in \Par{1, \frac{1}{\Cpriv\beta^2\log^2(\frac 1 \delta)}}.\]
\end{proposition}

\paragraph{Utility analysis.} We begin by proving a utility guarantee for Algorithm~\ref{alg:subsample_convex}, following \cite{FKT20}.

\begin{lemma}\label{lem:erm_convex_utility}
Let $\xsbx \defeq \argmin_{x \in \ball_{\bx}(r)} \widehat{\Ferm_\rho}(x)$. We have, for a universal constant $\Ccvx$,
\[\E\Brack{\widehat{\Ferm_\rho}(x_k)} - \widehat{\Ferm_\rho}(\xsbx) \le \Ccvx Lr\Par{\frac{\sqrt d}{\beta T} + \frac {1} {\sqrt T}}. \]
\end{lemma}
\begin{proof}
Denote $F \defeq \widehat{\Ferm_\rho}$, $\by_0 \defeq \xsbx$, and $\zeta_0 \defeq \bx - \xsbx$, where by assumption $\norm{\zeta_0} \le r$. We begin by observing that in each run of Line~\ref{line:psgd}, by combining the first property in Lemma~\ref{lem:stochastic_varbound} with the definition of $\Ferm$, we have that $\E\big[\tnbx \hf_\rho^{z_{i, j}}(y_{j-1})\mid y_{j-1} \big] \in \partial F(y_{j - 1})$. Moreover, by the second property in Lemma~\ref{lem:stochastic_varbound} and the fact that $f^{z_{i,j}}$ is $L$-Lipschitz,
\[\E \norm{\tnbx \hf_\rho^{z_{i, j}}(y_{j-1})}^2 \le 3L^2. \]
We thus have
\begin{equation}\label{eq:telescope_error}
\begin{aligned}
\E\Brack{F(x_k)} - F(\xsbx) &= \sum_{i \in [k]} \E[F(\by_i) - F(\by_{i - 1})] + \E\Brack{F(x_k) - F(\by_k)} \\
&\le \sum_{i \in [k]} \Par{\frac{\E\Brack{\norm{x_{i - 1} - \by_{i - 1}}^2 } }{2\eta_i T_i} + \frac{3\eta_i L^2}{2}} + L \E\Brack{\norm{x_k - \by_k}} \\
&\le \frac{8r^2}{\eta T} + 4\sum_{i \in [k - 1]} \frac{\sigma_i^2 d}{\eta_i T_i} + \sum_{i \in [k]} \frac{3\eta_i L^2}{2} + L\sigma_k \sqrt d.
\end{aligned}
\end{equation}
In the second line, we used standard regret guarantees on projected stochastic gradient descent, e.g.\ Lemma 7 of \cite{HazanK14}, where we used that all $\by_i \in \ball_{\bx}(r)$; in the third line, we used 
\[\E[\norm{x_k - \by_k}] \le \sqrt{\E\Brack{\norm{x_k - \by_k}^2}} = \sqrt{ \E\Brack{\norm{\zeta_k}^2}} = \sigma_k\sqrt d\] 
by Jensen's inequality. Continuing, we have by our choice of parameters that $\frac{\sigma_i^2}{\eta_i T_i} \le 2^{-i}\frac{L^2 \eta}{2\beta^2 \hT}$, hence
\begin{align*}\E\Brack{F(x_k)} - F(\xsbx) &\le \frac{8r^2}{\eta T} + \frac{4L^2\eta d}{\beta^2 \hT} + \frac{3\eta L^2}{2} + \frac{L^2\eta \sqrt d}{\beta} \cdot \frac{1}{\hT^2} \\
&\le \Par{\frac{8Lr}{\sqrt T} + \frac{8Lr\sqrt d}{\beta T}} + \frac{8Lr\sqrt d}{\beta T} + \frac{3Lr}{2\sqrt T} + \frac{Lr}{\sqrt T}. \end{align*}
Here we used that $2\hT \ge T$ and $\hT^2 \ge \sqrt T$, for all $T \in \N$.
\end{proof}

\paragraph{Privacy analysis.} We now show that our algorithm satisfies a strong (approximate) RDP guarantee. Let $\data' = \{s'_i\}_{i \in [n]} \in \calS^n$ be such that $\data = \{s_i\}_{i \in [n]}$ and $\data'$ are neighboring, and without loss of generality assume $s'_1 \neq s_1$. Define the multiset
\begin{equation}\label{eq:inddef}\ind \defeq \{z_{i, j} \mid i \in [k], j \in [T_i]\}\end{equation}
to contain all sampled indices in $[n]$ throughout Algorithm~\ref{alg:subsample_convex}. We begin by giving an (approximate) RDP guarantee conditioned on the number of times ``$1$'' appears in $\ind$. The proof of Lemma~\ref{lem:group_privacy_convex} is primarily based on providing a potential-based proof of a ``drift bound,'' i.e., how far away iterates produced by two neighboring datasets drift apart (coupling all other randomness used). To carry out this potential proof, we rely on the local stability properties afforded by Lemma~\ref{lem:p_moment_Gaussian}.

\begin{lemma}\label{lem:group_privacy_convex}
Define $\ind$ as in \eqref{eq:inddef} in one call to Algorithm~\ref{alg:subsample_convex}. Let $\ind$ be deterministic (i.e., this statement is conditioned on the realization of $\ind$). Let $b$ be the number of times the index $1$ appears in $\ind$. Let $\mu$ be the distribution of the output of Algorithm~\ref{alg:subsample_convex} run on $\data$, and $\mu'$ be the distribution when run on $\data'$, such that $\data$ and $\data'$ are neighboring and differ in the first entry, and the only randomness is in the Gaussian samples used to define ReSQue estimators and on Line~\ref{line:addnoise}. Suppose $\frac{\rho}{r} \ge 1728\log^2(\frac{\log T}{\delta})$. Then we have for any $\alpha > 1$,
\[D_{\alpha, \delta}(\mu \| \mu') \le 1500\alpha \beta^2 b^2. \]
\end{lemma}
\begin{proof}
Throughout this proof we treat $\ind$ as fixed with $b$ occurrences of the index $1$. Let $b_i$ be the number of times $1$ appears in $\ind_i \defeq \{z_{i, j} \mid j \in [T_i]\}$, such that $\sum_{i \in [k]} b_i = b$. We first analyze the privacy guarantee of one loop, and then analyze the privacy of the whole algorithm.

We begin by fixing some $i \in [k]$, and analyzing the RDP of the $i^{\text{th}}$ outer loop in Algorithm~\ref{alg:subsample_convex}, conditioned on the starting point $y_0$. Consider a particular realization of the $T_i$ Gaussian samples used in implementing Line~\ref{line:psgd}, $\Xi_i \defeq \{\xi_{i,j}\}_{j \in [T_i]}$, where we let $\xi_{i,j} \sim \Nor(0, \rho^2 \id_d)$ denote the Gaussian sample used to define the update to $y_{j - 1}$. Conditioned on the values of $\ind_i$, $\Xi_i$, the $i^{\text{th}}$ outer loop in Algorithm~\ref{alg:subsample_convex} (before adding $\zeta_i$ in Line~\ref{line:addnoise}) is a deterministic map. For a given realization of $\ind_i$ and $\Xi_i$, we abuse notation and denote $\{y_j\}_{j \in [T_i]}$ to be the iterates of the $i^{\text{th}}$ outer loop in Algorithm~\ref{alg:subsample_convex} using the dataset $\data$ starting at $y_0$, and $\{y'_j\}_{j \in [T_i]}$ similarly using $\data'$. Finally, define
\[\Phi_j \defeq \norm{y_j - y'_j}^2,\; p \defeq \left\lceil5\log\Par{\frac{\log T}{\delta}}\right\rceil.\]
In the following parts of the proof, we will bound for this $p$ the quantity $\E \Phi_{T_i}^p$, to show that with high probability it remains small at the end of the loop, regardless of the location of the $1$ indices. 

\textit{Potential growth: iterates with $z_{i, j} \neq 1$.}  We first bound the potential growth in any iteration $j \in [T_i]$ where $z_{i, j} \neq 1$. Fix $y_0,y_0'$ and $\{\xi_{i,t}\}_{t \in [j - 1]}$, so that $\Phi_{j - 1}$ is deterministic. We have (taking expectations over only $\xi_{i,j}$), 
\begin{equation}\label{eq:phipgood}
\E_{\xi_{i,j}} \Phi_j^p \le \E \Par{\Phi_{j - 1} + A_j + B_j}^p,
\end{equation}
where
\begin{align*}
A_j &\defeq -2\eta_i Z_j \inprod{\partial f^{z_{i, j}}(\bx + \xi_{i,j})}{y_{j - 1} - y'_{j - 1}}, \\
B_j &\defeq \eta_i^2 Z_j^2 \norm{\partial f^{z_{i, j}}(\bx + \xi_{i,j})}^2, ~~\mbox{and}\\
Z_j &\defeq \frac{\gamma_\rho(y_{j - 1} - \bx - \xi_{i,j}) - \gamma_\rho(y'_{j - 1} - \bx - \xi_{i,j})}{\gamma_\rho(\xi_{i,j})}.
\end{align*}
The inequality in \eqref{eq:phipgood} follows from expanding the definition of the update to $\Phi_j$ before projection, and then using the fact that Euclidean projections onto a convex set only decrease distances. By the second part of Lemma~\ref{lem:p_moment_Gaussian}, for all $q \in [2, p]$, if $\sqrt{\Phi_{j-1}}\le \frac \rho p$ (which is always satisfied as $\sqrt{\Phi_{j-1}}\le r$),
\[\E_{\xi_{i,j}} Z_j^q \le \Par{\frac{24q\sqrt{\Phi_{j - 1}}}{\rho}}^q.\]
By Lipschitzness of $f^{z_{i,j}}$ and Cauchy-Schwarz (on $A_j$), we thus have
\begin{equation}\label{eq:abbounds}
\begin{aligned}
\E_{\xi_{i,j}} |A_j|^q &\le \Par{\frac{48 \eta_i L q \Phi_{j - 1}}{\rho}}^q \text{ for all } q \in [2, p], \\
\E_{\xi_{i,j}} B_j^q &\le \Par{\frac{48\eta_i L q}{\rho}}^{2q} \Phi_{j - 1}^q \text{ for all } q \in [1, p].
\end{aligned}
\end{equation}
Next, we perform a Taylor expansion of \eqref{eq:phipgood}, which yields
\begin{equation}\label{eq:taylorp}
\begin{aligned}
\E_{\xi_{i,j}} \Phi_j^p &\le \Phi_{j - 1}^p + p\Phi_{j - 1}^{p - 1} \E_{\xi_{i,j}}\Brack{A_j + B_j} \\
&+ p(p-1)\int_0^1 (1 - t) \E_{\xi_{i,j}}\Brack{\Par{\Phi_{j - 1} + t(A_j + B_j)}^{p - 2} \Par{A_j + B_j}^2}\d t.
\end{aligned}
\end{equation}
By monotonicity of convex gradients and the first part of Lemma~\ref{lem:stochastic_varbound}, we have
\begin{equation}\label{eq:A1}
\E_{\xi_{i,j}}\Brack{A_j} = -2\eta_i\inprod{\partial \hf_\rho^{z_{i, j}}(y_{j - 1}) - \partial \hf_\rho^{z_{i, j}}(y'_{j - 1})}{y_{j - 1} - y'_{j - 1}} \le 0.
\end{equation}
By applying \eqref{eq:abbounds}, we have
\begin{equation}\label{eq:B1}
p\Phi_{j - 1}^{p - 1}\E_{\xi_{i,j}} B_j \le p\Par{\frac{48\eta_i L}{\rho}}^2 \Phi_{j - 1}^p.
\end{equation}
Next we bound the second-order terms. For any $t \in [0, 1]$ we have denoting $C_j \defeq A_j + B_j$,
\begin{equation}\label{eq:ABge2}
\begin{aligned}
\E_{\xi_{i,j}}\Brack{\Par{\Phi_{j - 1} + tC_j}^{p - 2}C_j^2} &= \sum_{q = 0}^{p - 2} \binom{p - 2}{q} \Phi_{j - 1}^{p - 2 - q} \E_{\xi_{i,j}}\Brack{t^{2 + q}C_j^{2 + q}} \\
&\le 4\sum_{q = 0}^{p - 2} 2^q \binom{p - 2}{q} \Phi_{j - 1}^{p - 2 - q}\E_{\xi_{i,j}}\Brack{|A_j|^{2 + q}} \\
&+ 4\sum_{q = 0}^{p - 2} 2^q \binom{p - 2}{q} \Phi_{j - 1}^{p - 2 - q}\E_{\xi_{i,j}}\Brack{B_j^{2 + q}} \\
&\le 4\Phi_{j - 1}^p \Par{\frac{48 \eta_i L p}{\rho}}^2 \sum_{q = 0}^{p - 2} 2^q \binom{p - 2}{q} \Par{\frac{48 \eta_i L q}{\rho}}^q \\
&+ 4\Phi_{j - 1}^p \Par{\frac{48 \eta_i L p}{\rho}}^2 \sum_{q = 0}^{p - 2} 2^q \binom{p - 2}{q} \Par{\frac{48 \eta_i L (2 + q)}{\rho}}^{2q + 2} \\
&\le 8\Phi_{j - 1}^p \Par{\frac{48 \eta_i L p}{\rho}}^2 \Par{1 + \frac{96\eta_i Lp}{\rho}}^{p - 2} \\
&\le 16\Phi_{j - 1}^p \Par{\frac{48 \eta_i L p}{\rho}}^2.
\end{aligned}
\end{equation}
The first inequality used $(a + b)^p \le 2^p (a^p + b^p)$ for any nonnegative $a, b$ and $0\le t\le 1$, the second inequality used \eqref{eq:abbounds}, and the third and fourth inequalities used
\[ \frac{48\eta_i L(2 + q)}{\rho} \le \frac 1 {2p}\]
for our choices of $\eta_i L \le \frac r 4$ and $\rho$. Finally, plugging \eqref{eq:A1}, \eqref{eq:B1}, and \eqref{eq:ABge2} into \eqref{eq:taylorp},
\[
\E_{\xi_{i,j}} \Phi_j^p \le \Phi_{j - 1}^p \Par{1 + 16p^2 \Par{\frac{48 \eta_i L p}{\rho}}^2} \le \Phi_{j - 1}^p\Par{1 + 16p \Par{\frac{48 \eta_i L p}{\rho}}^2}^p.\]
Finally, using $(\eta_i L)^2 \le \frac {r^2} {16T} \le \frac {r^2} {16T_i}$ and our assumed bound on $\frac r \rho$, which implies $\frac{16p}{\rho^2} (48\eta_i Lp)^2 \le \frac 1 {T_i}$, taking expectations over $\{\xi_t\}_{t \in [j - 1]}$ yields
\begin{equation}\label{eq:good_total}
\begin{aligned}
\E \Phi_j^p \le \E \Phi_{j - 1}^p \Par{1 + \frac 1 {T_i}}^p
\text{ when } z_{i, j} \neq 1.
\end{aligned}\end{equation}

\textit{Potential growth: iterates with $z_{i, j} = 1$.}
Next, we handle the case where $z_{i, j} = 1$. We have that conditional on fixed values of $\{\xi_{i,t}\}_{t \in [j - 1]}$, $y_{0}$ and $y_{0}'$,
\begin{equation}\label{eq:phibad}
\begin{aligned}
\E_{\xi_{i,j}} \Phi_j^p &\le \E_{\xi_{i,j}}\Par{\Phi_{j - 1} + D_j + E_j}^p \\
&\le \E_{\xi_{i,j}} \Par{\Par{1 + \frac 1 {b_i}} \Phi_{j - 1} + 2b_i E_j}^p,
\end{aligned}
\end{equation}
where overloading $f \gets f(\cdot; s_1)$, $h \gets f(\cdot; s'_1)$,
\begin{align*}
D_j &\defeq -2\eta_i \inprod{\tnbx \hf_\rho(y_{j - 1}) - \tnbx \widehat{h}_\rho(y'_{j - 1})}{y_{j - 1} - y'_{j - 1}}, \\
E_j &\defeq \eta_i^2 \norm{\tnbx \hf_\rho(y_{j - 1}) - \tnbx \widehat{h}_\rho(y'_{j - 1})}^2,
\end{align*}
and we use $D_j \le \frac 1 {b_i} \Phi_{j - 1} + b_i E_j$ by Cauchy-Schwarz and Young's inequality. Next, convexity of $\norm{\cdot}^{2q}$ implies that
\[E_j^q \le \eta_i^{2q} 2^{2q - 1} \Par{\norm{\tnbx \hf_\rho(y_{j - 1})}^{2q} + \norm{\tnbx \widehat{h}_\rho(y'_{j - 1})}^{2q}}. \]
Next, we note that since $f$ is Lipschitz, the first part of Lemma~\ref{lem:p_moment_Gaussian} implies for all $q \le p$,
\[\E \norm{\tnbx \hf_\rho(y_{j - 1})}^{2q} \le L^{2q} \E\Brack{\Par{\frac{\gamma_\rho(y_{j - 1} - \bx - \xi)}{\gamma_\rho(\xi)}}^{2q}} \le 2(L)^{2q}, \]
and a similar calculation holds for $h$. Here we used our assumed bound on $\frac r \rho$ to check the requirement in Lemma~\ref{lem:p_moment_Gaussian} is satisfied.
By linearity of expectation, we thus have
\begin{equation}\label{eq:ejbound}\E_{\xi_{i,j}} E_j^{q} \le \Par{9\eta_i L}^{2q}. \end{equation}
Finally, expanding \eqref{eq:phibad} and plugging in the moment bound \eqref{eq:ejbound},
\begin{align*}
\E_{\xi_{i,j}} \Phi_j^p &\le \sum_{q = 0}^{p} \binom{p}{q}\Par{1 + \frac 1 {b_i}}^q \Phi_{j - 1}^q (2b_i)^{p - q} \E_{\xi_{i,j}}\Brack{E_j^{p - q}} \\
&\le \sum_{q = 0}^{p} \binom{p}{q}\Par{1 + \frac 1 {b_i}}^q \Phi_{j - 1}^q (2b_i)^{p - q} (9 \eta_i L)^{2(p - q)} \\
&= \Par{\Par{1 + \frac 1 {b_i}} \Phi_{j - 1} + 2b_i (9\eta_i L)^2}^p.
\end{align*}
Taking expectations over $\{\xi_{i, t}\}_{t \in [j - 1]}$, and using Fact~\ref{fact:shiftbase} with $Z \gets (1 + \frac 1 {b_i})\Phi_{j - 1}$ and $C \gets 2b_i(9\eta_i L)^2$,
\begin{equation}\label{eq:bad_total}
\E \Phi_j^p \le \Par{\Par{1 + \frac 1 {b_i}} \E\Brack{\Phi_{j - 1}^p}^{\frac 1 p} + 2b_i (9\eta_i L)^2}^p,\text{ when } z_{i,j} = 1.  \end{equation}

\textit{One loop privacy.} We begin by obtaining a high-probability bound on $\Phi_{T_i}$. Define
\[W_j \defeq \E[\Phi_j^p]^{\frac 1 p}.\]
By using \eqref{eq:good_total} and \eqref{eq:bad_total}, we observe
\begin{align*}
W_j \le \begin{cases}
\Par{1 + \frac 1 {T_i}} W_{j - 1} & z_{i, j} \neq 1 \\
\Par{1 + \frac 1 {b_i}} W_{j - 1} + 2b_i(9\eta_i L)^2 & z_{i, j} = 1
\end{cases}.
\end{align*}
Hence, regardless of the $b_i$ locations of the $1$ indices in $\ind_i$, we have
\[W_{T_i} \le \Par{1 + \frac 1 {T_i}}^{T_i}\Par{1 + \frac 1 {b_i}}^{b_i}\Par{2b_i^2(9\eta_i L)^2} \le 1200 b_i^2 (\eta_i L)^2. \]
Thus, by Markov's inequality, with probability at least $1 - \frac \delta {\log T}$ over the randomness of $\Xi_i = \{\xi_{i,j}\}_{j \in [T_i]}$, we have using our choice of $p$,
\begin{equation}\label{eq:bound_sensitivity}\norm{y_{T_i} - y'_{T_i}}^2 \le 1200 b_i^2(\eta_i L)^2 \cdot \Par{\frac{\log T}{\delta}}^{\frac 1 p} \le 1500 b_i^2(\eta_i L)^2.\end{equation}
In the last inequality, we used our choice of $p$. Call $\event_i$ the event that the sampled $\Xi_i$ admits a deterministic map which yields the bound in \eqref{eq:bound_sensitivity}. By the second part of Proposition~\ref{prop:rdp_facts}, the conditional distribution of the output of the $i^{\text{th}}$ outer loop under $\event_i$ satisfies $(\alpha, 1500\beta^2 b_i^2)$-RDP, where we use the value of $\sigma_i$ in Line~\ref{line:params} of Algorithm~\ref{alg:subsample_convex}. We conclude via Fact~\ref{fact:cond_TV} with $\event \gets \event_i$ that the $i^{\text{th}}$ outer loop of Algorithm~\ref{alg:subsample_convex} satisfies
\[\Par{\alpha, 1500\alpha\beta^2 b_i^2, \frac{\delta}{\log T}}\text{-RDP.}\]
\textit{All loops privacy.} By applying composition of RDP (the third part of Proposition~\ref{prop:rdp_facts}), for a given realization of $\ind = \cup_{i \in [k]} \ind_i$ with $b$ occurrences of $1$, applying composition over the $\log T$ outer iterations (Lemma~\ref{lem:composition_rdp}), Algorithm~\ref{alg:subsample_convex} satisfies
\[\Par{\alpha, 1500\alpha\beta^2 b^2, \delta}\text{-RDP.}\]
Here, we used $\sum_{i \in [k]} b_i^2 \le b^2$. This is the desired conclusion.
\end{proof}

We next apply amplification by subsampling to boost the guarantee of Lemma~\ref{lem:group_privacy_convex}. To do so, we use the following key Proposition~\ref{prop:subsample_fkt}, which was proven in \cite{BunDRS18}. The use case in \cite{BunDRS18} involved subsampling with replacement and was used in a framework they introduced termed truncated CDP, but we will not need the framework except through the following powerful fact.

\begin{proposition}[Theorem 12, \cite{BunDRS18}]\label{prop:subsample_fkt}
Let $\tau \le \frac 1 3$, $s \in (0, \frac 1 {40})$. Let $P$, $Q$, $R$ be three distributions over the same probability space, such that for each pair $P_1, P_2 \in \{P, Q, R\}$, we have $D_\alpha(P_1 \| P_2) \le \alpha\tau$ for all $\alpha > 1$. Then for all $\alpha \in (1, \frac 3 \tau)$,
\[D_\alpha(sP + (1 - s)R \| sQ + (1 - s)R) \le 13 s^2 \alpha\tau.\]
\end{proposition}

We also require a straightforward technical fact about binomial distributions. 

\begin{lemma}\label{lem:zero_one_couple}
Let $m, n \in \N$ satisfy $\frac m n \le \frac 1 {60}$. Consider the following partition of the elements $\ind \in [n]^m$ with at most $b$ copies of $1$:
\begin{align*}S_0 &\defeq \{\ind \in [n]^m \mid \ind_i \neq 1 \text{ for all } i \in [m]\},\\
S_1 &\defeq \{\ind \in [n]^m \mid \ind_i = 1 \text{ for between 1 and }b\text{ many } i \in [m]\}. \end{align*}
Let $\pi_0$ and $\pi_1$ be the uniform distributions on $S_0$ and $S_1$ respectively. Then there exists a coupling $\Gamma(\pi_0, \pi_1)$ such that for all $(\ind, \ind')$ in the support of $\Gamma$,
\[\Abs{\Brace{i \mid \ind_i \neq \ind'_i}} \le b. \]
\end{lemma}
\begin{proof}
Define a probability distribution $p$ on elements of $[b]$ such that
\[p_a \defeq \frac{\binom m a (n - 1)^{m - a}}{\sum_{a \in [b]} \binom m a (n - 1)^{m - a}} \text{ for all } a \in [b].\]
Clearly, $\sum_{a \in [b]} p_a = 1$. Our coupling $\Gamma \defeq \Gamma(\pi_0, \pi_1)$ is defined as follows.
\begin{enumerate}
    \item Draw $\ind \sim \pi_0$ and $a \sim p$ independently.
    \item Let $\ind'$ be $\ind$ with a uniformly random subset of $a$ indices replaced with $1$. Return $(\ind, \ind')$.
\end{enumerate}
This coupling satisfies the requirement, so it suffices to verify it has the correct marginals. This is immediate for $S_0$ by definition. For $\ind' \in S_1$, suppose $\ind'$ has $a$ occurrences of the index $1$. The total probability $\ind'$ is drawn from $\Gamma$ is then indeed
\[\frac{(n - 1)^a}{(n - 1)^m} \cdot \frac{p_a}{\binom m a} = \frac 1 {\sum_{a \in [b]} \binom{m}{a} (n - 1)^{m - a}} = \frac 1 {|S_1|}.\]
The first equality follows as the probability we draw $\ind \sim \pi_0$ which agrees with $\ind'$ on all the non-$1$ locations is $(n - 1)^{a - m}$, and the probability $\ind'$ is drawn given that we selected $\ind$ is $p_a \cdot \binom{m}{a}^{-1}$. 
\end{proof}

Finally, we are ready to state our main privacy guarantee for Algorithm~\ref{alg:subsample_convex}.

\begin{lemma}\label{lem:erm_convex_privacy}
There is a universal constant $\Cpriv \in [1, \infty)$, such that if $\frac T n \le \frac 1 {\Cpriv}$, $\beta^2 \log^2(\frac 1 \delta) \le \frac 1 {\Cpriv}$, $\delta \in (0, \frac 1 6)$, and $\frac \rho r \ge \Cpriv\log^2(\frac{\log T}{\delta})$, Algorithm~\ref{alg:subsample_convex} satisfies $(\alpha, \alpha\tau, \delta)$-RDP for
\[\tau \defeq \Cpriv\Par{\beta\log\Par{\frac 1 \delta} \cdot \frac T n}^2,\; \alpha \in \Par{1, \frac{1}{\Cpriv\beta^2\log^2(\frac 1 \delta)}}.\]
\end{lemma}
\begin{proof}
Let $\data$, $\data'$ be neighboring, and without loss of generality, suppose they differ in the first entry. Let $\Cpriv \ge 60$, and let $\ind$ be defined as in \eqref{eq:inddef}. Let $\event$ be the event that $\ind$ contains at most $b$ copies of the index $1$, where
\[b \defeq 2 \log\Par{\frac 2 \delta}.\]
By a Chernoff bound, $\event$ occurs with probability at least $1 - \frac \delta 2$ over the randomness of $\ind$. We define $P$ to be the distribution of the output of Algorithm~\ref{alg:subsample_convex} when run on $\data$, conditioned on $\event$ and $\ind$ containing at least one copy of the index $1$ (call this total conditioning event $\event_1$, i.e., there are between $1$ and $b$ copies of the index $1$). Similarly, we define $Q$ to be the distribution when run on $\data'$ conditioned on $\event_1$, and $R$ to be the distribution conditioned on $\event \cap \event_1^c$ (when run on either $\data$ or $\data'$). We claim that for all $P_1, P_2 \in \{P, Q, R\}$, we have
\begin{equation}\label{eq:pairwise_renyi}D_{\alpha, \frac \delta 2}(P_1 \| P_2) \le 1500\alpha\beta^2b^2,\text{ for all } \alpha > 1. \end{equation}
To see \eqref{eq:pairwise_renyi} for $P_1 = P$ and $P_2 = Q$ (or vice versa), we can view $P$, $Q$ as mixtures of outcomes conditioned on the realization $\ind$. Then, applying quasiconvexity of R\`enyi divergence (over this mixture), and applying Lemma~\ref{lem:group_privacy_convex} (with $\delta \gets \frac \delta 2$), we have the desired claim. To see \eqref{eq:pairwise_renyi} for the remaining cases, we first couple the conditional distributions under $\event_1$ and $\event \cap \event_1^c$ by their index sets, according to the coupling in Lemma~\ref{lem:zero_one_couple}. Then applying quasiconvexity of R\'enyi divergence (over this coupling) again yields the claim, where we set $m \gets \hT - 1 \le T$. Finally, let 
\begin{align*}
s \defeq \Pr[\event_1 \mid \event] = 1 - \frac{ \Par{1 - \frac 1 n}^{\hT - 1}}{\Pr[\event]} \le 1 - \frac{1 - \frac{1.1T}{n}}{1 - \frac \delta 2} \le \frac {1.2T} n.
\end{align*}
Note that conditional on $\event$ and the failure event in Lemma~\ref{lem:group_privacy_convex} not occurring, the distributions of Algorithm~\ref{alg:subsample_convex} using $\data$ and $\data'$ respectively are $sP + (1 - s)R$ and $sQ + (1 - s)R$. Hence, union bounding with $\event^c$ (see Fact~\ref{fact:cond_TV}), the claim follows from Proposition~\ref{prop:subsample_fkt} with $\tau \gets 6000\beta^2\log^2(\frac 2 \delta)$.
\end{proof}

\paragraph{Regularized extension.} We give a slight extension to Algorithm~\ref{alg:subsample_convex} which handles regularization, and enjoys similar utility and privacy guarantees as stated in Proposition~\ref{prop:subsample_convex_bounds}. Let
\begin{equation}\label{eq:lam_opt_def}
\xsbxl \defeq \argmin_{x \in \ball_{\bx}(r)}\Brace{\widehat{\Ferm_\rho}(x) + \frac \lam 2 \norm{x - \bx}^2}.
\end{equation}
Our extension Algorithm~\ref{alg:subsample_convex_reg} is identical to Algorithm~\ref{alg:subsample_convex}, except it requires a regularization parameter $\lam$, allows for an arbitrary starting point with an expected distance bound (adjusting the step size accordingly), and takes composite projected steps incorporating the regularization.

\begin{algorithm2e}\label{alg:subsample_convex_reg}
\caption{Subsampled ReSQued ERM solver, regularized case, convex rate}
\textbf{Input:} $\bx \in \R^d$, ball radius, convolution radius, privacy parameter, and regularization parameter $r, \rho, \beta, \lam > 0$, dataset $\data \in \calS^n$, iteration count $T \in \N$, distance bound $r' \in [0, 2r]$, initial point $x_0 \in \ball_{\bx}(r)$ satisfying $\E \|x_0 - \xsbxl\|^2 \le (r')^2$ \\
$\hT \gets 2^{\lfloor \log_2 T \rfloor}$, $k \gets \log_2 \hT$, $\eta \gets \frac {r'} L \min(\frac{1}{\sqrt{T}}, \frac \beta {\sqrt d})$ \\
\For{$i \in [k]$}
{
    $T_i \gets 2^{-i} \hT$, $\eta_i \gets 4^{-i}\eta$, $\sigma_i \gets \frac{L\eta_i}{\beta}$\\
    $y_0 \gets x_{i-1}$\\
    \For{$j \in [T_i]$}
    {  $z_{i,j}\sim_{\text{unif.}} [n]$ \\
        $y_j \gets \argmin_{y \in \ball_{\bx}(r)}\{\langle\eta_i\tnbx \hf_\rho^{z_{i, j}}(y_{j-1}), y\rangle + \half \norm{y - y_{j - 1}}^2 + \frac {\eta_i \lam} 2 \norm{y - \bx}^2\}$ \label{line:psgd_reg}
    }
    $\by_i \gets \frac 1 {T_i} \sum_{j \in [T_i]} y_j$ \\
    $x_{i} \gets \by_i + \zeta_i$, for $\zeta_i\sim\Nor(0,\sigma^2_i \id_d)$
}

\textbf{return} $x_{k}$

\end{algorithm2e}

\begin{corollary}\label{cor:subsample_reg_bounds}
Let $\xsbxl$ be defined as in \eqref{eq:lam_opt_def}. Algorithm~\ref{alg:subsample_convex_reg} uses at most $T$ gradients and produces $x \in \ball_{\bar{x}}(r)$ such that, for a universal constant $\Ccvx$,
\[\E\Brack{\widehat{\Ferm_\rho}(x) + \frac \lam 2 \norm{x - \bx}^2} - \Par{\widehat{\Ferm_\rho}(\xsbxl) + \frac \lam 2 \norm{\xsbxl - \bx}^2}\le \Ccvx Lr'\Par{\frac{\sqrt d}{\beta T} + \frac {1} {\sqrt T}}. \]
Moreover, there is a universal constant $\Cpriv \ge 1$, such that if $\frac T n \le \frac 1 {\Cpriv}$, $\beta^2 \log^2(\frac 1 \delta) \le \frac 1 {\Cpriv}$, $\delta \in (0, \frac 1 6)$, and $\frac \rho r \ge \Cpriv\log^2(\frac{\log T}{\delta})$, Algorithm~\ref{alg:subsample_convex_reg} satisfies $(\alpha, \alpha\tau, \delta)$-RDP for
\[\tau \defeq \Cpriv \Par{\beta\log\Par{\frac 1 \delta} \cdot \frac T n}^2,\; \alpha \in \Par{1, \frac{1}{\Cpriv\beta^2\log^2(\frac 1 \delta)}}.\]
\end{corollary}
\begin{proof}
The proof is almost identical to Proposition~\ref{prop:subsample_convex_bounds}, so we only discuss the differences. Throughout this proof, for notational convenience, we define
\[F^\lam(x) \defeq \widehat{\Ferm_\rho}(x) + \frac \lam 2 \norm{x - \bx}^2. \]

\textit{Utility.} Standard results on composite stochastic mirror descent (e.g.\ Lemma 12 of \cite{CarmonJST19}) show the utility bound in \eqref{eq:telescope_error} still holds with $F^\lam$ in place of $F$. In particular each term $\E[F^\lam(\by_i) - F^\lam(\by_{i - 1})]$ as well as $\E[F^\lam(x_k) - F^\lam(\by_k)]$ enjoys the same bound as its counterpart in \eqref{eq:telescope_error}. The only other difference is that, defining $\zeta_0 \defeq x_0 - \xsbxl$ in the proof of Lemma~\ref{lem:erm_convex_utility}, we have $\E \zeta_0^2 \le (r')^2$ in place of the bound $r^2$, and we appropriately changed $\eta$ to scale as $r'$ instead. 

\textit{Privacy.} The subsampling-based reduction from Lemma~\ref{lem:erm_convex_privacy} to Lemma~\ref{lem:group_privacy_convex} is identical, so we only discuss how to obtain an analog of Lemma~\ref{lem:group_privacy_convex} for Algorithm~\ref{alg:subsample_convex_reg}. In each iteration $j \in [T_i]$, by completing the square, we can rewrite Line~\ref{line:psgd_reg} as
\[y_j \gets \argmin_{y \in \ball_{\bx}(r)}\Brace{\half\norm{y - \Par{\frac 1 {1 + \eta_i \lam} y_{j - 1} + \frac {\eta_i \lam} {1 +\eta_i \lam} \bx - \frac{\eta_i}{1 + \eta_i \lam}\tnbx \hf_\rho^{z_{i, j}}(y_{j-1})} }^2 }.\]
Now consider our (conditional) bounds on $\E_{\xi_{i,j}} \Phi_j$ in \eqref{eq:phipgood} and \eqref{eq:phibad}. We claim these still hold true; before projection, the same arguments used in \eqref{eq:phipgood} and \eqref{eq:phibad} still hold (in fact improve by $(1 + \eta_i\lam)^2$), and projection only decreases distances. Finally, note that the proof of Lemma~\ref{lem:group_privacy_convex} only used the choice of step size $\eta$ through $\eta L \sqrt{T} \le r$ and used the assumed bound on $\frac r \rho$ to bound the drift growth. As we now have $\eta L \sqrt{T} \le r' \le 2r$, we adjusted the assumed bound on $\frac r \rho$ by a factor of $2$. The remainder of the proof of Lemma~\ref{lem:group_privacy_convex} is identical.
\end{proof}

Without loss of generality, $\Cpriv$ is the same constant in Proposition~\ref{prop:subsample_convex_bounds} and Corollary~\ref{cor:subsample_reg_bounds}, since we can set both to be the maximum of the two. The same logic applies to the following Proposition~\ref{prop:subsample_strongly_convex_bounds} and Lemma~\ref{lem:proxgradbound} (which will also be parameterized by a $\Cpriv$) so we will not repeat it. Finally, the following fact about initial error will also be helpful in the following Section~\ref{ssec:erm_sc}.

\begin{lemma}
\label{lem:bound_initial_ferm}
We have
\[\widehat{\Ferm_\rho}(\bx) - \Par{\widehat{\Ferm_\rho}(\xsbxl) + \frac \lam 2 \norm{\xsbxl - \bx}^2} \le \frac{2L^2}{\lam}.\]
\end{lemma}
\begin{proof}
By strong convexity and Lipschitzness of $\widehat{\Ferm_\rho}$, we have
\begin{align*}\frac \lam 2 \norm{\xsbxl - \bx}^2 &\le \widehat{\Ferm_\rho}(\bx) - \Par{\widehat{\Ferm_\rho}(\xsbxl) + \frac \lam 2 \norm{\xsbxl - \bx}^2} \\
&\le \widehat{\Ferm_\rho}(\bx) - \widehat{\Ferm_\rho}(\xsbxl) \le L\norm{\xsbxl - \bx}. \end{align*}
Rearranging gives $\|\xsbxl - \bx\| \le \frac{2L}{\lam}$, which can be plugged in above to yield the conclusion.
\end{proof}

We also state a slight extension to Lemma~\ref{lem:bound_initial_ferm} which will be used in Section~\ref{ssec:private_erm}.

\begin{lemma}\label{lem:outside_ball}
Define $x^\star_{\bx, x', \lam} \defeq \argmin_{x \in \ball_{\bx}(r)}\{\widehat{\Ferm_\rho}(x) + \frac \lam 2 \norm{x - x'}^2\}$, where $x' \in \R^d$ is not necessarily in $\ball_{\bx}(r)$. Let $x_0 \defeq \proj_{\ball_{\bx}(r)}(x')$. We have
\[\Par{\widehat{\Ferm_\rho}(x_0) + \frac \lam 2 \norm{x_0 - x'}^2} - \Par{\widehat{\Ferm_\rho}(x^\star_{\bx, x', \lam}) + \frac \lam 2 \norm{x^\star_{\bx, x', \lam} - x'}^2} \le \frac{2L^2}{\lam}.\]
\end{lemma}
\begin{proof}
The proof is identical to Lemma~\ref{lem:bound_initial_ferm}, where we use $\frac \lam 2 \norm{x_0 - x'}^2 \le \frac \lam 2 \|x^\star_{\bx, x', \lam} - x'\|^2$.
\end{proof}

\subsection{Subsampled smoothed ERM solver: the strongly convex case}\label{ssec:erm_sc}

We next give an ERM algorithm similar to Algorithm~\ref{alg:subsample_convex_reg}, but enjoys an improved optimization rate. In particular, it again attains RDP bounds improving with the subsampling parameter $\frac T n$, and we obtain error guarantees against $\xsbxl$ defined in \eqref{eq:lam_opt_def} at a rate decaying as $\frac 1 T$ or better.

\begin{algorithm2e}\label{alg:subsample_strong_convex}
\caption{Subsampled ReSQued ERM solver, strongly convex case}
\textbf{Input:} $\bx \in \R^d$, ball radius, convolution radius, privacy parameter, and regularization parameter $r, \rho,\beta, \lam > 0$, dataset $\data \in \calS^n$, iteration count $T \in \N$\\
$k\gets \lceil \log \log T\rceil,x_0\gets\bx$\\
\For{$i\in [k]$}
{
$\beta_{i - 1} \gets 2^{\frac {k - i + 1} 2}\beta$, $r_{i - 1} \gets \min(2r, \sqrt{\frac {2D_{i - 1}} \lam})$ (see \eqref{eq:DEdef}), $T_{i - 1} \gets 2^{i - 1 - k} T$ \\
$x_i \gets$ output of Algorithm~\ref{alg:subsample_convex_reg} with inputs $(\bx,r, \rho,\beta_{i - 1},\lam, \data,T_{i - 1}, r_{i - 1}, x_{i - 1})$
}
\textbf{return} $x_{k+1}$

\end{algorithm2e}
 We now give our analysis of Algorithm~\ref{alg:subsample_strong_convex} below. The proof follows a standard reduction template from the strongly convex case to the convex case (see e.g.\ Lemma 4.7 in \cite{KLL21}).
\begin{proposition}
\label{prop:subsample_strongly_convex_bounds}
Let $\xsbxl$ be defined as in \eqref{eq:lam_opt_def}. Algorithm~\ref{alg:subsample_strong_convex} uses at most $T$ gradients and produces $x$ such that, for a universal constant $\Csc$,
\[\E\Brack{\widehat{\Ferm_\rho}(x)+\frac{\lambda}{2}\|x-\bx\|^2 } - \widehat{\Ferm_\rho}(\xsbxl)-\frac{\lambda}{2}\|\xsbxl-\bx\|^2 \le 
\frac{\Csc L^2}{\lambda}\Par{\frac{d}{\beta^2 T^2}+\frac{1}{T}}. \]
Moreover, there is a universal constant $\Cpriv \ge 1$, such that if $\frac T n \le \frac 1 {\Cpriv}$, $\beta^2\log^2(\frac{\log \log T}{\delta}) \le \frac 1 {\Cpriv}$, $\delta \in (0, \frac 1 6)$, and $\frac \rho r \ge \Cpriv \log^2(\frac{\log T}{\delta})$, Algorithm~\ref{alg:subsample_strong_convex} satisfies $(\alpha, \alpha\tau,\delta)$-RDP for
\[\tau \defeq \Cpriv\Par{\beta \log\Par{\frac{\log \log T}{\delta}}\cdot \frac T n}^2,\; \alpha \in \Par{1, \frac{1}{\Cpriv\beta^2\log^2(\frac{\log\log T}{\delta})}}.\]
\end{proposition}
\begin{proof} We analyze the utility and privacy separately. \\

\textit{Utility.} 
Denote for simplicity $F^\lam(x) :=\widehat{\Ferm_\rho}(x)+\frac{\lambda}{2}\|x-\bx\|^2$, $F^\lam_\star \defeq F^\lam(\xsbxl)$, and
$\Delta_i \defeq \E[F^\lam(x_i)-F^\lam_\star]$. Moreover, define for all $0 \le i \le k$,
\begin{equation}\label{eq:DEdef}E_i \defeq \frac{2 \Ccvx^2 L^2}{\lam} \cdot \Par{\frac{ \sqrt d}{\beta_i T_i} + \frac{1}{\sqrt{T_i}}}^2,\; D_i \defeq 4E_i \sqrt[2^i]{\frac{2L^2}{\lam} \cdot \frac 1 {4E_0}},\end{equation}
where we define $T_k = T$ and $\beta_k = \beta$. By construction, for all $0 \le i \le k - 1$, $E_{i + 1} = \half E_i$, and so
\begin{equation}\label{eq:sqrt}\frac{D_{i + 1}}{4E_{i + 1}} = \sqrt{\frac{D_i}{4E_i}} \implies \sqrt{D_i E_i} = D_{i + 1}.\end{equation}
We claim inductively that for all $0 \le i \le k$, $\Delta_i \le D_i$. The base case of the induction follows because by Lemma~\ref{lem:bound_initial_ferm}, we have $\Delta_0 \le \frac{2L^2} \lam = D_0$. Next, suppose that the inductive hypothesis is true up to iteration $i$. By strong convexity,
\[\E\Brack{\norm{x_i - \xsbxl}^2} \le \frac{2\Delta_i}{\lam} \le \frac{2 D_i}{\lam},\]
where we used the inductive hypothesis. Hence, the expected radius upper bound (defined by $r_i$) is valid for the call to Algorithm~\ref{alg:subsample_convex_reg}. Thus, by Corollary~\ref{cor:subsample_reg_bounds}, 
\begin{align*}
\Delta_{i + 1} = \E\Brack{F^\lam(x_{i + 1}) - F^\lam_\star} &\le \Ccvx Lr_i\Par{\frac{\sqrt d}{\beta_i T_i} + \frac{1}{\sqrt{T_i}}}\\ &\le \Ccvx L\sqrt{\frac{2D_i}{\lam}}\Par{\frac{ \sqrt d}{\beta_i T_i} + \frac{1}{\sqrt{T_i}}} = \sqrt{D_i E_i} = D_{i + 1}.
\end{align*}
Here we used \eqref{eq:sqrt} in the last equation, which completes the induction. 
 Hence, iterating \eqref{eq:sqrt} for $k = \lceil \log_2\log_2 T\rceil$ iterations, where we use $E_0 \ge \frac{L^2}{2\lam T}$ so that $D_k \le 8E_k$, we have
\[\Delta_k \le 8E_k \le \frac{32 \Ccvx^2 L^2}{\lam} \cdot \Par{\frac{d}{\beta^2 T^2} + \frac{1}{T}}. \]

\textit{Privacy.} The privacy guarantee follows by combining the privacy guarantee in Corollary~\ref{cor:subsample_reg_bounds} and composition of approximate RDP (Lemma~\ref{lem:composition_rdp}), where we adjusted the definition of $\delta$ by a factor of $k$. In particular, we use that the privacy guarantee in each call to Corollary~\ref{cor:subsample_reg_bounds} is a geometric sequence (i.e., $\beta_i^2 T_i^2$ is doubling), and at the end it is $\half \beta^2 T^2$.
\end{proof}

\subsection{Private stochastic proximal estimator}\label{ssec:proxgrad}

In this section, following the development of \cite{AsiCJJS21}, we give an algorithm which calls Algorithm~\ref{alg:subsample_strong_convex} with several different iteration counts and returns a (random) point $\hx$ which enjoys a substantially reduced bias for $\xsbxl$ defined in \eqref{eq:lam_opt_def} compared to the expected number of gradient queries.

\begin{algorithm2e}
\caption{Bias-reduced ReSQued stochastic proximal estimator}
\label{alg:low_bias_est}
{\bf Input:} $\bx \in \R^d$, ball radius, convolution radius, privacy parameter, and regularization parameter $r,\rho,\beta,\lambda>0$, dataset $\data\in\calS^n$, iteration count $T\in\N$ with $T \le \lfloor\frac n {2\Cpriv}\rfloor$\\
$\Tmax \gets \lfloor\frac n {\Cpriv}\rfloor, \jmax\gets \lfloor\log_2 \frac{\Tmax}{T}\rfloor$\\
\For{$k \in [\jmax]$}
{
Draw $J \sim \Geom(\half)$\\
$x_0 \gets $ output of Algorithm~\ref{alg:subsample_strong_convex} with inputs $(\bx, r, \rho, \beta, \lam, \data, T)$\\
\If{$J \le \jmax$}
{
$x_J \gets$ output of Algorithm~\ref{alg:subsample_strong_convex} with inputs $(\bx,r,\rho,2^{-\frac J 2}\beta,\lambda,\data,2^{J}T)$ \\
$x_{J - 1}\gets$ output of Algorithm~\ref{alg:subsample_strong_convex} with inputs $(\bx,r,\rho,2^{-\frac{J - 1}{2}}\beta,\lambda,\data,2^{J - 1}T)$\\
$\hx_k \gets x_0+2^J(x_J-x_{J-1})$ 
}
\Else{
$\hx_k \gets x_0$ 
}
}
{\bf Return:} $\hx\gets\frac{1}{\jmax}\sum_{k\in[\jmax]}\hx_k$
\end{algorithm2e}

\begin{proposition}
\label{prop:gradient_estimator}
Let $\xsbxl$ be defined as in \eqref{eq:lam_opt_def}.
We have, for a universal constant $\Cbias$:
\begin{align*}
\|\E\hx-\xsbxl\|\leq \Cbias \Par{\frac L \lam \cdot \Par{\frac {\sqrt d} {\beta n}+ \frac 1 {\sqrt n}}},
\end{align*}
and, for a universal constant $\Cvar$,
\begin{align*}
\E\|\hx-\xsbxl\|^2 \le
\frac{\Cvar L^2}{\lam^2}\Par{\frac{d}{\beta^2 T^2} + \frac 1 T}.
\end{align*}
\end{proposition}

\begin{proof}
We begin by analyzing the output $\hx_k$ of a single loop $k \in [\jmax]$.
For $J\sim\Geom(\half)$, we have $\Pr[J=j]=2^{-j}$ if $j \in [\jmax]$, and $\Pr[J=j]=0$ otherwise. We denote $x_j$ to be the output of Algorithm 3 with privacy parameter $2^{-\frac j 2} \beta$ and gradient bound $2^j T$. First,
\begin{align*}
\E\hx_k = \E x_0+\sum_{j \in [\jmax]} \Pr[J=j]2^j (\E x_j-\E x_{j-1})=\E x_{\jmax}.
\end{align*}

Since $T \cdot 2^{\jmax} \ge \frac{\Tmax}{2}\ge \frac{n}{2\Cpriv}$, applying Jensen's inequality gives
\begin{align*}
    \|\E x_{\jmax}-\xsbxl\|\le \sqrt{\E\|x_{\jmax}-\xsbxl\|^2}\le \frac{\sqrt{2\Csc}L}{\lambda}\Par{ \frac{\sqrt{d}}{\beta n} + \frac{1}{\sqrt{n}}},
\end{align*}
where the last inequality follows from
Proposition~\ref{prop:subsample_strongly_convex_bounds} and strong convexity of the regularized function to convert the function error bound to a distance bound. This implies the first conclusion, our bias bound. Furthermore, for our variance bound, we have
\begin{align*}
    \E\|\hx_k-\E\hx_k\|^2\le \E\|\hx_k-\xsbxl\|^2\le 2\E\|\hx_k-x_0\|^2+2\E\|x_0-\xsbxl\|^2.
\end{align*}
By Proposition~\ref{prop:subsample_strongly_convex_bounds} and strong convexity, $\E\|x_0-\xsbxl\|^2\le \frac{\Csc L^2}{\lam^2}(\frac{d}{\beta^2 T^2}+\frac{1}{T})$.
Next,
\begin{align*}
    \E\|\hx_k-x_0\|^2=\sum_{j \in [\jmax]}\Pr[J=j]2^{2j}\E\|x_{j}-x_{j-1}\|^2=\sum_{j\in[\jmax]}2^j\E\|x_j-x_{j-1}\|^2.
\end{align*}
Note that
\begin{align*}
    \E\|x_j-x_{j-1}\|^2\le 2\E\|x_j-\xsbxl\|^2+2\E\|x_{j-1}-\xsbxl\|^2\le 2^{-j} \cdot \frac{6 \Csc L^2}{\lam^2}\Par{\frac{d}{\beta^2 T^2} + \frac 1 T},
\end{align*}
and hence combining the above bounds yields
\[\E\norm{\hx_k - \E \hx_k}^2 \le \frac{14 \Csc \jmax L^2}{\lam^2} \cdot \Par{\frac{d}{\beta^2 T^2} + \frac 1 T}.\]
Now, averaging $\jmax$ independent copies shows that
\begin{align*}
\E \norm{\hx - \xsbxl}^2 &= \norm{\hx - \E \hx}^2 + \norm{\E \hx - \xsbxl}^2 \\
&\le \frac{1}{\jmax} \cdot \Par{\frac{14 \Csc \jmax L^2}{\lam^2} \cdot \Par{\frac{d}{\beta^2 T^2} + \frac 1 T}} + \Cbias^2 \Par{\frac L \lam \cdot \Par{\frac {\sqrt d} {\beta n}+ \frac 1 {\sqrt n}}}^2,
\end{align*}
where we used our earlier bias bound. The conclusion follows by letting $\Cvar = \Cbias^2 + 14\Csc$.
\end{proof}

We conclude with a gradient complexity and privacy bound, depending on the sampled $J$.

\begin{lemma}\label{lem:proxgradbound}
There is a universal constant $\Cpriv \ge 1$, such that if $\beta^2\log^2(\frac{\log \log n}{\delta}) \le \frac{1}{\Cpriv}$, $\delta \in (0, \frac 1 2)$, and $\frac \rho r \ge \Cpriv \log^2(\frac{\log T}{\delta})$, the following holds.
Consider one loop indexed by $k \in [\jmax]$, and let $J$ be the result of the $\Geom(\half)$ draw. If $J \in [\jmax]$, loop $k$ of Algorithm~\ref{alg:low_bias_est} uses at most $2^{J + 1} T$ gradients. Furthermore, the loop satisfies $(\alpha, \alpha\tau, \delta)$-RDP for
\[\tau \defeq 2^J \cdot \Cpriv\Par{\beta\log\Par{\frac{\log \log n}{\delta}} \cdot \frac T n}^2,\; \alpha \in \Par{1,\frac{1}{\Cpriv\beta^2\log^2\Par{\frac{\log \log n}{\delta}}}}. \]
If $J \not\in [\jmax]$, Algorithm~\ref{alg:low_bias_est} uses at most $T$ gradients, and the loop satisfies $(\alpha, \alpha\tau, \delta)$-RDP for
\[\tau \defeq  \Cpriv\Par{\beta\log\Par{\frac{\log \log n}{\delta}} \cdot \frac T n}^2,\; \alpha \in \Par{1,\frac{1}{\Cpriv\beta^2\log^2\Par{\frac{\log \log n}{\delta}}}}.\]
\end{lemma}
\begin{proof}
This is immediate by Proposition~\ref{prop:subsample_strongly_convex_bounds}, where we applied Lemma~\ref{lem:composition_rdp} and set $\delta \gets \frac \delta 3$ (taking a union bound over the at most $3$ calls to Algorithm~\ref{alg:subsample_strong_convex}, adjusting $\Cpriv$ as necessary).
\end{proof} %
\subsection{Private ERM solver}\label{ssec:private_erm}

In this section, we give our main result on privately solving ERM in the setting of Problem~\ref{prob:sco_basic}, which will be used in a reduction framework in Section~\ref{ssec:private_sco} to solve the SCO problem as well. Our ERM algorithm is an instantiation of Proposition~\ref{prop:mainballaccel}. We first develop a line search oracle (see Definition~\ref{def:Ols}) based on the solver of Section~\ref{ssec:erm_sc} (Algorithm~\ref{alg:subsample_strong_convex}), which succeeds with high probability. To do so, we leverage the following geometric lemma for aggregating independent runs of our solver.

\begin{lemma}[Claim 1, \cite{KelnerLLST22}]\label{lem:agg}
There is an algorithm $\agg$ which takes as input $(S, \Delta) \in (\R^d)^k \times \R_{\ge 0}$, and returns $z \in \R^d$ such that $\norm{z - y} \le \Delta$, if for some unknown point $y\in\R^d$ satisfying at least $0.51k$ points $x \in S$, $\norm{x - y} \le \frac \Delta 3$. The algorithm runs in time $O(dk^2)$.
\end{lemma}

\begin{algorithm2e}
\label{alg:privat_line_search}
\caption{High probability ReSQued ERM solver, strongly convex case}
\textbf{Input:} $\bx\in \R^d$, ball radius, convolution radius, privacy parameter, regularization parameter, and failure probability $r,\rho,\beta,\lambda,\zeta>0$, dataset $\data\in\calS^n$, iteration count $T\in\N$\\
$k\gets 20\log(\frac 1 \zeta)$\\
\For{$i\in[k]$}
{
$x_i\gets$ output of Algorithm~\ref{alg:subsample_strong_convex} with inputs $(\bx,r,\rho,\beta,\lambda,\data,T)$ 
}
{\bf Return:} $x'\gets \agg(\{x_i\}_{i\in[k]}, \frac{9\sqrt{2\Csc}L}{\lambda}(\frac{d}{\beta^2T^2}+\frac{1}{T})^{\half})$
\end{algorithm2e}

\begin{proposition}
\label{prop:privat_line_search}
Let $\xsbxl$ be defined as in \eqref{eq:lam_opt_def}. Algorithm~\ref{alg:privat_line_search} uses at most $18T\log(\frac 1 \zeta)$ gradients and produces $x'$ such that with probability at least $1-\zeta$, for a universal constant $\Cls$,
\begin{align*}
    \|x'-\xsbxl\|\le \frac{\Cls L}{\lambda}\cdot \Par{\frac{\sqrt d}{\beta T} + \frac 1 {\sqrt{T}}}.
\end{align*}
Moreover, there exists a universal constant $\Cpriv\ge 1 $ such that $\frac{T}{n}\le \frac{1}{\Cpriv},\delta\in(0,\frac{1}{6})$ and $\frac{\rho}{r}\ge \Cpriv\log^2(\frac 1 \delta \log (\frac T {\zeta}))$, Algorithm~\ref{alg:privat_line_search} satisfies $(\alpha,\alpha\tau,\delta)$-RDP for 
\begin{align*}
\tau \defeq \Cpriv \log\Par{\frac 1 \zeta}\Par{\beta \log\Par{\frac 1 \delta \log \Par{\frac T {\zeta}}} \cdot \frac T n}^2,\; \alpha \in \Par{1,\frac{1}{\Cpriv \beta^2\log^2\Par{\frac 1 \delta \log \Par{\frac T {\zeta}}}}}.
\end{align*}
\end{proposition}

\begin{proof}
For each $x_i$, by Proposition~\ref{prop:subsample_strongly_convex_bounds},
\[\E\Brack{\widehat{\Ferm}_r(x_i)+\frac{\lambda}{2}\|x_i-\bx\|^2 } - \widehat{\Ferm}_r(\xsbxl)-\frac{\lambda}{2}\|\xsbxl-\bx\|^2 \le 
\frac{\Csc L^2}{\lambda}\Par{\frac{d}{\beta^2 T^2}+\frac{1}{T}}. \]

Further, by strong convexity and Jensen's inequality we have
$$\E[\| x_i-\xsbxl\|]\le \frac{\sqrt{2\Csc}L}{\lambda}\Par{\frac{d}{\beta^2T^2}+\frac{1}{T}}^{\half}.$$
Hence, by Markov's inequality, for each $i \in [k]$ we have
\begin{align*}
    \Pr\left[\|x_i-\xsbxl\|\ge \frac{3\sqrt{2\Csc}L}{\lambda}\Par{\frac{d}{\beta^2T^2}+\frac{1}{T}}^{\half}\right]\le \frac{1}{3}.
\end{align*}
Hence by a Chernoff bound, with probability $\ge 1-\zeta$, at least $0.51k$ points $x\in\{x_i\}_{i\in[k]}$ satisfy \[\|x-\xsbxl\|\le \frac{3\sqrt{2\Csc}L}{\lambda}\Par{\frac{d}{\beta^2T^2}+\frac{1}{T}}^{\half}.\]
Hence the precondition of Lemma~\ref{lem:agg} holds, giving the distance guarantee with high probability. The privacy guarantee follows from Proposition~\ref{prop:subsample_strongly_convex_bounds} and the composition of approximate RDP, where we adjusted $\Cpriv$ by a constant and the definition of $\delta$ by a factor of $k$.
\end{proof}

Now we are ready to prove our main result on private ERM.

\begin{theorem}[Private ERM]
\label{thm:DP-ERM}
In the setting of Problem~\ref{prob:sco_basic}, let $\epsdp \in (0, 1)$ and $\delta \in (0, \frac 1 6)$. There is an $(\epsdp,\delta)$-DP algorithm which takes as input $\data$ and outputs $\hx \in \R^d$ such that
\begin{align*}
    \E\Brack{\Ferm(\hx)-\min_{x\in\ball(R)}\Ferm(x)} \le  O\Par{LR\cdot\Par{\frac{1}{\sqrt{n}}+\frac{\sqrt{d\log\frac 1 \delta}\log^{1.5}(\frac n \delta)\log n}{n\epsdp}}}.
\end{align*}
Moreover, with probability at least $1 - \delta$, the algorithm queries at most  \begin{align*}O\Par{\log^6\Par{\frac n \delta}\Par{\min\Par{n,\frac{n^2\epsdp^2}{d}} + \min\Par{\frac{(nd)^{\frac 2 3}}{\epsdp}, n^{\frac 4 3}\epsdp^{\frac 1 3}}}} \text{ gradients.}\end{align*}
\end{theorem}

\begin{proof}
Throughout this proof, set for a sufficiently large constant $C$,
\begin{equation}\label{eq:fixed_params}
\begin{aligned}
\ptot \defeq CLR\Par{\frac 1 {\sqrt n} + \frac{\sqrt {d\log \frac 1 \delta} \log^{1.5}(\frac n \delta) \log n}{n \epsdp}},
\;\kappa \defeq \frac{LR}{\ptot},\\
\rho \defeq \frac{\ptot}{L\sqrt d},\; r \defeq \frac{\rho}{\sqrt C \log^2(\frac n \delta)},\;
\alpha \defeq \frac{4\log \frac 2 \delta}{\epsdp},\;
\beta \defeq \frac{\epsdp}{C\log(\frac n \delta)\sqrt{\log \frac 1 \delta}}.
\end{aligned}
\end{equation}
Note that for the given parameter settings, for sufficiently large $C$, we have
\begin{equation}\label{eq:kappabound}\kappa \le \frac 1 C \min\Par{\sqrt n, \frac {n\epsdp}{\sqrt{d\log \frac 1 \delta} \log^{1.5}(\frac n \delta) \log n}},\; \frac R r \le n\log^2\Par{\frac{\log n}{\delta}}.  \end{equation}
Our algorithm proceeds as follows. We apply Proposition~\ref{prop:mainballaccel} with $x^\star \gets \arg\min_{x \in \ball(R)} \Ferm(x)$ and $F \gets \widehat{\Ferm_\rho}$, and instantiate the necessary oracles as follows for $\Cba K\log\kappa$ iterations.
\begin{enumerate}
    \item We use Algorithm~\ref{alg:privat_line_search} with $r, \rho, \beta$ defined in \eqref{eq:fixed_params}, and 
    \begin{equation}\label{eq:ols_params}T_1 \defeq \sqrt C \Par{\frac{\kappa \sqrt d}{\sqrt{K}\beta \log^2\kappa } + \frac{\kappa^2}{K\log^3\kappa \log \frac n \delta}},\; \zeta \defeq \frac{1}{\kappa \Cba K \log \kappa}, \end{equation}
    as a $(\frac r {\Cba}, \lam)$-line search oracle $\Ols$.
    \item We use Algorithm~\ref{alg:subsample_strong_convex} with $r, \rho, \beta$ defined in \eqref{eq:fixed_params}, and 
    \begin{equation}\label{eq:obo_params}T_2 \defeq \sqrt C\Par{\frac{\kappa \sqrt d}{\sqrt K \beta \sqrt{\log \kappa}} + \frac{\kappa^2}{K\log\kappa}}, \end{equation}
    as a $(\frac{\lam r^2}{\Cba\log^3\kappa}, \lam)$-ball optimization oracle $\Obo$. 
    \item We use Algorithm~\ref{alg:low_bias_est} with $r, \rho, \beta$ defined in \eqref{eq:fixed_params}, and
    \begin{equation}\label{eq:opg_params}
    T_3 \defeq \sqrt C\Par{\frac{\kappa \sqrt d}{\sqrt K \beta} + \frac{\kappa^2}{K}}
    \end{equation}
    as a $(\frac{\ptot}{\Cba R}, \frac{\ptot \sqrt K}{\Cba R}, \lam)$-stochastic proximal oracle $\Opg$.
\end{enumerate}

We split the remainder of the proof into four parts. We first show that the oracle definitions are indeed met. We then bound the overall optimization error against $\Ferm$. Finally, we discuss the privacy guarantee and the gradient complexity bound. 

\textit{Oracle correctness.} For the line search oracle, by Proposition~\ref{prop:privat_line_search}, it suffices to show
\[\frac{\Cls L}{\lam} \cdot \Par{\frac{\sqrt d}{\beta T_1} + \frac 1 {\sqrt{T_1}}} \le \frac{r}{\Cba}.\]
This is satisfied for $T_1$ in \eqref{eq:ols_params}, since Proposition~\ref{prop:mainballaccel} guarantees $\lam \ge \frac{\ptot K^2 \log^2\kappa}{R^2 \Cba}$. Hence,
\begin{align*}
\frac{\Cls L}{\lam} \cdot \frac{\sqrt d}{\beta T_1} \cdot \frac{\Cba}{r} \le \Cls \Cba^2 \cdot \frac{\kappa\sqrt d}{\beta \log^2\kappa} \cdot \frac 1 {\sqrt K} \cdot \frac 1 {T_1} \le \half, \\
\frac{\Cls L}{\lam} \cdot \frac{1}{\sqrt {T_1}} \cdot \frac{\Cba}{r} \le \Cls \Cba^2 \cdot \frac{\kappa}{ \log^2\kappa} \cdot \frac 1 {\sqrt K} \cdot \frac 1 {\sqrt {T_1}} \le \half,
\end{align*}
for a sufficiently large $C$, where we used $K^{1.5} = \frac R r$ to simplify. By a union bound, the above holds with probability at least $1 - \frac{\ptot}{LR}$ over all calls to Algorithm~\ref{alg:privat_line_search}, since there are at most $\Cba K \log \kappa$ iterations. For the remainder of the proof, let $\event_{\textup{ls}}$ be the event that all line search oracles succeed. For the ball optimization oracle, by Proposition~\ref{prop:subsample_strongly_convex_bounds}, it suffices to show
\[\frac{\Csc L^2}{\lam} \Par{\frac{d}{\beta^2 T_2^2} + \frac 1 {T_2}} \le \frac{\lam r^2}{\Cba \log^3\kappa}. \]
This is satisfied for our choice of $T_2$ in \eqref{eq:obo_params}, again with $\lam \ge \frac{\ptot K^2 \log^2\kappa}{R^2 \Cba}$. Hence,
\begin{align*}
\frac{\Csc L^2}{\lam} \cdot \frac d {\beta^2 T_2^2} \cdot \frac{\Cba \log^3\kappa}{\lam r^2} \le \Csc\Cba^3 \cdot \frac{\kappa^2 d}{ \beta^2\log\kappa} \cdot \frac 1 K \cdot \frac 1 {T_2^2} \le \half, \\
\frac{\Csc L^2}{\lam} \cdot \frac 1 {T_2} \cdot \frac{\Cba \log^3\kappa}{\lam r^2} \le \Csc\Cba^3 \cdot \frac{\kappa^2}{\log\kappa} \cdot \frac 1 K \cdot \frac 1 {T_2} \le \half,
\end{align*}
again for large $C$. Finally, for the proximal gradient oracle, by Proposition~\ref{prop:gradient_estimator}, it suffices to show
\begin{align*}\Cbias\Par{\frac L \lam \cdot \Par{\frac{\sqrt d}{\beta n} + \frac 1 {\sqrt n}}} \le \frac{\ptot}{\Cba \lam R}, \\
\frac{\Cvar L^2}{\lam^2} \Par{\frac{d}{\beta^2 T_3^2} + \frac 1 {T_3}} \le \frac{\ptot^2 K}{\Cba^2 \lam^2 R^2}.
\end{align*}
The first inequality is clear. The second is satisfied for our choice of $T_3$ in \eqref{eq:opg_params}, which implies
\begin{align*}
\frac{\Cvar L^2}{\lam^2} \cdot \frac{d}{\beta^2 T_3^2} \cdot \frac{\Cba^2 \lam^2 R^2}{\ptot^2 K} = \Cvar \Cba^2 \cdot \frac{\kappa^2 d}{\beta^2} \cdot \frac 1 K \cdot \frac 1 {T_3^2} \le \half, \\
\frac{\Cvar L^2}{\lam^2} \cdot \frac{1}{T_3} \cdot \frac{\Cba^2 \lam^2 R^2}{\ptot^2 K} = \Cvar \Cba^2 \cdot \kappa^2 \cdot \frac 1 K \cdot \frac 1 {T_3} \le \half.
\end{align*}

\textit{Optimization error.} By Proposition~\ref{prop:mainballaccel}, the expected optimization error against $\widehat{\Ferm_\rho}$ is bounded by $\ptot$ whenever $\event_{\textup{ls}}$ occurs. Otherwise, the optimization error is never larger than $LR$ as long as we return a point in $\ball(R)$, since the function is $L$-Lipschitz. Further, we showed $\Pr[\event_{\textup{ls}}] \ge 1 - \frac{\ptot}{LR}$, so the total expected error is bounded by $2\ptot$. Finally, the additive error between $\widehat{\Ferm_\rho}$ and $\Ferm$ is bounded by $\rho L \sqrt d = \ptot$. The conclusion follows by setting the error bound to $3\ptot$.

\textit{Privacy.} We first claim that each call to $\Ols$, and $\Obo$ used by Proposition~\ref{prop:mainballaccel} satisfies
\[\Par{\alpha, \frac {\epsdp} {6\Cba K \log\kappa}, \frac \delta {18\Cba K \log\kappa}}\text{-RDP}. \]
We first analyze $\Ols$. The preconditions of Proposition~\ref{prop:privat_line_search} are met, where $\log(\frac {18\Cba K\log\kappa} \delta \log(\frac T \zeta)) \le 2\log \frac n \delta$ for our parameter settings. Moreover, our $\alpha$ is in the acceptable range. Finally, by Proposition~\ref{prop:privat_line_search} it suffices to note
\begin{align*}\frac{8\alpha\Cpriv \beta^2 T_1^2 \log^3\Par{\frac n \delta}}{n^2} \le  \frac{128C\Cpriv \beta^2\log^3(\frac n \delta)\log\frac 1 \delta}{n^2 \epsdp} \cdot \Par{\frac{\kappa^2 d}{K\beta^2\log\kappa} + \frac{\kappa^4}{K^2\log^2\kappa}}\le \frac{\epsdp}{6\Cba K \log\kappa}, \end{align*}
where the second inequality follows for sufficiently large $C$ due to  \eqref{eq:kappabound}. Next, we analyze the privacy of $\Obo$. 
The preconditions of Proposition~\ref{prop:subsample_strongly_convex_bounds} are met, where $\log(\frac{\log\log T}{\delta}) \le \log \frac n \delta$ for our parameter settings, and our $\alpha$ is again acceptable. Finally, by Proposition~\ref{prop:subsample_strongly_convex_bounds} it suffices to note
\begin{align*}
\frac{\alpha\Cpriv\beta^2 T_2^2 \log^2(\frac n \delta)}{n^2} \le \frac{16C\Cpriv\beta^2\log^2(\frac n \delta)\log \frac 1 \delta}{n^2 \epsdp} \cdot \Par{\frac{\kappa^2 d}{K\beta^2\log\kappa} + \frac{\kappa^4}{K^2\log^2\kappa}} \le \frac{\epsdp}{6\Cba K \log\kappa},
\end{align*}
again for sufficiently large $C$ from \eqref{eq:kappabound}. Hence, by applying Lemma~\ref{lem:composition_rdp}, all of the at most $\Cba K\log\kappa$ calls to $\Ols$ and $\Obo$ used by the algorithm combined satisfy
\[\Par{\alpha, \frac{\epsdp}{3}, \frac{\delta}{9}}\text{-RDP}. \]
Finally, we analyze the privacy of $\Opg$. Let
\[\jmax \defeq \left\lfloor\log_2\Par{\frac{1}{T_3}\cdot\left\lfloor\frac{n}{\Cpriv}\right\rfloor}\right\rfloor\]
be the truncation parameter in Algorithm~\ref{alg:low_bias_est}. 
The total number of draws from $\Geom(\half)$ in Algorithm~\ref{alg:low_bias_est} over the course of the algorithm is $\Cba K\log\kappa \cdot \jmax$. It is straightforward to check that the expected number of draws where $J = j$ for all $j \in [\jmax]$ is 
\[2^{-\jmax} \Cba \kappa\log\kappa \cdot \jmax = \Omega\Par{\frac{T_3}{n} \cdot K\log\kappa \cdot \jmax},\]
which is superconstant. By Chernoff and a union bound, with probability $\ge 1 - \frac \delta {n}$, there is a constant $C'$ such that for all $j \in [\jmax]$, the number of times we draw $J = j$ is bounded by
\[2^{-j} C' K\log\kappa \log\frac n \delta. \]
Similarly, the number of times we draw $J \not\in [\jmax]$ is bounded by $C' K \log\kappa\log\frac n \delta$. 
This implies by Lemma~\ref{lem:composition_rdp} that all calls to $\Opg$ used by the algorithm combined satisfy
\[\Par{\alpha, \frac {\epsdp} 6, \frac \delta {18}}\text{-RDP.}\]
Here, we summed the privacy loss in Lemma~\ref{lem:proxgradbound} over $0 \le J \le \jmax$, which gives
\begin{align*}\sum_{0 \le j \le \jmax} \Par{2^j \cdot \frac{\alpha\Cpriv \beta^2 \log^2(\frac n \delta) T_3^2}{n^2}} \cdot \Par{2^{-j} C' K\log\kappa \log\frac n \delta} \\
\le (\jmax + 1) \cdot \frac{16C C'\Cpriv K\beta^2 \log^3(\frac n \delta)\log \frac 1 \delta\log\kappa}{n^2 \epsdp} \cdot \Par{\frac{\kappa^2 d}{K\beta^2} + \frac{\kappa^4}{K^2}} \le \frac {\epsdp} 6,\end{align*}
for sufficiently large $C$, where we use $\log \kappa, \jmax \le \log n$, and $K \ge \log \frac 1 \delta$ for our parameter settings. Finally, combining these bounds shows that our whole algorithm satisfies $(\alpha, \frac {\epsdp} 2, \frac \delta 6)$-RDP, and applying Corollary~\ref{cor:approx_rdp}, gives the desired privacy guarantee.

\textit{Gradient complexity.} We have argued that with probability at least $1 - \delta$, the number of times we encounter the $J = j$ case of Lemma~\ref{lem:proxgradbound} for all $0 \le j \le \jmax$ is bounded by $2^{-j} C' K \log\kappa\log \frac n \delta$. Under this event, Proposition~\ref{prop:privat_line_search}, Proposition~\ref{prop:subsample_strongly_convex_bounds}, and Lemma~\ref{lem:proxgradbound} imply the total gradient complexity of our algorithm is at most
\begin{align*}
\Cba K\log\kappa \cdot \Par{18T_1\log \frac 1 \zeta + T_2 + \sum_{0 \le j \le \jmax} \Par{2^{-j} C'\log \frac n \delta}\Par{2^{j + 1} T_3}} \\
\le 36 \Cba C' K\log n\Par{T_1 \log n + T_2 + T_3 \log n \log \frac n \delta},
\end{align*}
where we use $\zeta \ge n^{-2}$, $\jmax \le \log n$, and $\kappa \le n$. The conclusion follows from plugging in our parameter choices from \eqref{eq:ols_params}, \eqref{eq:obo_params}, and \eqref{eq:opg_params}.
\end{proof}

Finally, we note that following the strategy of Section~\ref{ssec:erm_sc}, it is straightforward to extend Theorem~\ref{thm:DP-ERM} to the strongly convex setting. We state this result as follows.

\begin{corollary}[Private regularized ERM]
\label{cor:DP-ERM_strongly_convex}
In the setting of Problem~\ref{prob:sco_basic}, let $\epsdp \in (0, 1)$, $\delta \in (0, \frac 1 6)$, $\lam \ge 0$, and $x' \in \ball(R)$. There is an $(\epsdp,\delta)$-DP algorithm which outputs $\hx \in \ball(R)$ such that
\begin{align*}
    \E\Brack{\Ferm(\hx) + \frac \lam 2 \norm{x - x'}^2 -\min_{x\in\ball(R)}\Brace{\Ferm(x) + \frac \lam 2 \norm{x - x'}^2}} \le  O\Par{\frac{ L^2}{\lam}\cdot\Par{\frac{1}{n}+\frac{d\log\frac 1 \delta\log^{3}(\frac n \delta)\log^2 n}{n^2\epsdp^2}}}.
\end{align*}
Moreover, with probability at least $1 - \delta$, the algorithm queries at most  \begin{align*}O\Par{\log^6\Par{\frac n \delta}\Par{\min\Par{n,\frac{n^2\epsdp^2}{d}} + \min\Par{\frac{(nd)^{\frac 2 3}}{\epsdp}, n^{\frac 4 3}\epsdp^{\frac 1 3}}}} \text{ gradients.}\end{align*}
\end{corollary}
\begin{proof}
We first note that similar to Corollary~\ref{cor:subsample_reg_bounds} (an extension of Proposition~\ref{prop:subsample_convex_bounds}), it is straightforward to extend Theorem~\ref{thm:DP-ERM} to handle both regularization and an improved upper bound on the distance to the optimum, with the same error rate and privacy guarantees otherwise. The handling of the improved upper bound on the distance follows because the convergence rate of the \cite{AsiCJJS21} algorithm scales proportionally to the distance to the optimum, when it is smaller than $R$. The regularization is handled in the same way as Corollary~\ref{cor:subsample_reg_bounds}, where regularization can only improve the contraction in the privacy proof. One subtle point is that for the regularized problems, we need to obtain starting points for Algorithm~\ref{alg:subsample_strong_convex} when the constraint set is $\ball_{\bx}(r)$, but the regularization in the objective is centered around a point not in $\ball_{\bx}(r)$ (in our case, the centerpoint will be a weighted combination of $\bx$ and $x'$). However, by initializing Algorithm~\ref{alg:subsample_strong_convex} at the projection of the regularization centerpoint, the initial function error guarantee in Lemma~\ref{lem:bound_initial_ferm} still holds (see Lemma~\ref{lem:outside_ball}). 
\\

The reduction from the claimed rate in this corollary statement to the regularized extension of Theorem~\ref{thm:DP-ERM} then proceeds identically to the proof of Proposition~\ref{prop:subsample_strongly_convex_bounds}, which calls Corollary~\ref{cor:subsample_reg_bounds} repeatedly.
\end{proof}

\subsection{Private SCO solver}\label{ssec:private_sco}

Finally, we give our main result on private SCO in this section. To obtain it, we will combine Corollary~\ref{cor:DP-ERM_strongly_convex} with a generic reduction in \cite{FKT20, KLL21}, which uses a private ERM solver as a black box. The reduction is based on the iterative localization technique proposed by \cite{FKT20} (which is the same strategy used by Section~\ref{ssec:erm_sc}), and derived in greater generality by \cite{KLL21}.

\begin{proposition}[Modification of Theorem 5.1 in \cite{KLL21}]
\label{prop:reduce_to_ERM}
Suppose there is an $(\epsdp,\delta)$-DP algorithm $\alg_{\ERM}$ with expected excess loss 
\[O\Par{\frac{L^2}{\lam}\cdot\Par{\frac{1}{n}+\frac{d\log\frac 1 \delta\log^{3}(\frac n \delta)\log^2 n}{n^2\epsdp^2}}},\]
using $N(n,\epsdp,\delta)$ gradient queries, for some function $N$, when applied to an $L$-Lipschitz empirical risk (with $n$ samples, constrained to $\ball(R) \subset \R^d$) plus a $\lam$-strongly convex regularizer. Then there is an $(\epsdp,\delta)$-DP algorithm $\alg_{\SCO}$ using $\sum_{i \in \lceil \log n \rceil} N(\frac n {2^i}, \frac {\epsdp} {2^i}, \frac \delta {2^i})$ gradient queries, with expected excess population loss \[O\Par{LR\cdot\Par{\frac{1}{\sqrt{n}}+\frac{\sqrt{d\log\frac 1 \delta}\log^{1.5}(\frac n \delta)\log n}{n\epsdp}}}.\]
\end{proposition}

Theorem 5.1 in \cite{KLL21} assumes a slightly smaller risk guarantee for $\alg_{\ERM}$ (removing the extraneous $\log^3(\frac n \delta)\log^2 n$ factor), but it is straightforward to see that the proof extends to handle our larger risk assumption. Combining Proposition~\ref{prop:reduce_to_ERM} and Corollary~\ref{cor:DP-ERM_strongly_convex} then gives our main result.

\begin{theorem}[Private SCO]\label{thm:DP-SCO}
In the setting of Problem~\ref{prob:sco_basic}, let $\epsdp \in (0, 1)$ and $\delta \in (0, \frac 1 6)$. There is an $(\epsdp,\delta)$-DP algorithm which takes as input $\data$ and outputs $\hx \in \R^d$ such that
\begin{align*}
    \E\Brack{\Fpop(\hx)-\min_{x\in\ball(R)}\Fpop(x)} \le  O\Par{LR\cdot\Par{\frac{1}{\sqrt{n}}+\frac{\sqrt{d\log\frac 1 \delta}\log^{1.5}(\frac n \delta)\log n}{n\epsdp}}}.
\end{align*}
Moreover, with probability at least $1 - \delta$, the algorithm queries at most  \begin{align*}O\Par{\log^6\Par{\frac n \delta}\Par{\min\Par{n,\frac{n^2\epsdp^2}{d}} + \min\Par{\frac{(nd)^{\frac 2 3}}{\epsdp}, n^{\frac 4 3}\epsdp^{\frac 1 3}}}} \text{ gradients.}\end{align*}
\end{theorem}
\section*{Acknowledgements}

We thank Vijaykrishna Gurunathan for helpful conversations on parallel convex optimization that facilitated initial insights regarding ReSQue. We also thank the anonymous reviewers for their feedback.
YC was supported in part by the Israeli Science Foundation (ISF) grant no.\ 2486/21 and the Len Blavatnik and the Blavatnik Family foundation.
AS was supported in part by a Microsoft Research Faculty Fellowship, NSF CAREER Award CCF-1844855, NSF Grant CCF-1955039, a PayPal research award, and a Sloan Research Fellowship.

\newpage
\addcontentsline{toc}{section}{References}
\bibliographystyle{alpha}
\newcommand{\etalchar}[1]{$^{#1}$}

\newpage
\begin{appendix}
\section{Helper facts}\label{app:facts}

\begin{fact}\label{fact:polyzero}
Let $p \in \N$. For any integer $r$ such that $0 \le r \le p - 1$, $\sum_{0 \le q \le p} (-1)^q\binom{p}{q} q^r = 0$.
\end{fact}

\begin{proof}
We recognize the formula as a scaling of the Stirling number of the second kind with $r$ objects and $p$ bins, i.e., the number of ways to put $r$ objects into $p$ bins such that each bin has at least one object. When $r < p$ there are clearly no such ways.
\end{proof}

\begin{fact}\label{fact:exp_bound_p}
Let $p \in \N$ be even and $p \ge 2$. Let $\norm{x}, \norm{y} \le \frac 1 {p}$. Then
\begin{align*}\sum_{0 \le q \le p} (-1)^q \binom{p}{q} \exp\Par{\frac 1 {2} \Par{\Par{(p - q)^2 - (p - q)}\norm{x}^2 + (q^2 - q)\norm{y}^2 + 2q(p - q)\inprod{x}{y} }} \\
\le (12p\norm{x - y})^p.\end{align*}
\end{fact}
\begin{proof}
Fix some $x$.
Let $f_x(y)$ be the left-hand side displayed above, and let
\[f^q_x(y) \defeq \exp\Par{\frac 1 {2} \Par{\Par{(p - q)^2 - (p - q)}\norm{x}^2 + (q^2 - q)\norm{y}^2 + 2q(p - q)\inprod{x}{y} }}.\]
We will perform a $p^{\text{th}}$ order Taylor expansion of $f_x$ around $x$, where we show that partial derivatives of order at most $p - 1$ are all zero at $x$, and we bound the largest order derivative tensor. 

\textit{Derivatives of $f^q_x$.} Fix some $0 \le q \le p$, and define
\begin{equation}\label{eq:CFvdef}C_q \defeq q^2 - q,\; F_q \defeq f^q_x(y),\; v_q \defeq (q^2 - q)y + q(p - q)x.\end{equation}
Note that for fixed $q$, $F_q$ and $v_q$ are functions of $y$, and we defined them such that $\nabla_y v_q = C_q\id_d$, $\nabla_y F_q = v_q F_q$. Next, in the following we use $\symsum$ to mean a symmetric sum over all choices of tensor modes, e.g.\ $\symsum v_q^{\otimes 2} \otimes \id_d$ means we will choose $2$ of the $4$ modes where the action is $v_q^{\otimes 2}$. To gain some intuition for the derivatives of $F_q$,
we begin by evaluating the first few via product rule:
\begin{align*}
\nabla f^q_x(y) &= F_q v_q, \\
\nabla^2 f^q_x(y) &= F_q v_q^{\otimes 2} + C_q F_q \id_d, \\
\nabla^3 f^q_x(y) &= F_q v_q^{\otimes 3} + C_q F_q \symsum v_q \otimes \id_d, \\
\nabla^4 f^q_x(y) &= F_q v_q^{\otimes 4} + C_q F_q \symsum v_q^{\otimes 2} \otimes \id_d + 3C_q^2 F_q \id_d \otimes \id_d.
\end{align*}
For any fixed $0 \le r \le p$, we claim that the $r^{\text{th}}$ derivative tensor has the form
\begin{equation}\label{eq:rderivative}\nabla^r f^q_x(y) = F_q\Par{\sum_{0 \le s \le \lfloor \frac r 2 \rfloor} \frac{N_{r, s}}{\binom r {2s}} \Par{\Par{C_q}^s  \symsum v_q^{\otimes (r - 2s)} \otimes \id_d^{\otimes s}} }, \end{equation}
where the $N_{r, s}$ are nonnegative coefficients which importantly do not depend on $q$. To see this we proceed by induction; the base cases are computed above. Every time we take the derivative of a ``monomial'' term of the form $F_q (C_q)^s v_q^{\otimes(r - 2s)} \otimes \id_d^{\otimes s}$ via product rule, we will have one term in which $F_q$ becomes $v_q F_q$ (and hence we obtain a $F_q C_q^s v_q^{\otimes (r + 1 - 2s)} \otimes \id_d^{\otimes s}$ monomial), and $r - 2s$ many terms where a $v_q$ becomes $C_q \id_d$ (and hence we obtain a $F_q C_q^{s + 1} v_q^{\otimes (r - 1 - 2s)} \otimes \id_d^{\otimes (s + 1)}$ monomial). For fixed $0 \le s \le \lfloor \frac{r + 1}{2}\rfloor$, we hence again see that $N_{r + 1, s}$ has no dependence on $q$.

Next, note $\sum_{0 \le s \le \lfloor \frac r 2\rfloor} N_{r, s}$ has a natural interpretation as the total number of ``monomial'' terms of the form $F_q (C_q)^s v_q^{\otimes (r - 2s)} \otimes \id_d^{\otimes s}$ when expanding $\nabla^r f_x^q(y)$. We claim that for all $0 \le q \le p$ and $0 \le r \le p - 1$,
\begin{equation}\label{eq:monomial_growth}\frac{\sum_{0 \le s \le \lfloor \frac {r + 1} 2 \rfloor} N_{r + 1, s}}{\sum_{0 \le s \le \lfloor \frac r 2 \rfloor} N_{r, s}} \le p.\end{equation}
To see this, consider taking an additional derivative of \eqref{eq:rderivative} with respect to $y$. Each monomial of the form $F_q (C_q)^s v_q^{\otimes (r - 2s)} \otimes \id_d^{\otimes s}$ contributes at most $r - 2s + 1 \le p$ monomials to the next derivative tensor via product rule, namely one from $F_q$ and one from each copy of $v_q$. Averaging this bound over all monomials yields the claim \eqref{eq:monomial_growth}, since each contributes at most $p$.\\

\textit{Taylor expansion at $x$.} Next, we claim that for all $0 \le r \le p - 1$,
\begin{equation}\label{eq:smallr}\nabla^r f_x(x) = 0.\end{equation}
To see this, we have that $((p - q)^2 - (p - q)) + (q^2 - q) + 2q(p - q) = p^2 - p$ is independent of $q$, and hence all of the $F_q$ are equal to some value $F$ when $y = x$. Furthermore, when $y = x$ we have that $v_q = q(p - 1)x$. Now, from the characterization \eqref{eq:rderivative} and summing over all $q$, any monomial of the form $x^{\otimes (r - 2s)} \otimes \id_d^{\otimes s}$ has a total coefficient of
\[F N_{r, s} \sum_{0 \le q \le p} (-1)^q\binom{p}{q} (C_q)^s (q(p - 1))^{r - 2s} = F N_{r, s} (p - 1)^{r - 2s} \sum_{0 \le q \le p} (-1)^q \binom{p}{q} C_q^s q^{r - 2s}. \]
Since $C_q$ is a quadratic in $q$, each summand $(C_q)^s q^{r - 2s}$ is a polynomial of degree at most $r \le p - 1$ in $q$, so applying Fact~\ref{fact:polyzero} to each monomial yields the claim \eqref{eq:smallr}. \\

\textit{Taylor expansion at $y$.} Finally, we will bound the injective tensor norm of $\nabla^p f_x(y)$, where the injective tensor norm of a degree-$p$ symmetric tensor $\mathbf{T}$ is the maximum value of $\mathbf{T}[v^{\otimes p}]$ over unit norm $v$. We proceed by bounding the injective tensor norm of each monomial and then summing. 

 First, for any $0 \le p \le q$, under our parameter settings it is straightforward to see $\norm{v_q} \le p$ and $F_q \le 2$. Also, for any $0 \le s \le \frac p 2$ we have $C_q^s \le p^{2s}$, and by repeatedly applying \eqref{eq:monomial_growth}, we have $\sum_{0 \le s \le \lfloor \frac p 2 \rfloor} N_{p, s} \le p^p$. In other words, each of the monomials of the form $F_q (C_q)^s v_q^{\otimes (r - 2s)} \otimes \id_d^{\otimes s}$ has injective tensor norm at most $2p^p$ (since each $C_q$ contributes two powers of $p$, and each $v_q$ contributes one power of $p$), and there are at most $p^p$ such monomials. Hence, by triangle inequality over the sum of all monomials,
\[\Abs{\nabla^p f^q_x(y)[(y - x)^{\otimes p}]} \le 2p^{2p} \norm{y - x}^p. \]
By summing the above over all $q$ (reweighting by $(-1)^q \binom p q$), and using that the unsigned coefficients sum to $\sum_{0 \le q \le p} \binom q p = 2^p$, we have
\[\Abs{\nabla^p f_x(y)[(y - x)^{\otimes p}]} \le 4^{p} p^{2p}\norm{x - y}^p.\]
The conclusion follows by a Taylor expansion from $x$ to $y$ of order $p$, and using $p^p \le 3^p p!$.
\end{proof}

\begin{proof}[Proof of Lemma~\ref{lem:p_moment_Gaussian}]
 For the first claim,
\begin{align*}
\int \frac{(\gamma_\rho(x - \bx - \xi))^p}{(\gamma_\rho(\xi))^{p - 1}} \d \xi &= (2\pi\rho)^{-\frac d 2} \int \exp\Par{-\frac 1 {2\rho^2} \Par{p\norm{x - \bx}^2 - 2p\inprod{x - \bx}{\xi} + \norm{\xi}^2}} \d\xi \\
&= \exp\Par{\frac {p^2 - p} {2\rho^2}\norm{x - \bx}^2} \le 2,
\end{align*}
where the second equality used the calculation in \eqref{eq:gaussian_integral}, and the inequality used the assumed bound on $\norm{x - \bx}$.
We move onto the second claim. First, we prove the statement for all even $p \in \N$. Denote $v \defeq x - \bx$ and $v' \defeq x' - \bx$ for simplicity. Explicitly expanding the numerator yields that
	\begin{align*}
		(2\pi \rho)^{\frac d 2}\int \frac{\Par{\gamma_\rho(v - \xi) - \gamma_\rho(v' - \xi)}^p}{(\gamma_\rho(\xi))^{p - 1}} \d\xi =             \sum_{0 \le q \le p} (-1)^q \binom{p}{q} S_q
	\end{align*}
	where we define 
	\begin{align*}
		S_q & \defeq (2\pi\rho)^{\frac d 2}\int\frac{ (\gamma_\rho(v-\xi))^{p-q}(\gamma_r(v'-\xi))^{q}}{(\gamma_\rho(\xi))^{p-1}}\d\xi \\
		& = \int \exp\Par{-\frac 1 {2\rho^2}\Par{(p - q)\norm{v}^2 + q\norm{v'}^2 - 2(p - q)\inprod{v}{\xi} - 2q\inprod{v'}{\xi} + \norm{\xi}^2}} \d\xi \\
		& = (2\pi \rho)^{\frac d 2} \exp\Par{\frac 1 {2\rho^2} \Par{\Par{(p - q)^2 - (p - q)}\norm{v}^2 + (q^2 - q)\norm{v'}^2 + 2q(p - q)\inprod{v}{v'} }}. 
	\end{align*}
	In the last line, we again used \eqref{eq:gaussian_integral} to compute the integral.
	When $p\ge 2$ and is even, a strengthening of the conclusion then follows from Fact~\ref{fact:exp_bound_p} (where we overload $x \gets \frac v \rho$, $y \gets \frac {v'} \rho$ in its application). In particular, this shows the desired claim where the base of the exponent is $\frac{12p}{\rho}\norm{x - x'}$ instead of $\frac{24p}{\rho}\norm{x - x'}$. We move to general $p \ge 2$. Define the random variable
	\[Z \defeq \Abs{\frac{\gamma_\rho(x - \bx - \xi) - \gamma_\rho(x' - \bx - \xi)}{\gamma_\rho(\xi)}}.\]
	Recall that we have shown for all even $p \ge 2$, 
	\[\E Z^p \le \Par{\frac{12p\norm{x - x'}}{\rho}}^p.\]
	Now, let $p \ge 2$ be sandwiched between the even integers $q$ and $q + 2$. H\"older's inequality and the above inequality (for $p \gets q$ and $p \gets q + 2$) demonstrate
	\[\E Z^p \le \Par{\E Z^q}^{\frac {q + 2 - p} 2} \Par{\E Z^{q + 2}}^{\frac {p - q} 2} \le \Par{\frac{12(q + 2)\norm{x - x'}}{\rho}}^p,\]
	where we use $q(q + 2 - p) + (q + 2)(p - q) = 2p$. The conclusion follows since $q + 2 \le 2p$.
\end{proof}

\begin{fact}\label{fact:shiftbase}
Let $Z$ be a nonnegative scalar random variable, let $C \ge 0$ be a fixed scalar, and let $p \in \N$ and $p \ge 2$. Then
\[\Par{\E \Brack{(Z + C)^p}}^{\frac 1 p} \le \E\Brack{Z^p}^{\frac 1 p} + C.\]
\end{fact}
\begin{proof}
Denote $A \defeq \E\Brack{Z^p}^{\frac 1 p}$. Taking $p^{\text{th}}$ powers of both sides, we have the conclusion if
\[\Par{A + C}^p - \E\Brack{(Z + C)^p} \ge 0 \iff \sum_{q \in [p - 1]} \binom{p}{q} C^{p - q} \Par{A^q - \E\Brack{Z^q}} \ge 0. \]
Here we use that the $q = 0$ and $q = p$ terms cancel. We conclude since Jensen's inequality yields
\[\E[Z^p] \ge \E[Z^q]^{\frac p q} \implies A^q \ge \E[Z^q],\text{ for all } q \in [p - 1].\]
\end{proof} %
\section{Discussion of Proposition~\ref{prop:mainballaccel}}\label{app:ballexplain}

In this section, we discuss how to obtain Proposition~\ref{prop:mainballaccel} from the analysis in \cite{AsiCJJS21}. We separate the discussion into four parts, corresponding to the iteration count, the line search oracle parameters, the ball optimization oracle parameters, and the proximal gradient oracle parameters. We note that Proposition 2 in \cite{AsiCJJS21} states that they obtain function error $\epsopt$ with constant probability; however, examining the proof shows it actually yields an expected error bound. Additionally, Proposition 2 in \cite{AsiCJJS21} is stated for $x^\star$ (the comparison point in the error guarantee) defined to be the minimizer of $F$, but examining the proof shows that the only property about $x^\star$ it uses is that $x^\star \in \ball(R)$.

\paragraph{Iteration count.} The bound $\Cba K \log \kappa$ on the number of iterations follows immediately from the value $K_{\max}$ stated in Proposition 2 of \cite{AsiCJJS21}, where we set $\lam_{\min} \gets \lams$ and $\eps \gets \ptot$.

\paragraph{Line search oracle parameters.} The line search oracle is called in the implementation of Line 2 of Algorithm 4 in \cite{AsiCJJS21}. Our implementation follows the development of Appendix D.2.3 in \cite{AsiCJJS21}, which is a restatement of Proposition 2 in \cite{CarmonJJS21}. The bound $\Cba \log(\frac {R\kappa} r)$ on the number of calls to the oracle is immediate from the statement of Proposition 2. For the oracle parameter $\Delta = \frac r {\Cba}$, we note that the proof of Proposition 2 of \cite{CarmonJJS21} only requires that we obtain points at distance at most $\frac{r}{17}$ from $\xsbxl$, although it is stated as requiring a function error guarantee. This is evident where the proof applies Lemma 3 of the same paper.

\paragraph{Ball optimization oracle parameters.} The ball optimization oracle is called in the implementation of Line 5 of Algorithm 4 in \cite{AsiCJJS21}. In iteration $k$ of the algorithm, the error requirement is derived through the potential bound in Lemma 5 of \cite{AsiCJJS21}. More precisely, Lemma 5 shows that (following their notation), conditioned on all randomness through iteration $k$,
\begin{align*}
\E\Brack{A_{k + 1}\Par{F(x_{k + 1}) - F(x^\star)} + \norm{v_{k + 1} - x^\star}^2} - \Par{A_k\Par{F(x_k) - F(x^\star)} + \norm{v_k - x^\star}^2} \\
\le -\frac 1 6 \lam_{k + 1}A_{k + 1}\norm{\hx_{k + 1} - y_k}^2 + A_{k + 1} \phi_{k + 1} + a_{k + 1}^2 \sigma_{k + 1}^2 + 2Ra_{k + 1} \delta_{k + 1},
\end{align*}
where the terms $a_{k + 1}^2 \sigma_{k + 1}^2 + 2Ra_{k + 1} \delta_{k + 1}$ are handled identically in \cite{AsiCJJS21} and our Proposition~\ref{prop:mainballaccel} (see the following discussion). For the remaining two terms, Proposition 4 of \cite{AsiCJJS21} guarantees that as long as the method does not terminate, one of the following occurs.
\begin{enumerate}
\item $\norm{\hx_{k + 1} - y_k}^2 = \Omega(r^2)$.
\item $\lam_{k + 1} = O(\lams)$.
\end{enumerate}
In the first case, as long as $\phi_{k + 1}$ (the error tolerance to the ball optimization oracle) is set to be $\frac{\lam_{k + 1} r^2}{\Cba}$ for a sufficiently large $\Cba$ (which it is smaller than by logarithmic factors), up to constant factors the potential proof is unaffected. The total contributions to the potential due to all $A_{k + 1} \phi_{k + 1}$ losses from the iterations of the second case across the entire algorithm is bounded by
\[O\Par{\Par{K\log\kappa} \cdot \frac{R^2}{\ptot} \cdot \frac{\lams r^2}{\log^3\kappa}} = O\Par{R^2}. \]
Here, the first term is the iteration count, the second term is due to an upper bound on $A_{k + 1}$, and the third term is bounded since $\lam_{k + 1} = O(\lams)$. The initial potential in the proof of Proposition 2 of \cite{AsiCJJS21} is $R^2$, so the final potential is unaffected by more than constant factors.
For a more formal derivation of the same improved error tolerance, we refer the reader to \cite{CarmonH22}, Lemma 8.

\paragraph{Stochastic proximal oracle parameters.} Our stochastic proximal oracle parameters are exactly the settings of $\delta_k$, $\sigma_k$ required by Proposition 2 of \cite{AsiCJJS21}, except we simplified the bound on $\sigma_k^2 = O(\frac{\eps}{a_k})$ (note we use $\ptot$ in place of $\eps$). In particular, following notation of \cite{AsiCJJS21}, we have
\begin{align*}
\frac{\eps}{a_k} = \frac{\eps \sqrt{\lam_k}}{\sqrt{A_k}} = \Omega\Par{\eps \cdot \sqrt{\lams} \cdot \frac{\sqrt \eps}{R}} = \Omega\Par{\frac{\eps^2 K}{R^2} \log\kappa}.
\end{align*}
The first equality used $\lam_k a_k^2 = A_k$ for the parameter choices of Algorithm 4 in \cite{AsiCJJS21}. The second equality used that all $\lam_k = \Omega(\lams)$ and all $A_k = O(\frac{R^2}{\eps})$ in Algorithm 4 in \cite{AsiCJJS21}, where we chose $\lams = \frac{\eps K^2}{R^2} \log^2\kappa$. The final equality plugged in this bound on $\lams$ and simplified. Hence, obtaining a variance as declared in Proposition~\ref{prop:mainballaccel} suffices to meet the requirement. %
\section{Discussion of Proposition~\ref{prop:mainballaccel2}}\label{app:ballexplain2}

In this section, we discuss how to obtain Proposition~\ref{prop:mainballaccel2} (which is based on Proposition 1 in \cite{CarmonH22}) from the analysis in \cite{CarmonH22}. The iteration count discussion is the same as in Appendix~\ref{app:ballexplain}. We separate the discussion into parts corresponding to the two requirements in Proposition~\ref{prop:mainballaccel2}. Throughout, we will show how to use the analysis in \cite{CarmonH22} to guarantee that with probability at least $1 - \Omega(\frac{1}{\kappa})$, the algorithm has expected function error $O(\epsopt)$; because the maximum error over $\ball(R)$ is $\le LR$, this corresponds to an overall error $O(\epsopt)$, and we may adjust $\Cba$ by a constant to compensate.

\paragraph{Per-iteration requirements.} The ball optimization error guarantees are as stated in Proposition 1 of \cite{CarmonH22}, except we dropped the function evaluations requirement. To see that this is obtainable, note that \cite{CarmonH22} obtains their line search oracle (see Proposition~\ref{prop:mainballaccel}) by running $O(\log(\frac{R \kappa}{r}))$ ball optimization oracles to $O(\lam r^2)$ expected error, querying the function value, and applying Markov's inequality to argue at least one will succeed with high probability. We instead execute $O(\log(\frac{R \kappa}{r}))$ independent runs and apply a Chernoff bound to argue that with probability $O(\frac{1}{K\kappa \cdot \text{polylog}(K\kappa)})$, the preconditions of $\agg$ in Lemma~\ref{lem:agg} are met with $\Delta = O(r)$, as required by the line search oracle (see Algorithm~\ref{alg:privat_line_search}). Finally, applying a union bound over all iterations implies that the overall failure probability due to these line search oracles is $O(\frac 1 \kappa)$ as required by our earlier argument. 

\paragraph{Additional requirements.} The error requirements of the queries which occur every $\approx 2^{-j}$ iterations are as stated in \cite{CarmonH22}. The only difference is that we state the complexity deterministically (Proposition 1 of \cite{CarmonH22} implicitly states an expected gradient bound). The stochastic proximal oracle is implemented as Algorithm 2, \cite{CarmonH22}; it is also adapted with slightly different parameters as Algorithm~\ref{alg:low_bias_est} of this paper. The expected complexity bound is derived by summing over all $j \in [\lceil\log_2 K + \Cba\rceil]$, the probability $j$ is sampled in each iteration of Algorithm 2 of \cite{CarmonH22}. For all $j$ a Chernoff bound shows that the number of times in the entire algorithm $j$ is sampled is $O(2^{-j} K\log(\frac{R\kappa}{r}))$ (within a constant of its expectation), with probability $1 - \Omega(\text{poly}(\frac r {R\kappa}))$. Taking a union bound over all $j$ shows the failure probability of our complexity bound is $O(\frac 1 \kappa)$ as required.

%
\section{Discussion of Proposition~\ref{lem:parallel-agd}}\label{app:GLexplain}

In this section, we discuss how to obtain~\Cref{lem:parallel-agd} using results in~\cite{GhadimiL12}. We first state the following helper fact on the smoothness of a convolved function $\hf_\rho$ (see~\Cref{def:gaussian-convolution}).

\begin{fact}[Lemma 8,~\cite{BubeckJLLS19}]\label{fact:smoothness}
If $f: \R^d \to \R$ is $L$-Lipschitz, $\hf_\rho$ (see Definition~\ref{def:gaussian-convolution}) is $\frac L \rho$-smooth.
\end{fact}

The statement of~\Cref{lem:parallel-agd} then follows from recursively applying Proposition 9 of~\cite{GhadimiL12} on the objective $\Psi = \hf_\rho + \frac \lam 2 \norm{\cdot - \bx}^2$, which is $\lambda$-strongly convex and $(\frac L \rho + \lam)$-smooth, together with the divergence choice of $V(x_0, x^*) \defeq \frac{1}{2}\|x_0-x^*\|^2$, which satisfies $\nu = 1$. Our parameter choices in~Algorithm \ref{alg:parallel-agd} are the same as in \cite{GhadimiL12}, where we use that our variance bound is $3L^2$ (Lemma~\ref{lem:stochastic_varbound}). 

In particular, denote the iterate $x_T^\ag$ after the $k^{\text{th}}$ outer loop by $x^k$. We will inductively assume that $\E\frac{1}{2}\|x^{k-1}-x_{\bar{x},\lambda}^\star\|^2\le \frac{r^2}{2^{k-1}}$ (clearly the base case holds). This then implies
\[
\E\left[\frac{\lambda}{2}\|x^k-x_{\bar{x},\lambda}^\star\|^2\right] \le \E \left[\Psi(x^k)-\Psi(x_{\bar{x},\lambda}^\star)\right]\le \frac{2(\frac L \rho + \lam)\|x^{k-1}-x_{\bar{x},\lambda}^\star\|^2}{T(T+1)}+\frac{24L^2}{\lambda N_k(T+1)}\le \frac{\lambda}{2^{k}}r^2\;
\]
where the second inequality is Proposition 9 in~\cite{GhadimiL12} (cf.\ equation (4.21) therein), and the last is by our choice of $T$ and $N_k$. Thus, when $K > \log_2(\frac{\lam r^2}{\phi})$ we have $\E \Psi(x_T^\ag) - \Psi(\xsbxl) \le \phi$ as in the last outer loop $k = K$. 
The computational depth follows immediately from computing $TK$, and the total oracle queries and computational complexity follow since $N_K$ asymptotically dominates:
\[
T\cdot\Par{\sum_{k\in[K]}N_k} = O\Par{TN_K + TK} = O\Par{\sqrt{1 + \frac{L}{\rho\lam}}\log\Par{\frac{\lam r^2}{\phi}} + \frac{L^2}{\lam \phi}}.
\] \end{appendix}

\end{document}